\documentclass[10pt]{amsart}
\usepackage{amssymb,amsfonts}
\usepackage[all,arc]{xy}
\usepackage{enumitem}
\usepackage{mathrsfs}
\usepackage{ stmaryrd }
\usepackage{url}

\usepackage
[pdfauthor={Carl Wang-Erickson},
 pdftitle={Presentations of non-commutative deformation rings via A-infinity algebras and applications to deformations of Galois representations and pseudorepresentations},
 bookmarks=false]
{hyperref}

\newcounter{qcounter}

\newcommand\define{\newcommand}

\DeclareMathOperator{\Spec}{Spec}

\define\cC{\mathcal{C}}

\newcommand{\dia}[1]{{\langle #1 \rangle}}

\newcommand{\ttmat}[4]{\left( \begin{array}{cc}
#1 & #2 \\
#3 & #4
\end{array}
\right)}

\newcommand{\Z}{\mathbb{Z}}
\newcommand{\Q}{\mathbb{Q}}
\newcommand{\C}{\mathbb{C}}

\newcommand{\F}{\mathbb{F}}
\newcommand{\bN}{\mathbb{N}}

\newcommand{\m}{\mathfrak{m}}
\newcommand{\frt}{\mathfrak{t}}

\newcommand{\varep}{\varepsilon}

\newcommand{\Hom}{\mathrm{Hom}}
\newcommand{\Gal}{\mathrm{Gal}}

\newcommand{\Aut}{\mathrm{Aut}}

\newcommand{\Ext}{\mathrm{Ext}}

\newcommand{\End}{\mathrm{End}}

\newcommand{\Frob}{\mathrm{Frob}}

\newcommand{\lb}{{[\![}}
\newcommand{\rb}{{]\!]}}

\define\cA{\mathcal{A}}
\define\cD{\mathcal{D}}

\define\cG{\mathcal{G}}

\define\ord{{\mathrm{ord}}}

\define\GL{{\mathrm{GL}}}
\define\PGL{{\mathrm{PGL}}}

\define{\Fitt}{\mathrm{Fitt}}
\define{\Ann}{\mathrm{Ann}}

\newtheorem{thm}{Theorem}[subsection] 
\newtheorem*{thm*}{Theorem}

\newtheorem{cor}[thm]{Corollary}
\newtheorem{prop}[thm]{Proposition}
\newtheorem{lem}[thm]{Lemma}

\newtheorem{fact}[thm]{Fact}

\theoremstyle{definition}
\newtheorem{defn}[thm]{Definition}

\newtheorem{eg}[thm]{Example}

\newtheorem{note}[thm]{Notation}

\newtheorem{warn}[thm]{Warning}

\theoremstyle{remark}
\newtheorem{rem}[thm]{Remark}

\newcommand{\ra}{\rightarrow}

\newcommand{\lra}{\longrightarrow}
\newcommand{\lrisom}{\buildrel\sim\over\lra}
\newcommand{\risom}{\buildrel\sim\over\ra}
\newcommand{\rinj}{\hookrightarrow}

\newcommand{\rsurj}{\twoheadrightarrow}

\DeclareMathOperator{\Spf}{Spf}

\DeclareMathOperator{\ad}{ad}

\newcommand{\bF}{\mathbb{F}}

\newcommand{\bZ}{\mathbb{Z}}

\newcommand{\bT}{\mathbb{T}}
\newcommand{\cE}{\mathcal{E}}
\newcommand{\cO}{\mathcal{O}}
\newcommand{\Db}{{\bar D}}

\newcommand{\Rep}{{\mathrm{Rep}}}

\newcommand{\bG}{{\mathbb{G}}}

\DeclareMathOperator{\Tr}{\mathrm{Tr}}
\DeclareMathOperator{\Sym}{\mathrm{Sym}}

\usepackage{color}

\newcommand{\bro}{{\bar\rho}}

\newcommand{\Def}{{\mathrm{Def}}}
\DeclareMathOperator{\im}{\mathrm{im}}

\newcommand{\fS}{\mathfrak{S}}
\newcommand{\fT}{\mathfrak{T}}

\newcommand{\B}{\mathrm{Bar}}

\newcommand{\uotimes}{\underline{\otimes}}
\newcommand{\oQ}{\overline{\Q}}
\newcommand{\Fil}{\mathrm{Fil}}
\newcommand{\Alg}{\mathrm{Alg}}
\newcommand{\Aff}{{\mathcal{A}f\!f}}
\newcommand{\gr}{\mathrm{gr}}
\newcommand{\ps}{\mathrm{ps}}
\newcommand{\MC}{\mathrm{MC}}

\makeatletter
\let\c@equation\c@thm
\makeatother
\numberwithin{equation}{subsection}

\title[$A_\infty$-algebras and Galois deformations]{Presentations of non-commutative deformation rings via A-infinity algebras and applications to deformations of Galois representations and pseudorepresentations}

\date{April 3, 2020}

\author{Carl Wang-Erickson}
\address{Department of Mathematics, University of Pittsburgh \\
	Pittsburgh, PA 15260, USA}
\email{carl.wang-erickson@pitt.edu}

\begin{document}

\begin{abstract}
We introduce an $A_\infty$-algebra structure on the Hochschild cohomology of the endomorphism bimodule of a finite-dimensional representation of an associative algebra. We prove that this structure determines a presentation for non-commutative deformations of the representation. From this, we deduce presentations of universal deformation rings of Galois representations. In turn, we apply these presentations in order to deduce universal deformation rings of Galois pseudorepresentations, supplying a a tangent and obstruction theory for pseudorepresentations. This generalizes the broadly used tangent and obstruction theory for Galois representations. We also give applications, calculating the ranks of certain Hecke algebras. 
\end{abstract}

\maketitle

\tableofcontents

\newpage

\part{Introduction}

This paper is mainly occupied with identifying homotopy-algebraic structures on Galois cohomology groups that explain number-theoretic phenomena. More specifically, we describe an $A_\infty$-algebra structure on the Galois cohomology of the adjoint representation of a Galois representation $\rho$ and relate it to the deformation theory of $\rho$. We prove that a certain classical hull of this homotopy algebra represents the classical Galois deformation problem. This gives a presentation, in terms of Galois cohomology, of the Galois deformation rings first studied by Mazur \cite{mazur1989}. 

It is fair to call this homotopy algebra a ``derived enrichment'' of classical moduli rings of Galois representations, commonly known as universal deformation rings of Galois representations, or \emph{Galois deformation rings}. Therefore, there are some relations between this work and, for example, that of Galatius--Venkatesh \cite{GV2018} on derived Galois deformation rings; for more comments about this, see \S\ref{subsec: GV2018}. 

However, we do not discuss derived deformation problems here, as our motivation is to identify structures in Galois cohomology that control objects that live squarely in the classical world. Thus, the reader interested in recent developments in derived enrichment to the Langlands correspondence, initiated by Venkatesh and collaborators, may take the perspective that the present paper shows how to explicitly compute classical Galois deformation rings in terms of derived structures. 

Indeed, the original motivation for this study, introduced in \S\ref{sec: motivation}, is the search for a tangent and obstruction theory for Galois \textit{pseudorepresentations}. This goal is reached in two steps.\label{pg: twosteps}
\begin{enumerate}[leftmargin=2em]
\item As discussed above, find a presentation for classical moduli rings of Galois representations in terms of cohomological data. The usual tangent and obstruction theory of representations is augmented by computations of $A_\infty$-products, which is needed in order to find this presentation. 
\item Deduce a tangent and obstruction theory for Galois pseudorepresentations using the author's previous work \cite{WE2018}, which establishes that moduli rings of pseudorepresentations are invariant subrings of moduli rings of representations under a natural adjoint action. This refines work of Bella\"iche \cite{bellaiche2012}. 
\end{enumerate}
The theoretical content behind step (2) comes entirely from \cite{WE2018}. Thus what is new in this paper is to carry out step (1), and then deduce presentations for the adjoint-invariant subrings of the rings presented in (1). 

After steps (1) and (2) are complete, we 
\begin{enumerate}[leftmargin=2em, resume]
\item demonstrate that the (1) and (2) may be adapted to the deformation theory of Galois (pseudo)representations with a local conditions imposed, and
\item as a sample application, prove that the ranks of certain $p$-adic modular Hecke algebras -- both residually Eisenstein and residually cuspidal -- are determined by $A_\infty$-products in Galois cohomology. 
\end{enumerate}

As the presentations of step (1) are novel in number theory, but familiar in non-commutative geometry (see \S\ref{subsec: NCDT}), explaining how these fields interface is a secondary goal of this paper. For this reason, we hope that the reader can gain from the extended introduction -- \S\S1-4 -- what is new to them, and skip over familiar fundamental facts. In particular, the reader interested in the main results can skip ahead to \S\ref{sec: main results} or, alternatively, to the introduction to $A_\infty$-algebras in \S\ref{sec: A-infinity background}. 

The topics of this paper have been studied from a diverse array of mathematical fields and perspectives. To avoid distracting the reader with references to such connections throughout the text, we concentrate this discussion of ``Complements'' in \S\ref{sec: complements}.

\section{Why Galois representations and why $A_\infty$-algebras}
\label{sec: motivation}

In this section, we introduce number-theoretic motivation to non-commutative geometers and introduce non-commutative geometry techniques to number theorists. Readers familiar with the deformation theory of Galois representations and pseudorepresentations may wish to proceed to \S\ref{subsec: red case} or \S\ref{subsec: illustrate}. 

\subsection{2-dimensional Galois representations and modular Hecke eigenforms}

Consider 2-dimensional representations $\rho$ of the absolute Galois group $G_\Q$ of $\Q$, with coefficients in $\oQ_p$. We will call these \textit{Galois representations} valued in $\oQ_p$, and presume that all functions out of $G_\Q$ are continuous without further remark. In fact, up to conjugacy, $\rho$ is valued in the integral closure $\cO_E$ of $\Z_p$ in a finite subextension $E/\Q$ of $\oQ_p/\Q_p$,
\[
\rho: G_\Q \lra \GL_2(\cO_E). 
\]
On the other hand, from the theory of Hecke actions on modular forms, we have normalized modular Hecke eigenforms $f$. As these are holomorphic functions of a complex variable $z$ in the upper half of the complex plane with period 1, which we think of as $q$-series in $\C\lb q\rb$ for $q = e^{2 \pi i z}$. In fact, these have algebraic integer coefficients $a_n(f)$, that is, 
\[
f = \sum_{n \geq 0} a_n(f) q^n  \in \overline{\Z}\lb q\rb \quad \text{ after the normalization } a_1(f) = 1. 
\]

The Fontaine--Mazur conjecture \cite{FM1995} predicts that certain $\rho$ are expected to be \textit{modular}, that is, to be isomorphic to the Galois representation $\rho_f : G_\Q \to \GL_2(\oQ_p)$ arising from a Hecke eigenform $f$. The key property of $\rho_f$ is that its trace function is characterized by the condition
\[
\Tr \rho_f(\mathrm{Frob}_\ell) = a_\ell(f)
\]
for all but finitely many prime numbers $\ell$, where $\Frob_\ell \in G_\Q$ is a choice of Frobenius element relative to a prime $\ell$. We have implicitly fixed an isomorphism $\oQ_p \simeq \C$ to make this comparison. 

Because both the $\rho$ and $f$ have the $p$-integral structure that we have just explained, they are organized into congruence classes modulo $p$. We label the congruence classes of Galois representations and Hecke eigenforms by
\[
\bar\rho : G_\Q \ra \GL_2(\F) \text{ and } \bar f \in \F\lb q\rb,
\]
respectively. Here $\F$ is a finite field of characteristic $p$. Work of Mazur \cite{mazur1989, mazur1978} introduced the moduli-theoretic study of these congruence classes, putting them in bijection with 
\begin{itemize}[leftmargin=2em]
\item (Galois representations) homomorphisms $R_\bro \ra \oQ_p$ out of a deformation ring $R_\bro$ of Galois representations, designed so that commutative $\Z_p$-algebra homomorphisms $R_\bro \ra A$ correspond to strict equivalence classes of Galois representations $\rho_A$ with coefficients in $A$ such that
\begin{enumerate}[label=(\roman*), leftmargin=2em]
\item $\rho_A : G_\Q \ra \GL_2(A)$ is congruent to $\bar\rho$, and 
\item $\rho_A$ satisfies properties expected of Galois representations arising from Hecke eigenforms.
\end{enumerate}
\item (Hecke eigenforms) homomorphisms $\bT \ra \oQ_p$ out of a Hecke algebra $\bT$, where $\bT$ is the completion of a Hecke algebra (arising from the Hecke action on a finite dimensional $\C$-vector space of modular forms) at a maximal ideal (with finite residue field of characteristic $p$) corresponding to $\bar f$.
\end{itemize}
When $\bar\rho$ and $\bar f$ may be chosen compatibly, which we now assume, it is natural to ask whether there is a local homomorphism $R_\bro \ra \bT$ arising from the $p$-adic Galois representations attached to Hecke eigenforms. Then, one is led to ask about ``$R_\bro = \bT$,'' a statement that Galois representations and Hecke eigenforms interpolate compatibly.  

\subsection{The irreducible case}
\label{sssec: irred case}
When there are no Eisenstein series congruent to $\bar f$, or, equivalently, $\bar\rho$ is absolutely irreducible, it is often possible to prove that $R_\bro \risom \bT$ -- this was first carried out by Mazur \cite{mazur1989} and Wiles \cite{wiles1995}. Since Wiles's work, one crucial aspect of this argument is control over $R_\bro$ in terms of the arithmetic invariants of Galois cohomology. In the most basic setting, these are 
\[
H^1(G, \End_\F(\rho)) \text{ and } H^2(G, \End_\F(\rho)),
\]
where $\rho$ may now have arbitrary dimension $d$. This control is called a \textit{tangent and obstruction theory}, which we now explain. 

For $n \geq 0$, an $n$th-order \textit{lift} of $\rho$ is a representation
\[
\rho_n: G \lra \GL_d(\F[\varep]/\varep^{n+1}) \quad \text{such that } (\rho_n \mod{\varep}) = \rho. 
\]
These lifts are called \emph{strictly equivalent} when they are conjugate by the adjoint action of a matrix of the form $1_{d \times d}+ \varep \cdot M_{d \times d}(\F[\varep]/\varep^{n+1})$. Strict equivalence classes of $n$th-order lifts are called \emph{$n$th-order deformations} of $\rho$. 

The content of a tangent and obstruction theory is that 
first-order deformations of $\rho$ (which comprise the \emph{tangent space}) are in bijection with 
\[
H^1(G, \End_\F(\rho)) \cong \Ext^1_{\F[G]}(\rho,\rho). 
\]
and that, for $n \geq 1$, an $n$th-order deformation $\rho_n$ induces an element of the \emph{obstruction space} 
\[
H^2(G, \End_\F(\rho)) \cong \Ext^2_{\F[G]}(\rho,\rho)
\]
that is zero if and only if $\rho_n$ can be extended to an $(n+1)$st order deformation, i.e.\ 
\[
\rho_{n+1} : G \lra \GL_d(\F[\varep]/\varep^{n+2}) \quad \text{such that } (\rho_{n+1} \mod{\varep^{n+1}}) = \rho_n.
\]
When one exists, these extensions are a torsor over $H^1(G,\End_\F(\rho))$. 

In practice, one carries out the Taylor--Wiles method of \cite{wiles1995, TW1995} and subsequent developments. Very roughly speaking, these involve auxiliary cases where the tangent and obstruction theory reduces to the simple case $H^2(\mathrm{aux}, \End_\F(\rho)) = 0$. In contrast, our goal, expressed in the language of deformations and obstructions, is to provide a systematic framework to find the obstruction in $H^2(G,\End_\F(\rho))$ that arises from an $n$th-order deformation $\rho_n$. 

\subsection{The reducible case}
\label{subsec: red case}

When $\bar\rho$ is reducible, it is necessary to modify the approach above. Let $d=2$ for simplicity. In order to hope for a correspondence with $\bT$ in general, we must replace $R_\bro$ by a ring $R^\mathrm{ps}$ that parameterizes 2-dimensional \emph{pseudorepresentations} of Galois groups. So we write $R^\mathrm{ps}$ for clarity. Indeed, isomorphisms $R^\mathrm{ps} \cong \bT$ have been proven; see e.g.\ \cite{BK2011, BK2015, deo2017, WWE1}, which also give reasons for studying $R^\mathrm{ps}$ instead of $R_\bro$. 

A \emph{2-dimensional pseudorepresentation} of $G_\Q$ valued in $A$, written $D_A : G \ra A$, amounts to a pair of functions 
\[
D_A = \{\Tr, \det : G \ra A\}
\]
obeying properties expected of such functions arising from characteristic polynomials of a 2-dimensional representation with coefficients in $A$. For a precise definition of a $d$-dimensional pseudorepresentation due to Chenevier \cite{chen2014}, see  \S\ref{subsec: review RR case}. Given a representation $\rho$, we write $\psi(\rho)$ for the induced pseudorepresentation. 

For the moment, we take $\bar D = \psi(\bar\rho) : G \ra \F$ to be the pseudorepresentation given by $\{\Tr \bar\rho, \det \bar\rho\}$. Then we let $R^\mathrm{ps}$ be the ``universal pseudodeformation ring'' for $\Db$; it has the universal property that local $\Z_p$-algebra homomorphisms $R^\mathrm{ps} \ra A$ are in bijection with pseudorepresentations $D_A : G \ra A$ such that 
\begin{enumerate}[label=(\roman*), leftmargin=2em]
\item $D_A$ is congruent to $\Db$, i.e.\ the composite of $D_A$ and $A \rsurj A/\m_A \cong \F$ is equal to $\Db$; and
\item $D_A$ satisfies properties of Galois representations arising from Hecke eigenforms; such conditions are translated from representations to pseudorepresentations (of any dimension) in the author's work with Preston Wake \cite{WWE4}. 
\end{enumerate}

However, a tangent and obstruction theory for pseudorepresentations has been lacking. For example, Thorne remarks that ``the ring $R^\mathrm{ps}$ is difficult to control using Galois cohomology, a tool which is essential in other arguments'' \cite[pg.\ 786]{thorne2015}. Indeed, to this author's knowledge, no formulation of obstruction theory had been produced. There is a partial characterization of a canonical filtration on the tangent space due to Bella\"iche \cite{bellaiche2012} (following on his work with Chenevier \cite{BC2009}), in the case where the semi-simple $\bar\rho$ inducing $\Db$ has distinct simple factors. This is known as the \emph{residually multiplicity-free} case. However, the tangent space is only characterized when there are two simple factors \cite[Thm.\ A]{bellaiche2012}. 

The terminal result of this paper, Theorem \ref{thm: main pseudo}, supplies a tangent and obstruction theory for a multiplicity-free residual pseudorepresentation. Theorem \ref{thm: main pseudo} arises from the product of two antecedent theorems and a combinatorial calculation. The two antecedent theorems comprise the two steps mentioned above (pg.\ \pageref{pg: twosteps}). 
\begin{enumerate}[leftmargin=2em]
\item Present moduli rings of Galois representations in terms choices determining an $A_\infty$-algebra structure on the adjoint cohomology. 
\item Apply the expression for $R^\ps$ as an invariant subring of a moduli ring for Galois representations proved in \cite{WE2018}. 
\end{enumerate}
Step (1), which is carried out in Part 2, is this paper's main theoretical development. Part 3 works out Step (2), which is an application of the theory of \cite{WE2018}, working toward applications to deformations of Galois representations and pseudorepresentations. 

\begin{rem}
In fact we produce a \emph{presentation} for $R^\ps$ to stand in for a conventional tangent and obstruction theory, which only gives the tangent space and the space where obstructions may be computed. Conventional tangent and obstruction theory relies on the residual $\F$-valued object being an object in a $\F$-linear category. But there is no abelian category structure on pseudorepresentations. Thus we rely on deformation theory of representations and the invariant-theoretic step (2) above. 
\end{rem}

\begin{rem}
As is well-known, subrings of invariants of a group action on a regular ring may not be regular; that is, step (2) induces obstructions, which are of a combinatorial nature (details in \S\ref{subsec: GIT intro}). Moreover, one sees in the calculation of invariants that the tangent space of the pseudodeformation functor is ``obstructed,'' in the sense that obstructions of deformations of representations cut down the possible first-order deformation of pseudorepresentations. More specifically, if the deformed pseudorepresentation arises from a semi-simple representation $\rho$ with $n$ distinct irreducible summands, then one must know certain deformations of $\rho$ up to $n$th-order in order to merely calculate the tangent space of pseudodeformations. 
\end{rem}

\subsection{Concrete illustration of the role of $A_\infty$-products and adjoint invariants}
\label{subsec: illustrate}

In this section, we aim to explain to the reader why an $A_\infty$-algebra on adjoint Galois cohomology naturally supplies a presentation of a Galois deformation ring. First, we put forth the perspective that the calculation of Massey products is the natural context to test whether an $n$th-order lift extends to an ($n+1$)-st order lift. Then, we explain that Massey products are efficiently subsumed into products occurring in an $A_\infty$-algebra structure on adjoint Galois cohomology, and outline the method of presenting the Galois deformation ring in terms of the $A_\infty$-structure. 

\subsubsection{Calculating obstructions inductively}

\label{sssec: massey calc intro}

Suppose that we want to calculate the $n$th-order lifts of an irreducible representation $\rho: G \to \GL_d(\F)$. These are homomorphisms $\rho_n : G \to \GL_d(\F[\varep_n])$ reducing modulo $\varep$ to $\rho$, where $\F[\varep_n] := \F[\varep]/(\varep^{n+1})$. Taking a 1-cocycle $e$ on $G_\Q$ valued in $\End_\F(\F^{\oplus d})$ with action via $\rho$, there is a first-order lift $\rho_1$ determined by $e$ whose expression with respect to the obvious $\F$-basis of $M_d(\F[\varep_1])$ is
\[
\rho_1 = 
\begin{pmatrix}
\rho & e \\ 
 & \rho
\end{pmatrix}. 
\]
A lift of $\rho_1$ to second order has the form
\[
\rho_2 = 
\begin{pmatrix}
\rho & e & f_2 \\
 & \rho & e \\
 & & \rho
\end{pmatrix}
\]
for some function $f_2 : G_\Q \to M_d(\F)$. One should think of $f_2$ as a 1-cochain on $G_\Q$ under the adjoint $G_\Q$ action on $\End_\F(\rho) = M_d(\F)$. We write $C^1(G_\Q,\End_\F(\rho))$ for this cochain space and we write $d$ for the coboundary map. One may calculate that the function must satisfy
\[
df_2 = e \smile e,
\]
where ``$\smile$'' denotes the case $i = j = 1$ of the cup product map, which is the composite 
\begin{gather}
\label{eq: cup product in cochains}
\begin{split}
%\smile : 
C^i(G_\Q,\End_\F(\rho)) \times C^{j}(G_\Q, &\, \End_\F(\rho)) \to C^{i+j}(G_\Q,\End_\F(\rho) \otimes_\F \End_\F(\rho))) \\
& \to C^{i+j}(G_\Q,\End_\F(\rho)). 
\end{split}
\end{gather}
The first arrow is the standard cup product in group cochains, while the second arrow arises from composition in $\End_\F(\rho)$. The coboundary map $d$ along with this multiplication endows $C^\bullet(G_\Q,\End_\F(\rho))$ with a differential graded algebra (dg-algebra) structure (see e.g.\ \cite[Prop.\ 1.4.1]{NSW2008}). 

\begin{rem}
\label{rem: intro rem 2nd order def}
In particular, we see that $\rho_1$ extends to second order if and only if $e \smile e$ is a 2-coboundary, and that the choices for $f_2$ are a torsor over the 1-cocycles $Z^1(G_\Q,\End_\F(\rho))$. In particular, the cohomology class in $H^2(G_\Q, \End_\F(\rho))$ of $e \smile e$ is the obstruction to extending $\rho_1$ to second order. This is a well-known fact, cf.\ \cite[\S1.6, Remark; pg.\ 400]{mazur1989}.
\end{rem}

A prospective third-order lift extending $\rho_2$ has the form
\[
\rho_3 = 
\begin{pmatrix}
\rho & e & f_2 & f_3 \\
 & \rho& e & f_2 \\
 & & \rho & e \\
 & & & \rho
\end{pmatrix}. 
\]
One may calculate that $\rho_3$ is a homomorphism if and only if $f_3$ satisfies 
\[
df_3 = e \smile f_2 + f_2 \smile e. 
\]
For the induction step in general, whenever an $(n-1)$st order lift $\rho_{n-1}$ exists (and has the form \eqref{eq: rho_n def}), a prospective extension to $n$th-order has the form
\begin{equation}
\label{eq: rho_n def}
\rho_n = 
\begin{pmatrix}
\rho & e & f_2 & f_3 & \dotsm & f_n  \\
& \rho & e & & & f_{n-1} \\
& & \ddots & \ddots & &  \vdots \\
& &  & \rho& e & f_2 \\
&  && & \rho & e \\
& & && & \rho
\end{pmatrix},
\end{equation}
and is a bona-fide $n$th-order lift of $\rho$ if and only if 
\begin{equation}
\label{eq: nth massey product}
e\smile f_{n-1} + f_2 \smile f_{n-2} + \dotsm + f_{n-2} \smile f_2 + f_{n-1} \smile e 
\end{equation}
is a 2-coboundary. Because this expression is known to be a 2-cocycle, we may think of $H^2(G_\Q,\End_\F(\rho))$ as the set of possible obstructions to deforming representations. 

\subsubsection{Appearance of Massey products}

The theory of \textit{Massey products} in dg-algebras, with codomain in tuples of 1-cocycles and valued in $H^2(G_\Q,\End_\F(\rho))$, is the proper context for the computation of obstructions as above. Deferring details to \S\ref{sec: massey}, we give an overview of Massey products adapted to apply only to the computations of lifts of $\rho$ above. 

Massey products of degree $n$ on $C^1(G_\Q, \End_\F(\rho))$ are partially defined and multi-valued $\F$-multi-linear functions  
\begin{equation}
\label{eq: massey adjoint ex}
H^1(G_\Q, \End_\F(\rho))^{\times n} \lra H^2(G_\Q,\End_\F(\rho)),
\end{equation}
where each of the multiple values arises from a choice of \emph{defining system}. The term \emph{Massey product $\langle e_1, \dotsc, e_n\rangle$} for an $n$-tuple in the domain of \eqref{eq: massey adjoint ex} refers to the set of all values over all possible defining systems. In the present context, an $(n-1)$st order lift of the form \eqref{eq: rho_n def} furnishes a particular defining system for the Massey product on the $n$-tuple $(e, e, \dotsc, e)$, and the element of the Massey product $\langle e, \dotsc, e\rangle$ determined by this defining system is the cohomology class of the 2-cocycle written down in \eqref{eq: nth massey product}. 

Not all values of a Massey product on $n$-tuples of copies of $e$ are realized by the defining systems that test the existence of lifts of $\rho_1$ to $n$th-order, as general defining systems are allowed much more flexibility than the defining systems arising from lifts of $\rho_1$ to $(n-1)$st order. In turn, these other defining systems are related to other lifting questions, and vice versa. 

Indeed, in principle, Massey products have all of the information needed to determine deformations of $\rho$. This approach is taken in \cite{laudal2002}. However, one disadvantage of the theory of Massey products is its apparent dependence on a multitude of choices that grows inductively with the order of the deformation. This is why we confine ourselves to saying that the theory of Massey products in the dg-algebra $C^\bullet(G_\Q,\End_\F(\rho))$ encompasses all of the calculations needed to understand all possible lifts of $\rho$. 

Instead, we emphasize an alternative to Massey products -- $A_\infty$-products -- coming from the theory of $A_\infty$-algebras. See \S\ref{subsec: A-inf to Massey} for a thoroughgoing explanation of the relationship between Massey products and $A_\infty$-products. 

\subsubsection{$A_\infty$-algebras are homotopy-associative}
\label{sssec: A-inf intro for app}

The notion of a $A_\infty$-algebra is a generalization of the notion of dg-algebra, where multiplication, now denoted $m_2$, is no longer required to be associative. Instead, it is required to be associative up to homotopy, and further data keeping track of these homotopies comprises the rest of the $A_\infty$-algebra structure. This notion was first introduced by Stasheff  as a ``strong homotopy associative algebra'' \cite{stasheff1963}. 

We will proceed in this introduction by using the following informal definition of an $A_\infty$-structure on a $\Z$-graded $\F$-vector space $H$, where $\F$ is a field. See \S\ref{subsec: A-inf defn} for details. 
\begin{itemize}[leftmargin=2em]
\item $m_1$ denotes the differential on $H$, of degree 1 
\item  $m_2$ denotes multiplication (of degree 0), which may be non-associative but satisfies the Leibniz rule with respect to $m_1$
\item The rest of the $A_\infty$-structure consists of a sequence of homotopies $m_n, n \geq 3$, each of degree $2-n$. Together, they track failures of associativity. Namely, $m_3$ sends tuples $(a,b,c) \in H^{\times 3}$ $\F$-multilinearly to elements of $H$ whose boundary (with respect to $m_1$) is 
\[
(a \cdot b) \cdot c - a \cdot (b \cdot c),
\]
where the product refers to $m_2$. Subsequent $m_n$ have boundary measuring the failure of a generalized associativity notion for $n$ elements. 
\item Altogether, an $A_\infty$-algebra structure on $H$, written $m = (m_n)_{n \geq 1}$, consists of $\F$-linear maps
\[
m_n : (H^\bullet)^{\otimes n} \lra H^\bullet \quad \text{of graded degree } 2-n, \ n \in \Z_{\geq 1}, 
\]
satisfying certain compatibility conditions. In particular, it is required that $m_1 \circ m_1 = 0$, and $m_2$ satisfies the Leibniz rule with respect to $m_1$. 

\item The data of $n$th Massey products above appears in the restriction of $m_n$ to $H^1$, i.e.\
\[
m_n : (H^1)^{\otimes n} \lra H^2. 
\]
For more details on the relationship between $m_n$ and $n$th Massey products, see \S\ref{subsec: A-inf to Massey}. In summary, we take the perspective that the $A_\infty$-quasi-isomorphism $f$ is a way of ``choosing all defining systems of Massey products in advance.'' 
\end{itemize}
%We write $m = (m_n)_{n \geq 1}$ for an $A_\infty$-structure on $H$, where $m_n: H^{\otimes n} \to H$. 
We write $(H,m)$ for such an $A_\infty$-algebra. We defer details on morphisms of $A_\infty$-algebras to \S\ref{subsec: A-inf defn}.

In any quasi-isomorphism class of $A_\infty$-algebras, there are two extremal kinds of representatives. 
\begin{enumerate}[label=(\roman*)]
\item One is the \emph{dg-algebra case}, where the ``higher multiplications'' $m_n, n \geq 3$ vanish and, hence, $m_2$ is associative. Then $m_1$ is the differential map and $m_2$ is the multiplication map of a dg-algebra. 
\item The other is the \emph{minimal} case, where $m_1 = 0$; in other words, this $A_\infty$-algebra is equal to its cohomology. We sometimes write $m = (m_n)_{n \geq 2}$ for a minimal $A_\infty$-structure. 
\end{enumerate}
In fact, Kadeishvili \cite{kadeishvili1982} proved that when one starts with a dg-algebra $C^\bullet$, there is a minimal $A_\infty$-algebra structure $m$ on its cohomology complex $H^\bullet := H^\bullet(C^\bullet)$ that both
\begin{itemize}[leftmargin=2em]
\item extends to $m = (m_n)_{n \geq 1}$ the graded multiplication $m_2$ on $H^\bullet$ induced by the multiplication on $C^\bullet$, along with the trivial differential $m_1 = 0$ on $H^\bullet$, and
\item admits a quasi-isomorphism $f$ to $C^\bullet$. In particular, the quasi-isomorphism includes the data of a map of complexes from $H^\bullet$ to $C^\bullet$ that is a section of projection from cocycles to cohomology.
\end{itemize}
We use a refinement of Kadeishvili's result due to Merkulov \cite{merkulov1999}: a quasi-isomorphism $f: H^\bullet \to C^\bullet$ as above may be explicitly determined from a choice of section $H^\bullet \to C^\bullet$ of the standard projection from cocycles to cohomology. Such a section may be thought of as the following set of choices, using the differential $m_1$: for each $i \in \Z_{\geq 0}$, choose a decomposition of $C^i$ into $i$-coboundaries, $i$-cocycles that lift cohomology classes, and $i$-cochains that lift $(i+1)$-coboundaries. For details of the formula, see Example \ref{eg:KFT2}. 

\subsubsection{Application of $A_\infty$-algebras to deformations}

We now summarize the application of the framework above to deformations of a $\F$-linear representation $\rho$ of $G_\Q$. We start with adjoint cochains and cohomology 
\[
C^\bullet = C^\bullet(G_\Q,\End_\F(\rho)) \quad\text{and}\quad H^\bullet = C^\bullet(G_\Q,\End_\F(\rho)).
\]
As we noted in \S\ref{sssec: massey calc intro}, $C^\bullet$ is naturally a dg-algebra. The framework gives us an $A_\infty$-algebra structure $m$ on $H^\bullet$ and an $A_\infty$-quasi-isomorphism $f$ from $(H^\bullet,m)$ to $C^\bullet$. 

The structures $m$ and $f$ may be leveraged to produce a presentation for the deformation ring, using the following argument.

\noindent
\textbf{1.\ From $C^\bullet$ to deformations of $\rho$ via the Maurer--Cartan equation.}  
$C^\bullet$ has a direct connection with representations of $G_\Q$: 1-cochains are functions on $G_\Q$. Sacrificing the $\F$-linear formulation of the Massey products, the relationship between $C^\bullet$ and lifts of $\rho$ is concisely expressed by the \emph{Maurer--Cartan equation}. If $A$ is a local Artinian $\F$-algebra with residue field $A/\m_A \cong \F$, then lifts of $\rho$ to $A$ are in bijection the solutions to the Maurer--Cartan equation, denoted
\begin{equation}
\label{eq: dg MC intro}
\MC(C,A) := \{\xi \in C^1 \otimes_\F \m_A \mid d\xi + \xi \cdot \xi  = 0 \quad \text{in } C^2 \otimes_\F \m_A\},
\end{equation}
where we use shorthand $C^i = C^i(G_\Q,\End_\F(\rho))$. Full discussion of the Maurer--Cartan equation begins in \S\ref{subsec: MC functor intro}. 

\noindent
\textbf{2.\ Gauge action.}
 There is a standard notion of \emph{gauge action} on $\MC(C^\bullet,A)$, which we check amounts to conjugation of lifts to $A$ by elements of $1_{d \times d} + M_{d \times d}(\m_A)$. Thus, gauge equivalence classes in $\MC(C^\bullet,A)$ are deformations of $\rho$ to $A$.

\noindent
\textbf{3.\ The homotopy Maurer--Cartan equation} 
 There is an extension of the Maurer--Cartan equation from dg-algebras to $A_\infty$-algebras, sometimes known as the \emph{homotopy Maurer--Cartan equation}. On $\xi \in H^1 \otimes_\F \m_A$, it has the form 
\begin{align*}
\MC(&(H^\bullet,m),  A) :=  \{x \in H^1 \otimes_\F \m_A \mid  \\
& m_1(\xi) + m_2(\xi^{\otimes 2}) - \dotsm =  -\sum_{i = 1}^\infty  (-1)^{\frac{i(i+1)}{2}}m_i(\xi^{\otimes i}) = 0 \text{ in } H^2 \otimes_\F \m_A = 0\}.
\end{align*}

\noindent
\textbf{4.\ Functorality of the Maurer--Cartan functor.} 
 It is well-understood that Maurer--Cartan equation solutions are functorial, giving a map $\MC((H^\bullet,m), -) \to \MC(C^\bullet,-)$ arising from the quasi-isomorphism $f: H^\bullet \to C^\bullet$. Moreover, we show that quasi-isomorphisms of $A_\infty$-algebras induce isomorphisms on the sets of gauge equivalence classes of solutions to the Maurer--Cartan equation. Because the gauge action on $\MC((H^\bullet,m), -)$ turns out to be trivial, we find that $\MC(H^\bullet,-)$ represents the deformation problem for $\rho$.  

\noindent
\textbf{5.\ Presentation for the ring representing $\MC((H^\bullet,m), -)$.} 
 We now use the minimality ``$m_1 = 0$'' of the $A_\infty$-algebra structure $m$ on $H^\bullet$. In the Maurer--Cartan equation, this implies a crucial property that the equation for $C^\bullet$ lacks: the outputs of the equation are contained in $H^2 \otimes_\F \m_A^2$. From this property, we may read off a presentation for the ring representing the $A \mapsto \MC((H^\bullet,m), A)$ from the maps $m_n$. Also, in order to calculate $n$th-order deformations, only the $m_i$ for $i \leq n$ are required. Altogether, we get a presentation for $R_\rho$, which appears in Theorem \ref{thm: main A-inf}.

\section{Background on $A_\infty$-algebras}
\label{sec: A-infinity background}

In preparation to state the main results, and as this article is intended in part to acquaint number theorists with homotopy-algebraic notions, we now fully introduce the required background on $A_\infty$-algebras. We follow \cite[Ch.\ 9]{LV2012}. See \S\ref{sec:A-inf} for a more conventional introduction. 

\subsection{Definition via the bar construction}
\label{subsec: bar def}

Let $\hat T_\F(V)$ denote the complete free (associative) $\F$-algebra on a $\F$-vector space $V$. Let $\hat S_\F(V)$ denote the complete free commutative $\F$-algebra on $V$. A $\F$-basis $\{v_1, \dotsc, v_n\}$ for $V$ determines an isomorphism $\hat S_\F(V) \risom \F\lb v_1, \dotsc, v_n\rb$. That is, $\hat S_\F(V) \cong \prod_{n \geq 0} \Sym_\F^n V$; similarly, $\hat T_\F(V) \cong \prod_{n \geq 0} V^{\otimes n}$. 

When $H = H^\bullet$ is a $\Z$-graded $\F$-vector space, an $A_\infty$-algebra structure is a sequence $m = (m_n)_{n \geq 1}$ of $\F$-linear maps
\begin{equation}
\label{eq: m_n}
m_n : H^{\otimes n} \lra H \quad \text{of degree } 2-n, \quad \text{for } n \in \Z_{\geq 1}
\end{equation}
satisfying many compatibility relations (see \S\ref{subsec: A-inf defn} for this and further details). For the moment, we recall that when $m_n = 0$ for $n \geq 3$, the compatibility relations for $m_1, m_2$ are exactly the axioms of a differential graded algebra with differential $m_1$ and multiplication $m_2$. 

We introduce morphisms and quasi-isomorphisms of $A_\infty$-algebras. A \emph{morphism of $A_\infty$-algebras} $f : H \ra H'$ is a sequence $f = (f_n)_{n \geq 1}$ of maps
\begin{equation}
\label{eq: f_n}
f_n : H^{\otimes n} \lra H' \quad \text{of degree } 1-n, \quad \text{for } n \in \Z_{\geq 1},
\end{equation}
satisfying certain relations that will be explained later. One of the relations is that $f_1$ is a morphism of complexes $(H, m_1) \ra (H', m'_1)$. We call $f$ a \emph{quasi-isomorphism} when $f_1$ is a quasi-isomorphism of complexes. 

Both the definition of an $A_\infty$-algebra and the main result will become clearer by reformulating the notion of an $A_\infty$-algebra through a dualized version of the bar construction, which we now explain in a stepwise way. We use the notation $(-)^*$ to denote $\F$-linear duals. That is, $V^*$ is the $\F$-linear dual of a $\F$-vector space $V$. 

First, we suspend the maps $m_n$, i.e.\ we suspend $H$ using notation $\Sigma$ so that each $m_n$ induces a $\F$-linear map
\begin{equation}
\label{eq: m_n suspend}
\Sigma m_n : \Sigma H^{\otimes n} := (\Sigma H)^{\otimes n} \lra \Sigma H\quad  \text{ of degree } 1. 
\end{equation}
Then, presuming that $H^i$ is finite-dimensional for all $i \in \Z$, we take the linear duals of the composite of \eqref{eq: m_n suspend} with the projection $\Sigma H \ra \Sigma H^i$. Summing these dual maps over their domains yields 
\[
m_n^* : \Sigma H^* := \bigoplus_{i \in \Z} (\Sigma H^i)^* \lra (\Sigma H^*)^{\otimes n}, \text{ also of degree } 1.
\]
Next, we take the product over the codomains as $n$ varies, writing 
\[
m^* : \Sigma H^* \lra \prod_{n \geq 1} (\Sigma H^*)^{\otimes n}.
\]
Finally, we extend the domain of $m^*$ to the complete free associative graded $\F$-algebra $\hat T_\F \Sigma H^*$ via the Leibniz rule, producing
\[
m^* : \hat T_\F \Sigma H^* \lra \hat T_\F \Sigma H^* \quad  \text{of degree 1}. 
\]
Altogether, we write 
\[
\B^*(H) = \B^*(H, m) := (\hat T_\F \Sigma H^*, m^*, \pi)
\]
for this complete free graded associative $\F$-algebra with derivation $m^*$, where $\pi$ denotes the standard multiplication operation. 

Analogously, a sequence of maps $f = (f_n)_{n \geq 1}$ as in \eqref{eq: f_n} induces a homomorphism of complete free graded associative $\F$-algebras
\[
\B^*(f) : \B^*(H') \lra \B^*(H). 
\]

Notice that nothing we have said so far depends on the relations on $m = (m_n)_{n \geq 1}$ defining an $A_\infty$-algebra or the relations on $f$ defining a morphism of $A_\infty$-algebras. The following statement contains a concise and equivalent formulation of these relations. 
\begin{fact}[Bar construction]
\label{fact: bar}
Assume that $H^i$ is finite-dimensional for all $i \in \Z$. A sequence of maps $m=(m_n)_{n \geq 1}$ as in \eqref{eq: m_n}  make $(H, m)$ an $A_\infty$-algebra if and only if $\B^*(H,m)$ is a dg-algebra. That is, $m$ defines an $A_\infty$-structure on $H$ if and only if $m^*$ is a differential, i.e.\ $m^* \circ m^* = 0$. 

Likewise, a sequence of maps $f = (f_n)_{n \geq 1}$ as in \eqref{eq: f_n} make $f : H \ra H'$ a morphism of $A_\infty$-algebras if and only if the $\F$-algebra homomorphism $\B^*(f)$ is a homomorphism of dg-algebras, i.e.\ $\B^*(f) \circ m'^* = m^* \circ \B^*(f)$. Moreover, 
\begin{enumerate}[leftmargin=2em]
\item One can drop the condition that each $H^i$ is finite-dimensional and produce a co-complete co-free dg-coalgebra (see Definition \ref{defn: bar}). 
\item This construction induces an \emph{isomorphism} of categories between $A_\infty$-algebras and co-complete co-free dg-coalgebras. 
\end{enumerate}
\end{fact}

\begin{proof}
By direct computation. See e.g.\  \cite[Lem.\ 9.2.2 and \S9.2.11]{LV2012}. 
\end{proof}

\subsection{A theorem of Kadeishvili}
\label{subsec: kad intro}

The remainder of the background from homotopical algebra that we require originates in a theorem of Kadeishvili \cite{kadeishvili1982}. 
\begin{fact}[Kadeishvili]
\label{fact: kadeish}
Let $(C, d_C, m_{2,C})$ be a dg-$\F$-algebra. Let $H = H^\bullet(C)$ be the graded $\F$-vector space of cohomology of the complex $(C, d_C)$. 

There is an $A_\infty$-algebra structure $m = (m_n)_{n \geq 1}$ on $H$ and a quasi-isomorphism of $A_\infty$-$\F$-algebras 
\[
f = (f_n)_{n \geq 1} : (H, m) \ra (C, (m_{n,C})_{n\geq 1})
\]
where we let $m_{1,C} = d_C$ and $m_{n,C} = 0$ for $n \geq 3$ (the standard $A_\infty$-algebra structure arising from a dg-algebra structure). Moreover, these satisfy the properties 
\begin{enumerate}[leftmargin=2em]
\item $m_1 = 0$, i.e.\ $(H,m)$ is minimal
\item $m_2 = (m_{2,C} \mod{\mathrm{image}(d_C)})$
\item $\mathrm{pr} \circ f_1 = \mathrm{id}_H$, where $\mathrm{pr}$ is the projection from $\ker(d_C)$ to $H$. 
\end{enumerate}

These structures are unique up to non-unique isomorphism. 
\end{fact}

\begin{proof}
See e.g.\ \cite[Cor.\ 9.4.8]{LV2012}; see also \S\ref{subsec: dg to A-inf} for more details.
\end{proof}

The idea is that the standard graded algebra structure $(H, m_2)$ on the cohomology $H$ of a dg-algebra $C$ can be enriched into the structure of a minimal $A_\infty$-algebra $(H,m)$ that does not lose information from $C$. In \S\ref{subsec: dg to A-inf}, we explain that the choice of a homotopy retract between $H$ and $C$ induces a particular choice of $f$ and $m$. 

\begin{rem}
Another way of expressing Fact \ref{fact: kadeish} is that the higher $A_\infty$-products $(m_n, n \geq 3)$ encode the information lost by passing to the cohomology algebra. When no information is lost, i.e.\ there is a choice of $m$ in the statement such that $m_n = 0$ for $n \geq 3$, the dg-algebra is called \emph{formal}. The formal case makes for especially explicitly computable deformation theory (see Remark \ref{rem: cup case}).
\end{rem}

By way of giving an example of this fact, we introduce the Galois cohomology objects that appear in the main theorem, starting with the standard construction of Galois cohomology. Let $V$ be a $\F[G]$-module. Let 
\[
C^\bullet(G, V) \cong  \bigoplus_{i \geq 0} C^i(G,V)
\]
denote the complex of inhomogeneous continuous cochains on the profinite group $G$, graded by degree $i$. When $V$ has the structure of a $\F$-algebra, one may check (see e.g.\ \cite[Prop.\ 1.4.1]{NSW2008}) that the composition $m_{2,C}$ of the standard cup product of cochains with the multiplication map $V \otimes_\F V \ra V$, namely, 
\[
C^i(G, V) \otimes_\F C^j(G, V) \lra C^{i+j}(G,V \otimes_\F V) \lra C^{i+j}(G,V), 
\]
makes $C^\bullet(G,V)$ a dg-$\F$-algebra. That is, the Leibniz rule is satisfied. 

We write $H^\bullet(G,V)$ for the graded $\F$-vector space of cohomology of $C^\bullet(G,V)$. The dg-algebra structure on $C^\bullet(G,V)$ induces a graded algebra multiplication
\[
m_2 : H^\bullet(G,V) \otimes_\F H^\bullet(G,V) \lra H^\bullet(G,V)
\]
on $H^\bullet(G,V)$. Now we may apply Fact \ref{fact: kadeish}, producing an $A_\infty$-algebra structure $m= (m_n)_{n \geq 1}$ on $H^\bullet(G,V)$ extends the native dg-algebra structure on the graded algebra $(H^\bullet(G,V), 0, m)$. That is, it extends the usual cup product in cohomology 
\[
m_1 = 0, \quad m_2 = (m_{2,C} \mod{\mathrm{image}(d^C)}). 
\]

For the purposes of this introduction, our case of interest is where $V = \End_\F(\rho)$, where $\rho$ is absolutely irreducible as in \S\ref{sssec: irred case} above. 

\subsection{The classical hull}
\label{subsec: classical hull}

Finally, we define the \emph{classical hull} of a dg-algebra. 
\begin{defn}
\label{defn: classical hull}
The classical hull $\cA(C) = \cA(C,d,m)$ of a dg-algebra $(C, d, m_{2,C})$ is the ring $\cA(C) := C^0/d(C^{-1})$ concentrated in graded degree zero, taken as a (classical) ring. This functor is left adjoint to the functor sending classical (associative) algebras $(A, m)$ to the dg-algebra $(A[0], 0, m[0])$ concentrated in degree zero as $A$ and with the zero differential. That is, 
\[
\Hom_{\text{dg-}\F}(C, A[0]) = \Hom_\F(\cA(C), A).
\]
\end{defn}

Since we are most interested in the classical hull of the dg-algebra $(\hat T_F \Sigma H^\bullet, m^*, \pi)$ produced by the bar construction in Fact \ref{fact: bar}, we describe this case in particular. 

\begin{eg}
The classical hull of the dg-algebra $\B^*(H) = (\hat T_\F \Sigma H^*, m^*, \pi)$ produced by the dualized bar construction of Fact \ref{fact: bar} is
\[
\frac{\hat T_\F (\Sigma H^1)^*}
{(m^*((\Sigma H^2)^*)))}.
\]
Indeed, notice that any map $\hat T_\F \Sigma H^* \ra A$ factors through $\hat T_\F (\Sigma H^1)^* \ra A$, as $(\Sigma H^1)^*$ is the degree zero part of $\Sigma H^*$. Then, calculate using the Leibniz rule that the ideal generated by the projection of $m^*(\Sigma H^*)$ to $\hat T_\F (\Sigma H^1)^*$ is $(m^*((\Sigma H^2)^*)))$. 
\end{eg}

\section{Main results}
\label{sec: main results}

We present results toward each of the goals (1)-(4) on pg.\ \pageref{pg: twosteps}. In particular, we address deformations of pseudorepresentations in \S\ref{subsec: results2}. 

\subsection{Results, Part I: Determination of moduli spaces of representations}
\label{subsec: results1} 

As in \S\ref{subsec: red case}, let $\rho : G \ra \GL_d(\F)$ be a semi-simple representation with absolutely irreducible summands $\rho_i : G \ra \GL_{d_i}(\F)$, $1 \leq i \leq r$. We write
\[
\rho \cong \bigoplus_{i=1}^r \rho_i,
\]
thinking of this as a block diagonal decomposition of the homomorphism $\rho$. 

We impose the running assumption that $\rho_i \not\simeq \rho_j$ for $i \neq j$ (the \emph{multiplicity-free} condition). We call such $\rho$ a multiplicity-free \emph{residual semi-simplification}, as it turns out that the condition of deforming some $\F$-valued representation whose semi-simplification is $\rho$ is an open and closed condition in the moduli space of $d$-dimensional representations of $G$ valued in appropriately topologized $\F$-algebras. This connected component of moduli space, denoted $\Rep_\rho$, is set up in \S\ref{sec: profinite moduli}. 

Theorem \ref{thm: adapted pres} gives a presentation of $\Rep_\rho$ in terms of $A_\infty$-algebra structure on $H^\bullet(G, \End_\F(\rho))$. In this introduction we present the simplest case, where $\rho$ is irreducible, i.e.\ $r=1$. In this case $\Rep_\rho \cong \Spec R_\rho$, where $R_\rho$ is the usual deformation ring (see \S\ref{sssec: irred case}). We remark that the case $r=1$ has no less theoretical content than the case of general $r$, but has a simpler expression. 

\begin{thm}[{Special case ($r=1$) of Theorem \ref{thm: adapted pres}}]
\label{thm: main A-inf}
Let $\rho$ be an absolutely irreducible $\F$-valued representation of $G$. Assume that $H^i(G, \End_\F(\rho))$ is finite-dimensional for all $i \geq 0$. As described in Fact \ref{fact: kadeish}, there exists a minimal $A_\infty$-algebra structure $m= (m_n)_{n \geq 2}$ on $H = H^\bullet(G, \End_\F(\rho))$ that is compatible with the dg-algebra $C = C^\bullet(G, \End_\F(\rho))$ in that there exists quasi-isomorphism of $A_\infty$-algebras 
\[
f: (H,m) \lra (C, d_C, m_{2,C}) 
\]
Let $(\hat T_\F \Sigma H^\bullet(G, \End_\F(\rho))^*, m^*, \pi)$ denote the complete dg-algebra arising from the bar construction on $(H^\bullet(G, \End_\F(\rho)), (m_n)_{n \geq 2})$. 

These data determine a dg-algebra homomorphism 
\[
\B^*(H,m) = (\hat T_\F \Sigma H^\bullet(G, \End_\F(\rho)), m^*, \pi)  \lra R_\rho[0],
\]
which factors through an isomorphism of classical complete commutative algebras
\[
\frac{\hat S_\F \Sigma H^1(G, \End_\F(\rho))^*}
{(m^*(\Sigma H^2(G, \End_\F(\rho))^*))} \lrisom R_\rho, 
\]
from the abelianization of the classical hull of $\B^*(H,m)$ to $R_\rho$.  
\end{thm}

Conventional tangent and obstruction theory is an immediate corollary of this presentation for $R_\rho$. Let $h^i$ denote the $\F$-dimension of $H^i(G, \End_\F(\rho))$. 

\begin{cor}[Tangent and obstruction theory]
\label{cor: t and o theory intro}
	There is a tangent and obstruction theory for deformations of $\rho$, as outlined in \S\ref{sssec: irred case}. Moreover, there is a bound on Krull dimension
	\[
	h^1-h^2 \leq \dim(R_\rho) \leq h^1. 
	\]
\end{cor}
Corollary \ref{cor: t and o theory intro} was originally proved by Mazur \cite[\S1.5, Prop.\ 2, pg.\ 399]{mazur1989}. 

\begin{proof}[{Proof of Corollary \ref{cor: t and o theory intro}}]
The bound on Krull dimension can be read off from the presentation of $R_\rho$ in Theorem \ref{thm: main A-inf}.  The tangent and obstruction theory can be derived from the presentation as follows. The formula for the tangent space arises from the fact that $m^*(\Sigma H^2(-)^*)$ is valued in the square of the maximal ideal of $\hat S_\F \Sigma H^1(-)^*$. As far as obstructions, one can produce a class in $H^2(G,\End_\F(\rho))$ from the composition of the presentation of $R_\rho$ with a homomorphism $R_\rho \ra \F[\varep]/\varep^{n+1}$ corresponding to an $n$th-order lift $\rho_n$ of $\rho$. It turns out that this is equal to a class defined directly through the  calculations of obstructions using the theory of Massey products, as described in of \S\ref{subsec: illustrate}. 
\end{proof}

Here are some remarks that give further perspective on the presentation for $R_\rho$ in Theorem \ref{thm: main A-inf}.
\begin{rem}
It is implicit in the theorem statement that there is no canonical choice of $A_\infty$-structure on $H^\bullet(G, \End_\F(\rho))$. Indeed, the tangent theory of \S\ref{sssec: irred case} amounts to a canonical surjection of $\hat S_\F \Sigma H^1(G, \End_\F(\rho))^*$ onto $R_\rho/\m_\rho^2$, but one may readily calculate by hand to observe that it has no canonical lift to $R_\rho/\m_\rho^3$. 
\end{rem}

\begin{rem}
\label{rem: cup case}
Given the surjection $\hat S_\F \Sigma H^1(G, \End_\F(\rho))^* \rsurj R_\rho$, we see that the presentation for $R_\rho/\m_\rho^3$ is determined by the projection of $m^*$ to
\[
\Sigma H^2(G, \End_\F(\rho))^* \lra \Sym^2_\F (\Sigma H^1(G, \End_\F(\rho)^*). 
\]
In fact, this is the $\F$-linear dual of the cup product map $m_2$. This recovers the connection between the cup product and obstructions to second-order deformations that was mentioned in Remark \ref{rem: intro rem 2nd order def}. This observation has been applied to great effect when the minimal $A_\infty$-structure $(m_n)_{n \geq 2}$ may be chosen so that they vanish for $n \geq 3$, which is known as the \emph{formal} case. See, for example, \cite{GM1988}, for an instance of formal deformation theory found in nature, as is discussed a bit more in \S\ref{subsec: kuranishi}. 
\end{rem}

\begin{rem}
An abelianization appears in Theorem \ref{thm: main A-inf}. Accordingly, the theorem follows from a more general theorem, Corollary \ref{cor: r-pointed dual}, which is the main theorem of Part 2. Its additional generality consists of the following. 
\begin{itemize}[leftmargin=2em]
\item The source of the representations may be a finitely generated associative $\F$-algebra $E$. In particular, it need not be a Hopf algebra.
\item The target coefficients of representations may be non-commutative. That is, we allow for local associative finite-dimensional $\F$-algebras with residue field $\F$ as deformation-theoretic coefficients.
\item We may have $r \geq 1$. 
\end{itemize}
Corollary \ref{cor: r-pointed dual} is a new result in non-commutative deformation theory; see \S\ref{subsec: NCDT} for more comments on related results. Corollary \ref{cor: r-pointed dual} specializes to the case $r=1$ in Corollary \ref{cor: irred case A-inf} and yields Theorem \ref{thm: main A-inf} upon applying it to $E = \F[G]$. 
\end{rem}

\subsection{Geometric invariant theory} 
\label{subsec: GIT intro}

This section prepares notation for the statement of Theorem \ref{thm: main pseudo}, which provides presentations for pseudodeformation rings. This suite of notation expresses invariant subrings that arise from geometric invariant theory (GIT) and follows on work of Bella\"iche \cite{bellaiche2012}. We will suppress inexplicit invariant-theoretic formulations in favor of explicit combinatorial expressions.  For further details about translation from GIT to combinatorics, see \S\ref{subsec: apply to moduli}. 

Let the pseudorepresentation $D : G \ra \F$ arise from a semi-simple representation $\rho : G \ra \GL_d(\F)$ with distinct absolutely irreducible summands
\[
\rho \simeq \bigoplus_{i=1}^r \rho_i. 
\]

\begin{note}[{\cite[\S2.2]{bellaiche2012}}]
\label{note: cycles}
We set up the following combinatorial objects on the integers from $1$ to $r$. 
\begin{itemize}[leftmargin=2em]
\item Write $\bold r$ for the set $\{1, \dotsc, r\}$. 
\item A \emph{path} is a function $\gamma: \{0,\dotsc, l\} \ra \bold r$, for some $\ell \geq 0$. We write $l_\gamma = l$ for the \emph{length} of $\gamma$. 
\item We say that a path $\gamma$ \emph{goes from $i$ to $j$} when $\gamma(0) = i$ and $\gamma(l_\gamma) = j$. 
\item We call $\gamma$ \emph{closed} if $\gamma(0) = \gamma(l_\gamma)$. In this case, we may consider the domain of $\gamma$ to be $\Z/l_\gamma \Z$. 
\item We call a path $\gamma$  \emph{simple} if 
\[
\gamma(i) = \gamma(j) \text{ and } i \neq j \implies \{i,j\} = \{0,l_\gamma\},
\]
That is, a path is simple if it is injective, or it is closed and maximally injective. 
\item Write $SCP(l)$ for the set of simple closed paths in $\bold r$ of length $l$, and write $SCP(\bold r)$ for the set of all simple closed paths in $\bold r$ (of any length). 
\item A \emph{cycle} is an equivalence class of closed paths under the equivalence relation $\gamma \sim \gamma'$ defined by
\[
\gamma \sim \gamma' \buildrel\text{def}\over\iff
\left\{
\begin{array}{l}
l_{\gamma} = l_{\gamma'} =: l, \text{ and} \\
\exists \ k \in \Z/l\Z \text{  such that } \gamma(i) = \gamma'(i+k)\  \forall\ i \in \Z/l\Z.
\end{array}\right.
\]
\item A cycle is called \emph{simple} if one (equivalently, all) of its constituent closed paths is simple. 
\item Write $SC(l)$ for the set of simple cycles in $\bold r$ of length $l$, and write $SC(\bold r)$ for the set of all simple cycles in $\bold r$ (of any length). 
\item For $i,j \in \bold r$, write $SCC(i,j)$ for the set of paths $\gamma$ from $i$ to $j$ such that the concatenation of $\gamma$ with the length 1 path from $j$ to $i$ is a simple closed cycle. In particular, $i = j$ is allowed, but $SCC(i,i) = \varnothing$ in this case. (``SCC'' stands for ``simple closed complements.'') 
\item Given a cycle represented by a closed path $\gamma$ and $i,j \in \{0, 1, \dotsc, l_\gamma-1\}$, let $SCC_\gamma(i,j)$ be the subpath of $\gamma$ from $\gamma(i)$ to $\gamma(j)$. When $i < j$, this is the path given by the sequence $(\gamma(i), \gamma(i+1), \dotsc, \gamma(j))$. When $i> j$, the situation is similar, but note that this subpath of $\gamma$ (thought of as a loop) includes the sequence $(\gamma(l_\gamma-1), \gamma(0))$ in that order. When $i = j$, $SCC_\gamma(i,j)$ is taken to be the empty path. 
\end{itemize}
\end{note}

In the following notation, we express group cohomology as $\Ext$-groups for convenience. These expressions are all contained in the canonical isomorphism 
\[
\Ext^\bullet_{\F[G]}(\rho,\rho) \cong H^\bullet(G, \End_\F(\rho)).
\]
For more on this, see \S\ref{subsec: hochs-group}. 

\begin{note}
\label{note: Ext}
The following objects enrich the foregoing notation from finite sets to finite-dimensional $\F$-vector spaces coming from $\Ext$. Recall that $(-)^*$ refers to $\F$-linear duality. 
\begin{itemize}[leftmargin=2em]
\item For $i,j \in \bold r$,  let
\[
\Ext_G^k(j,i) := \Ext^k_{\F[G]}(\rho_j, \rho_i) \cong H^k(G, \rho_i \otimes_\F \rho_j^*). 
\] 
\item Given a path $\gamma$ on $\bold r$, we write 
\[
\Sigma \Ext^1_G(\gamma)^* := \bigotimes_{i=0}^{l_\gamma-1} \left(\Sigma \Ext^1_G(\gamma(i), \gamma(i+1))^*\right),
\]
where the order of tensor factors may matter. When $\gamma$ is a cycle, we use this same notation only when we are working with symmetric tensors, so that there is no dependence upon the choice of closed path in the cycle. 
\item Let $\cC(D)$ be the directed graph whose vertices $\{\rho_i\}_{i=1}^r$ and whose arrows from $\rho_i$ to $\rho_j$ are a choice of basis for $\Ext^1_G(\rho_j, \rho_i)$. We will refer to the property of being strongly connected, i.e.\ the existence of a directed path between any two vertices, as well as the decomposition into strongly connected components. 
\item Let $h_{ij}^k := \dim_\F \Ext^k_G(\rho_j, \rho_i)$. We will mainly use $k = 1, 2$. 
\item Let $H_1(\cC(D))$ be the simplicial homology of the simplicial 1-complex naturally arising from $\cC(D)$. Let $\bN\langle SC(\cC(D))\rangle$ be the free commutative monoid on simple cycles $\cC(D)$. Let $J$ be the kernel of its natural additive injection to $H_1(\cC(D))$. That is, $J$ consists of $\bN$-linear combinations of simple cycles that have the same underlying sets-with-multiplicity of arrows. Finally, let $h^2(\cC(D))$ be the set 
\[
J \smallsetminus \left(J \cdot (\bN\langle SC(\cC(D))\rangle \smallsetminus\{0\})\right)
\]
and let $H^2(\cC(D))$ be the $\F$-vector space with basis $h^2(\cC(D))$. That is, $h^2(\cC(D))$ is a minimal generating set for $J$ as a submonoid of $\bN\langle SC(\cC(D))\rangle$. 
\end{itemize}
\end{note}

Finally, we define an invariant subring that will appear often in our presentations for pseudodeformation rings. 

\begin{defn}
\label{defn: R1D}
Let $R^1_D$ denote the local ring that is the image of the map
\begin{equation}
\label{eq: R1D}
\hat S_\F \displaystyle\bigoplus_{\gamma \in SC(\bold r)} \Sigma\Ext^1_G(\gamma)^* \lra 
\hat S_\F \displaystyle\bigoplus_{i,j \in \bold r} \Sigma\Ext^1_G(j,i)^* \cong \hat S_\F \Sigma \Ext_{\F[G]}(\rho,\rho)^*. 
\end{equation}
arising natural from the canonical map from $\Sigma \Ext^1_G(\gamma)^*$ to the codomain. 
\end{defn}

\begin{rem}
The image of \eqref{eq: R1D} equals the invariant subring of the codomain induced by the canonical action of the torus $\bG_m^{\times r}$ on the vector space $\Ext^1_G(j,i)$ by the character labeled by the $r$-tuple $(a_k)_{k=1}^r \in \bZ^{\oplus r} \cong X^*(\bG_m^{\times r})$ determined by 
\[
a_k = \delta_{ik} - \delta_{jk}
\]
(using the Kronecker delta). That is, thinking of $\bG_m^{\times r}$ as the standard torus in $\GL_r$, its action on $\Ext^1_G(j,i)$ appears as its action on the $(i,j)$th coordinate of $r$-by-$r$ matrices by the adjoint action. 
\end{rem}

We will supply references for the following results of GIT in \S\ref{subsec: invariants quivers}. 

\begin{fact}
\label{fact: R1D}
$R^1_D$ is reduced, normal, and Cohen-Macaulay. If $\cC(D)$ decomposes into strongly connected components $\coprod \cC(D_a)$ where $D = \bigoplus_a D_a$, then there is a canonical isomorphism
\[
R^1_D \cong \hat\bigotimes_a R^1_{D_a},
\]
to a completed tensor product over $\F$. When $\cC(D)$ is strongly connected, then its Krull dimension is 
\[
\dim R^1_D = 1 - r + \sum_{1 \leq i, j \leq r} h^1_{ij}.
\]
\end{fact}

We substantiate this next basic fact in \S\ref{subsec: obs theory}.  
\begin{fact}
\label{fact: basic obstruction}
Let $K$ denote the kernel of \eqref{eq: R1D} and let $\m$ the maximal ideal of the codomain. There exists a canonical isomorphism 
\[
H_2(\cC(D))^* \lrisom K/\m K. 
\]
\end{fact}

\begin{rem}
The construction $\cC(D)$ is a \emph{quiver}, giving us access to the extensive literature studying representations of quivers and of quivers with relations. For more on this, see \S\ref{subsec: invariants quivers}. 
\end{rem}

\subsection{Results, Part II: determination of moduli spaces of pseudorepresentations}
\label{subsec: results2} 

Let $D = \psi(\rho) : G \ra \F$ be the $d$-dimensional pseudorepresentation induced by $\rho$, as in \S\ref{subsec: red case}. The proofs of the following results are found in \S\S\ref{subsec: PsDef proofs}-\ref{subsec: obs theory}. Here is the main theorem. 
\begin{thm}
\label{thm: main pseudo}
Let $D := \psi(\rho)$, where $\rho$ has $r$ distinct absolutely irreducible factors $\rho \simeq \bigoplus_{i=1}^r \rho_i$. Assume that $H^i(G, \End_\F(\rho))$ is finite-dimensional for all $i \geq 0$. 

Choose a structure of $A_\infty$-algebra $m=(m_n)_{n \geq 1}$ on $H^\bullet(G, \End_\F(\rho))$ and a quasi-isomorphism to the dg-algebra $(C^\bullet(G, \End_\F(\rho)), d_C, m_{2,C})$, as described in Fact \ref{fact: kadeish}, satisfying the  condition of compatibility with the decomposition 
\[
\End_\F(\rho) \cong \bigoplus_{i,j \in \bold r} \Hom_\F(\rho_i, \rho_j)
\]
explained in Example \ref{eg: r-pointed retract}. 

The choices above induce an isomorphism 
\begin{equation}
\label{eq: R_D pres}
\frac{R^1_D}{
\Bigg(\displaystyle\bigoplus_{i,j \in \bold r} m^* \Sigma \Ext^2_G(\rho_j, \rho_i)^* \otimes 
\Big(\displaystyle\bigoplus_{\gamma \in SCC(i,j)} \Sigma \Ext^1_G(\gamma)^*\Big)\Bigg)
}
\lrisom 
R_D.
\end{equation}
\end{thm}

\begin{rem}
Because the formula for $R_D$ is rather complex, we supply the following intuitive interpretation. 

\textit{Generators are cycles.} Only \emph{cycles} $\Ext^1_G(\gamma)^*$ of tensors of $\Sigma \Ext^1_G(\rho_j,\rho_i)^*$ ($i,j \in \bold r$) will be detected by pseudorepresentations. Indeed, a choice of extension class
\[
e = (e_{i,j}) \in \Ext^1(\rho,\rho) \cong \bigoplus_{i,j \in \bold r} \Ext^1(\rho_j, \rho_i)
\]
defines a first order deformation of the form (we take $r = 3$ for concreteness) 
\[
\rho_e := \rho + \varep e \cong 
\begin{pmatrix}
\rho_1 + \varep e_{11} & \varep e_{12} & \varep e_{13} \\
\varep e_{21} & \rho_2 + \varep e_{22} & \varep e_{23} \\
\varep e_{31} & \varep e_{32} & \rho_3 + \varep e_{33}
\end{pmatrix}
\]
and only products of the elements $e_{ij}$ over a cycle will appear in the diagonal, and thereby be detectable by the trace. The simple cycles generate the monoid of cycles, and $R^1_D$ is generated by these cycles. 

\textit{Relations are obstructed sub-paths of cycles.} However, not every cycle has factors that can multiply together and still form a homomorphism that is detectable by a central function. The obstructions to the appearance of a cycle represented by $\gamma$ consist precisely of elements of $\Ext^2_G(\rho_j,\rho_i)$ that are the image of $m_n$ on any sub-path $\gamma'$ from $j$ to $i$ of the cycle $\gamma$, that is
\[
m_n : \Ext^1_G(\rho_j, \rho_{\gamma(1)}) \otimes \dotsm\buildrel{n = l_{\gamma'} \text{ factors}}\over\dotsm\dotsm \otimes \Ext^1_G(\rho_{\gamma(n-1)}, \rho_i) \lra \Ext^2_G(\rho_j, \rho_i).
\]
This is the $m^* \Sigma \Ext^2_G(...)^*$-factor of the denominator of \eqref{eq: R_D pres}. The rest of the denominator accounts for the complement of $\gamma'$ in $\gamma$; that is, we must complete the obstructed path $\gamma'$ to the cycle $\gamma$ to calculate its influence on pseudorepresentations. 
\end{rem}

\begin{rem}
This expression for $R_D$ decomposes into the strongly connected components of the directed graph $\cC(D)$ of Notation \ref{note: Ext}. In that notation, we have
\[
R_D \cong \bigotimes_a R_{D_a}. 
\]
This is consonant with the fact that each cycle is supported on exactly one strongly connected component. Because this decomposition does not simplify the formulas, we do not use it in the expression of Theorem \ref{thm: main pseudo} and its corollaries. 
\end{rem}

A tangent and obstruction theory can be derived from Theorem \ref{thm: main pseudo}. First we discuss the tangent space. In what follows, we maintain the assumptions of Theorem \ref{thm: main pseudo}. 

Let 
\[
\frt_D := (\m_D/\m_D^2)^*
\]
denote the tangent space of $R_D$. 
\begin{cor}[The tangent space]
\label{cor: tangent pseudo} 
The tangent space $\frt_D$ is isomorphic to
\[
\bigoplus_{\gamma \in SC(\bold r)} 
\ker\left(
\Sigma\Ext^1_G(\gamma) \lra \bigoplus_{0\leq i, j < l_\gamma} \Sigma \Ext^2_G(\rho_{\gamma(j)}, \rho_{\gamma(i)}) \otimes \Sigma \Ext^1_G(SCC_\gamma(i,j))\right)
\]
where $\gamma$ is taken to be a representative for a simple cycle, and the map parameterized by $(\gamma, i,j)$ sends 
\[
e(\gamma) := e_0 \otimes \dotsm e_{l_\gamma} \mapsto m_{d_\gamma(i,j)+1}(e_j \otimes e_{j+1} \otimes \dotsm \otimes e_i) \otimes e(SCC_\gamma(i,j)).
\]
Here $e(SCC_\gamma(i,j))$ denotes the tensor factor of $e(\gamma)$ indexed by $SCC_\gamma(i,j)$, and $d_\gamma(i,j)$ is the number of steps from $\gamma(j)$ to $\gamma(i)$ around $\gamma$. 
\end{cor}

\begin{warn}
\label{warn: NC pseudo}
It is very important to keep in mind that direct sum expressions in the statement of Theorem \ref{thm: main pseudo} and Corollary \ref{cor: tangent pseudo} are non-canonical expressions of $R_D$ and $\frt_D$. Instead, there is a canonical filtration on $\frt_D$, whose graded pieces can be found among the summands in Corollary \ref{cor: tangent pseudo}. We now explain this canonical filtration. 
\end{warn}

There is a canonical filtration of the tangent space, the ``complexity filtration'' of Bella\"iche \cite[\S3]{bellaiche2012}. This is an increasing filtration, where lower complexity degree corresponds to greater reducibility (or less irreducibility), in the sense of ideals of reducibility defined in Bella\"iche--Chenevier \cite[\S1.5.1]{BC2009} and the derivative notion of complexity of a pseudodeformation of \cite[\S2.4]{bellaiche2012}. Correspondingly, a lower bound on reducibility (equivalently, an upper bound on complexity) produces a closed condition in $\Spec R_D$. 

In the terms of the presentation for the tangent space given in Corollary \ref{cor: tangent pseudo}, the complexity degree equals the number of tensor factors, i.e.\ the length of $\gamma$. So we may index it from $0$ to $r$ as 
\[
0 = \Fil_0 \frt_D \subset \Fil_1 \frt_D \subset \dotsm \subset \Fil_{r-1} \frt_D \subset \Fil_r \frt_D = \frt_D. 
\]

\begin{cor}
\label{cor: CFP}
The $k$th graded factor of the complexity filtration on $\frt_D$ is canonically isomorphic to the summand of the expression in Corollary \ref{cor: tangent pseudo} labeled by $\gamma \in SC(\bold r)$ such that $l_\gamma = k$. That is, there is a canonical isomorphism
\begin{align*}
&\frac{\Fil_k \frt_D}{\Fil_{k-1}\frt_D} \cong \\
&\bigoplus_{\substack{\gamma \in SC(\bold r) \\ l_\gamma = k}} 
\ker\left(\Sigma\Ext^1_G(\gamma) \to \bigoplus_{0\leq i, j < k} \Sigma \Ext^2_G(\rho_{\gamma(j)}, \rho_{\gamma(i)}) \otimes \Sigma \Ext^1_G(SCC_\gamma(i,j))\right) 
\end{align*}
where the map is as in Corollary \ref{cor: tangent pseudo}. 
\end{cor}

\begin{rem}
Implicit in the statement is the fact that the kernel does not depend on the choice of $A_\infty$-structure on $\Ext^\bullet_G(\rho,\rho)$. 
\end{rem}

\begin{rem}
\label{rem: refine Bel12}
This refines the main theorem of \cite{bellaiche2012}: Thm.\ 1 of \textit{loc.\ cit}.\ states that $\Fil_k \frt_D/\Fil_{k-1} \frt_D$ injects into a sum of kernels that is similar to the expression above, but lacks the $A_\infty$-products $m_n$ for $n \geq 3$. That is, only the terms arising from cup products are used in \textit{loc.\ cit}. 
\end{rem}

We derive bounds on the tangent dimension of $R_D$ from its presentation. Let
\[
h^k_{ij} := \dim_\F \Ext^k_{\F[G]}(\rho_j, \rho_i) \quad \text{for } 1 \leq i,j \leq r, 
\]
and let 
\[
h^1(\gamma) := \prod_{0 \leq i < l_\gamma} h^1_{\gamma(i), \gamma(i+1)} = \dim_\F \Ext^1_G(\gamma) \quad \text{for } \gamma \in SC(\bold r).
\]

\begin{cor}[Tangent dimension]
\label{cor: pseudo tangent dim}
The dimension of $\frt_D$ satisfies
\[
\sum_{\gamma \in SC(\bold r)} \left(h^1(\gamma) - \sum_{0 \leq i, j < l_\gamma} h^2_{ij} \cdot h^1(SCC_\gamma(i,j)) \right)
\leq \dim_\F \frt_D \leq \sum_{\gamma \in SC(\bold r)} h^1(\gamma). 
\]
\end{cor}

\begin{rem}
In some cases, the lower bound can be improved due to symmetry. For example, for a 2-dimensional Galois representation $\rho$ that is a direct sum of two characters $\rho_1$ and $\rho_2$, we have that
\[
\Ext^2_{\F[G]}(\rho_1, \rho_1) \cong \Ext^2_{\F[G]}(\rho_2, \rho_2) \cong H^2(G, \F), 
\]
and there is a symmetry in the cup products, as follows. For $b \in \Ext^1_{\F[G]}(\rho_2, \rho_1)$ and $c \in \Ext^1_{\F[G]}(\rho_1, \rho_2)$, 
\[
b \cup c = 0 \in \Ext^2_{\F[G]}(\rho_2, \rho_2) \iff c \cup b = 0 \in \Ext^2_{\F[G]}(\rho_1, \rho_1). 
\]
\end{rem}

Next we present an obstruction theory, expressed for deformations to $\F[\varep]/(\varep^{n+1})$. 
\begin{cor}[Obstruction theory]
\label{cor: obstruction}
Via the presentation of $R_D$ in Theorem \ref{thm: main A-inf}, there is associated to an $n$th-order pseudodeformation $D_n : G \ra \F[\varep]/\varep^{n+1}$ of $D$
\begin{enumerate}[leftmargin=2em]
\item an element of $\alpha(D_n) \in H_2(\cC(D))$ arising from the map $R^1_D \ra \F[\varep]/\varep^{n+1}$ associated to $D_n$.
\item If $\alpha(D_n) = 0$, there is an element $\beta(D_n)$ in 
\[
\bigoplus_{i,j \in \bold r} \Ext^2_G(\rho_j, \rho_i) \otimes \Big(
\bigoplus_{\gamma \in SCC(i,j)} \Ext^1_G(\gamma)\Big)
\]
associated to $D_n$.
\end{enumerate}
Moreover, $\alpha(D_n)$ and $\beta(D_n)$ vanish if and only if $D_n$ extends to an $(n+1)$st order pseudodeformation. 
\end{cor}

\begin{rem}
\label{rem: bounded deg obs}
There exists some $n_0 \in \Z_{\geq 1}$, dependent only on $r$ and $h^1_{ij}$ for $1 \leq i,j \leq r$, such that $\alpha(D_n)$ vanishes for all integers $n \geq n_0$. 
\end{rem}

When $H^2(G, \End_\F(\rho)) \cong \Ext^2_G(\rho,\rho) = 0$, the deformation theory of $\rho$ is smooth, or ``unobstructed.'' This well-known phenomenon is visible in Theorem \ref{thm: main A-inf} in the case $r=1$ (i.e.\ $\rho$ is irreducible), and remains the case for general $r \geq 1$ when we study the stack of representations $\Rep_\rho$ mentioned in \S\ref{subsec: results1}. We call the case $\Ext^2_{\F[G]}(\rho,\rho) = 0$ the ``representation-unobstructed case'' for clarity, when our focus is on the deformation theory of pseudorepresentations. 

When $r > 1$, it is possible for $R_D$ to be non-regular even when $\Ext^2_G(\rho,\rho) = 0$; actually, $R_D$ is not regular for generic choices of the integers $h^1_{ij}$. But we know from invariant theory that some ring-theoretic properties hold. 

\begin{cor}[The representation-unobstructed case]
\label{cor: RU case}
Assume $\Ext^2_G(\rho,\rho) = 0$, i.e.\ $H^2(G, \End_\F(\rho)) = 0$. Then the surjection $R^1_D \rsurj R_D$ of Theorem \ref{thm: main pseudo} is an isomorphism. In particular, we may read off the following properties of $R^1_D$. 
\begin{enumerate}[leftmargin=2em]
\item There exists an isomorphism
\[
\frt_D \lrisom \bigoplus_{\gamma \in SC(\bold r)} \Sigma\Ext^1_G(\gamma), \text{ and} 
\]
\item $R_D$ is reduced, normal, and Cohen-Macaulay. 
\item When $\cC(D)$ is strongly connected, the Krull dimension of $R_D$ is 
\[
\dim R_D = 1 - r + \sum_{i,j \in \bold r} h^1_{ij}. 
\]
\item When $\cC(D)$ is not strongly connected, then $R_D \cong \otimes_a R_{D_a}$ and $\dim R_D = \sum_a \dim R_{D_a}$, where $D = \bigoplus_a D_a$ is the decomposition of $D$ into strongly connected summands. 
\item An $n$th-order pseudodeformation $D_n$ of $D$ extends to an $(n+1)$st order pseudodeformation if and only if the obstruction class $\alpha(D_n)$ of Corollary \ref{cor: obstruction} vanishes. That is, the obstruction $\beta(D_n)$ is always zero, when it exists. 
\end{enumerate}
\end{cor}

\begin{proof}
The first statement is clear in Theorem \ref{thm: main pseudo}. The second statement follows the theorem combined with Fact \ref{fact: R1D}. 
\end{proof}

\begin{rem}
While $R^1_D$ is Cohen-Macaulay, for general $r$ and $h^1_{ij}$, it is very rare for $R^1_D$ to be Gorenstein, much less complete intersection or regular. See the discussion of \S\ref{subsec: invariants quivers}. But for small $r$ and $h^1_{ij}$, there are some cases where these ring-theoretic properties hold, and they are well-understood: see Example \ref{eg: R1D LCI}. 
\end{rem}

\begin{rem}
The Taylor--Wiles method \cite{TW1995}, and subsequent developments, involve auxiliary deformation problems where one arranges for $H^2(\text{aux},\End_\F(\rho))$ to vanish. This often goes under the moniker ``killing the dual Selmer group.'' This has the effect of making deformation rings $R_\rho^\mathrm{aux}$ isomorphic to a power series ring. We see in Corollary \ref{cor: RU case} what can be deduced about pseudodeformation rings $R_D$ from killing the dual Selmer group. This should be compared with the philosophy that, in situations where ``$\ell_0 = 0$'' so that the Taylor--Wiles method could possibly be applied (see e.g.\ \cite{CG2018} for more on this, including the definition of $\ell_0$), local Galois deformation rings -- when properly set up to correspond to Hecke algebras -- ought to be at least Cohen-Macaulay. See e.g.\ \cite[Thm.\ 4.6.2]{snowden2018}. 
\end{rem}

We conclude with bounds on the Krull dimension of $R_D$. 
\begin{cor}[Bounds on Krull dimension]
\label{cor: pseudo krull bounds}
Assume for simplicity that $\cC(D)$ is strongly connected. Then we have the following bounds on the Krull dimension of $R_D$. Letting $h^1_D := 1 - r + \sum_{i,j \in \bold r} h^1_{ij} = \dim R^1_D$, we have 
\[
h^1_D -
\sum_{\substack{\gamma \in SC(\bold r)\\ 0 \leq i,j < l_\gamma}} h^2_{ij} \cdot h^1(\gamma') 
 \leq \dim R_D \leq h^1_D
\]
\end{cor}

\begin{rem}
Note that $h^1(\gamma')$ is multiplicative in the $h^1_{ij}$, while $h^1_D$ is additive in the $h^1_{ij}$. Therefore, for large dimensions of $\Ext^1$-groups in the presence of non-zero $\Ext^2$-groups, this lower bound on $\dim R_D$ is trivial. 
\end{rem}

\begin{rem}
We remark on the computations involved in the proof of Theorem \ref{thm: main pseudo}. The ring $R^1_D$ is the invariant subring of the codomain of \eqref{eq: R1D} under the adjoint co-action of the $\F$-algebraic torus related to units of
\[
\End_{\F[G]}(\rho) \cong \bigoplus_{i=1}^r \End_{\F[G]}(\rho_i) \cong \F^r. 
\]
Similarly, as this action is linearly reductive (in any characteristic), the presentation of $R_D$ in Theorem \ref{thm: main pseudo} follows via a calculation of invariants from Theorem \ref{thm: adapted pres}, which generalizes Theorem \ref{thm: main A-inf} to the case that $\rho$ is semi-simple with distinct simple factors. 
\end{rem}

\subsection{Amplification: Galois representations with conditions}
\label{subsec: Gal conditions}

In this section, we state a meta-result: all of the theorems and corollaries of \S\ref{subsec: results1} and \S\ref{subsec: results2} may be applied to deformation rings and pseudodeformation rings parameterizing Galois representations \textit{satisfying additional Galois-theoretic conditions} of certain kinds. 

Let $\cC$ be a condition that applies to finite-length $\F[G]$-modules. We say that $\cC$ is a \emph{stable} condition when the full subcategory of finite-length $\F[G]$-modules satisfying $\cC$ is closed under the formation of subquotients and finite direct sums. It has been understood, since the work of Ramakrishna \cite{ramakrishna1993}, that there exists a quotient $R_\rho \rsurj R^\cC_\rho$ parameterizing exactly those deformations with property $\cC$. 

For stable $\cC$, the author's joint work with Wake \cite{WWE4} (see \S\ref{subsec: stable or CH} for a summary) explains that 
\begin{enumerate}[leftmargin=2em]
\item there exist quotient algebras of $\F[G]$ factoring the action on representations with residual pseudorepresentation $D$ and condition $\cC$, and 
\item there exists a sensible notion of ``pseudorepresentation of $G$ with property $\cC$'' and a quotient $R_D \rsurj R_D^\cC$ parameterizing exactly those pseudodeformations with property $\cC$.
\end{enumerate}

In the following theorem statement, we use a quotient algebra $E^\cC$ of $\F\lb G \rb$ that is constructed in \S\ref{subsec: stable constr}, which has the property that it factors the action map on exactly those finite-length $\F[G]$-modules that satisfy condition $\cC$. 

\begin{thm}
\label{thm: conds}
Let $\cC$ be a stable condition on finite length $\F[G]$-modules. Let $\rho$ be a $\F$-linear representation of $G$ that is semi-simple with distinct absolutely irreducible factors satisfying $\cC$. Then the dg-sub-algebra
\[
C^\bullet(E^\cC, \End_\F(\rho)) \subset C^\bullet(G, \End_\F(\rho)),
\]
and a $A_\infty$-algebra structure of Fact \ref{fact: kadeish} on its cohomology $H^\bullet(E^\cC, \End_\F(\rho))$ that induces
\begin{enumerate}[leftmargin=2em]
\item when $\rho$ is irreducible, a presentation of $R^\cC_\rho$, as in Theorem \ref{thm: main A-inf}, and 
\item when $D$ is the pseudorepresentation induced by $\psi$, a presentation of $R^\cC_D$, as in Theorem \ref{thm: main pseudo}. 
\end{enumerate}
\end{thm}

All of the corollaries to Theorems \ref{thm: main A-inf} and \ref{thm: main pseudo} also apply in the condition $\cC$ context of Theorem \ref{thm: conds}, as they hinge only on the existence of the $A_\infty$-algebra structure on cohomology and its relation to the deformation rings. 

\begin{eg}
Stable conditions $\cC$ of interest in the study of Galois representations include such conditions as 
\begin{enumerate}[leftmargin=2em]
\item when $G = G_F$ is the absolute Galois group of a $p$-adic field $F$, or when $G$ admits a homomorphism $G_F \ra G$, we ask for the property that the $G_F$-action arises from the $\overline{F}$-points of a finite flat group scheme defined over the ring of integers $O_F$ of $F$. 
\item More generally than (1) when $F/\Q_p$ is unramified, there are Fontaine--Laffaille conditions, which are $p$-integral crystalline conditions. 
\end{enumerate}
\end{eg}

\begin{rem}
\label{rem: Ext2 for C}
For a choice of subcategory $\cC$ of the category of finite-length $\F[G]$-modules, there may exist a notion of $\Ext^\bullet_\cC(\rho, \rho')$ for $\rho, \rho' \in \cC$. We emphasize that the $\Ext^\bullet_{E^\cC}(\rho, \rho)$ above may not be the same as $\Ext^\bullet_\cC(\rho,\rho')$ in all degrees. It is, however, the same in degrees 0 and 1. In degree 2, we have
\[
\Ext^2_{E^\cC}(\rho, \rho) \subset \Ext^2_\cC(\rho, \rho). 
\]
Nonetheless, the $A_\infty$-structure on the Hochschild cohomology $H^\bullet(E^\cC,\End_\F(\rho))$ correctly calculates $R_\rho^\cC$ and $R_D^\cC$. In fact, this is a general observation: isomorphism in degrees 0 and 1 and injection in degree 2 result in the same classical deformation problem; see e.g.\ \cite[Thm.\ 2.4]{GM1988}. We explain how this difference arises in Remark \ref{rem: Ext2 discrepancy}. 
\end{rem}

\begin{rem}
It is possible to produce an version of $E^\cC$ for some conditions $\cC$ which are of arithmetic importance but are not stable, such as the ``ordinary'' condition on 2-dimensional representations and pseudorepresentations studied in \cite{mazur1989, WWE1, CS2019}. We describe this in \S\ref{subsec: stable or CH}. 
\end{rem}

\section{Complements}
\label{sec: complements}

In this section we discuss an alternate formulation of the main theorems in terms of Massey products, examples that illustrate the main theorems and computations of the paper, relationships with other works in number theory, related and/or antecedent works outside number theory, and acknowledgements of the influence of colleagues. Here is a list of contents. 
\begin{itemize}[leftmargin=3em]
\item[\S\ref{subsec: massey intro}] Massey products and their relationship to $A_\infty$-products, and previous appearances of Massey products in number theory in work of Sharifi \cite{sharifi2007}. 
\item[\S\ref{subsec: eg NT}] The computations of ranks of $p$-adic modular Hecke algebras in terms of $A_\infty$-products that appear in \S\ref{sec: ranks}.
\item[\S\ref{subsec: NRIT}] The case of the trivial representation. 
\item[\S\ref{subsec: GV2018}] The derived Galois deformation rings of Galatius--Venkatesh \cite{GV2018}, and the analogue of the cotangent complex in our setting. 
\item[\S\ref{subsec: NCDT}] Non-commutative geometry and deformation theory. 
\item[\S\ref{subsec: kuranishi}] The Kuranishi map. 
\item[\S\ref{subsec: acknow}] Acknowledgements. 
\end{itemize}

In particular, we point out in \S\ref{subsec: NCDT} how the content of Part 2 of this paper relies on and also advances the line of inquiry in non-commutative deformation theory pursued by Laudal \cite{laudal2002} and Segal \cite{segal2008}. 

\subsection{Massey products}
\label{subsec: massey intro}

Massey products and their defining systems provide an alternative to products comprising $A_\infty$-structures for the purposes of this paper. (See \S\ref{sec: massey} for an introduction to Massey products.) This is true in a formal sense: the main theorems stated above are the outcome of Part 3 of this paper, which in turn relies on results in non-commutative deformation theory in Part 2. The main result of Part 2, Corollary \ref{cor: r-pointed dual}, gives a presentation of the completion $\F[G]^\wedge_\rho$ of $\F[G]$ at the kernel of $\rho$ in terms of $A_\infty$-structures and certain choices of idempotents. In comparison, the main theorem of \cite{laudal2002} gives an expression of the same algebra in terms of Massey products. We make further comparisons with previous work of Laudal and Segal in \S\ref{subsec: NCDT}. 

The definition of a Massey product, their defining systems, and the relationship between these and $A_\infty$-algebras is given in \S\ref{sec: massey}. We explain how Massey products determine non-commutative deformation theory in \S\ref{sec: lifts-massey}. The emphasis of \S\ref{sec: lifts-massey} differs from the rest of the paper: we illustrate how Massey products naturally arise when doing explicit computations of deformations in a continuation of \S\ref{subsec: illustrate}. Since any $A_\infty$-product may be realized as a Massey product, it also serves to concretely illustrate the influence of $A_\infty$-products on deformation theory. 

We discuss other studies of Massey products in Galois cohomology in \S\ref{subsec: NRIT}. Also, Massey products in the Galois cohomology of an endomorphism algebra has been connected to deformations of Galois representations and ranks of Hecke algebras in \cite{WWE3}. We discuss this next. 

\subsection{Ranks of Hecke algebras}
\label{subsec: eg NT}

In \S\ref{sec: ranks}, as an example of an application of our main results, we determine the ranks of some $p$-adic modular Hecke algebras in terms of the presentations given by $A_\infty$-products. This relies on a known isomorphism $R^\cC_\star \risom \bT/\m \bT$, where 
\begin{itemize}[leftmargin=2em]
\item $\bT$ is the Hecke algebra in question,
\item $\m$ is the maximal ideal, with residue field $\F$, of the regular local base ring over which $\bT$ is known to be free and for which we are measuring the rank of $\bT$, so we can calculate $\mathrm{rank}\, \bT = \dim_\F \bT/\m \bT = \dim_\F R^\cC_\rho$
\item ``$\star$'' in $R^\cC_\star$ stands in for either 
\begin{itemize}[leftmargin=2em]
\item $\rho$, a 2-dimensional absolutely irreducible $\F$-valued representation of a global Galois group $G$, which the residual semi-simplification associated to the Hecke eigensystem modulo $p$ cut out by the maximal ideal of $\bT$; or 
\item $D$, a 2-dimensional pseudorepresentation given by $D = \psi(\rho)$ for some 2-dimensional representation $\rho$ of $G$ such that $\rho \cong \chi_1 \oplus \chi_2$ where $\chi_1 \neq \chi_2$, determined similarly by the residual Hecke eigensystem of $\bT$
\end{itemize}
\item $\cC$ is a condition on finite-length $\Z_p[G]$-modules 
\item $R^\cC_\star$ is a deformation (resp.\ pseudodeformation) ring of $\rho$ with condition $\cC$
\end{itemize}
This rank is an expression of size of a congruence class of modular eigenforms modulo $p$. The result of each example given in \S\ref{sec: ranks} is an expression of $\mathrm{rank}\, \bT$ in terms of an arithmetic invariant expressed in terms of the vanishing of $A_\infty$-products (or, equivalently, Massey products). 

The first two examples that we give in \S\ref{sec: ranks} are drawn from the finite-flat case and the ordinary case of Wiles's $R \cong \bT$ theorem \cite{wiles1995}. In contrast, the third example is residually reducible, having to do with the Galois representations and modular forms appearing in Ribet's proof of the converse to Herbrand's theorem \cite{ribet1976}. The fourth case is also residually reducible and set in the case of Mazur's study of the Eisenstein ideal \cite{mazur1978}. In this case, the rank of $\bT$ was determined in \cite{WWE3} using Massey products, but it is not possible to use $A_\infty$-products; the contrast is discussed. 

\subsection{Galois cohomology of the trivial representation}
\label{subsec: NRIT}

The universal commutative deformation ring of a character (i.e.\ 1-dimensional representation) is relatively straightforward: it does not depend on the character, and is simply a completed and abelianized group algebra, as pointed out by Mazur \cite[\S1.4]{mazur1989}. Thus Theorem \ref{thm: main A-inf} gives a cohomological expression for the completion of $\F[G]^\mathrm{ab}$ at the kernel of the trivial $\F$-valued representation, which is naturally isomorphic to $\F\lb G^{\mathrm{ab}, \text{pro-}p}\rb$, the completed group algebra of the maximal abelian pro-$p$ quotient group of $G$. 

In Part 2 we prove a non-commutative version of Theorem \ref{thm: main A-inf} with exactly the same formula for the deformation ring. This non-commutative deformation ring is simply the completion $\F[G]^\wedge$ of $\F[G]$ at the kernel of the trivial representation, which is the completed group algebra of the image of $G$ in its modulo $p$ Malcev completion. This image retains much more information about $G$ than $G^{\mathrm{ab}, \text{pro-}p}$. 

There is a particular case where the $A_\infty$-products controlling $\F[G]^\wedge$ are particularly classically well-studied. Namely, let $G = G_F$ be the absolute Galois group of a field $F$ with characteristic different than $p$ and containing the $p$-roots of unity. Upon a choice of $p$th root of unity, the dg-$\F_p$-algebra
\[
\bigoplus_{i \geq 0} C^i(G_F, \mu_p^{\otimes i}) \cong \bigoplus_{i \geq 0} C^i(G_F, \F_p)
\]
has cohomology $\F_p$-algebra whose form is given by Milnor $K$-theory according to the norm residue isomorphism theorem of Rost and Voevodsky \cite{voevodsky2011} (i.e.\ the proved motivic Bloch--Kato conjecture, or Milnor conjecture when $p=2$). As $\F_p \cong \End_{\F_p}(\rho)$ for any 1-dimensional representation $\rho$ over $\F_p$, the $A_\infty$-structure (or higher Massey products) enrich this ring and retain extra information. Since the work of Hopkins--Wickelgren \cite{HW2015}, there has been attention to the vanishing of higher Massey products on $H^1(G_F, \F_p)$ and their links with the arithmetic of $F$. For example, as proved in  \cite{HW2015, MT2017} triple Massey products on $H^1(G_F, \F_p)$ vanish. There has also long been interest in determining the structure of pro-$p$ completions of Galois groups, which is clearly very much related. The interested reader can look into the extensive literature on these topics; the introduction of \cite{MT2017} contains a survey. 

We observe that the natural $A_\infty$-structure on cohomology $H^\bullet(G_F, \End_\F(\rho))$ for arbitrary $\rho$ provides the setting for an ``unstable'' generalization of these questions, at least when $F$ contains the $p$th roots of unity. By ``unstable,'' we mean that $\rho$ becomes trivial after restriction to a finite subgroup of $G_F$. Indeed, just as the existence of $F$-points on an algebraic variety is connected with the vanishing of a triple Massey product in \cite{HW2015}, the study of deformations of Galois representations has been motivated by its applications to arithmetic algebraic geometry. 

In contrast to the setting of the norm residue isomorphism theorem, the study of Galois groups of global fields with restricted ramification, $G_{F,S}$, is quite different. For one thing, the norm residue isomorphism theorem does not apply. It also is understood, for example, that there are non-vanishing triple Massey products. See in particular \cite[Ex.\ 2.11]{HW2015}, which is due to G\"artner, and the references in \cite[\S1]{HW2015}. 

Similarly, Sharifi \cite{sharifi2007} works over the cyclotomic $\Z_p$-extension of a number field, so that $p$-power Kummer extensions appear in the first cohomology of the trivial representation. Thus the work of Sharifi can also be interpreted deformation-theoretically as deformations of the trivial character. He relates the vanishing behavior of certain Massey products (in cohomology with restricted ramification) to Iwasawa-theoretic class groups. 

\subsection{The derived deformation rings of Galatius--Venkatesh}
\label{subsec: GV2018} 

We draw some comparisons between the approaches to deformation theory in present paper and the work of Galatius--Venkatesh \cite{GV2018} on derived Galois deformation rings. On the way to doing this, we explain the limitations and advantages of the setting of $A_\infty$-algebras chosen in this paper. 

We expect that the dg-algebra $C^\bullet(G, \End_\F(\rho))$ and the $A_\infty$-algebra structure on its cohomology have the information of a derived enrichment of conventional deformation theory that we study in this paper, in the sense of e.g.\ \cite[\S3]{DAG-X}. The coefficient rings of this enrichment are $\F$-augmented dg-Artin algebras that are associative but not commutative in any sense. 

To this author's knowledge, the most straightforward example of an exposition of such a derived enrichment of a moduli functor of representations has been done by Kapranov \cite{kapranov2001}. He studies to local systems on a finite CW-complex valued in an affine algebraic group over $\C$, with coefficient rings in \emph{commutative dg-$\C$-algebras.} 

The contrast in coefficient rings between the present work and \cite{kapranov2001} are indicative of the reasons for our choice of setting. 
\begin{itemize}[leftmargin=2em]
\item It is well known that only over $\Q$ do commutative dg-algebras furnish a satisfactory category of coefficient rings for derived algebra geometry. In positive characteristic or mixed characteristic, where the desired applications of this paper are located, formulations of Koszul duality between Lie and commutative operads are topics of contemporary homotopy-theoretic research. Indeed, since the first version of this manuscript appeared, Brantner and Mathew have proposed a notion of \textit{partition Lie algebra} as a Lie notion that furnishes a dual to simplicial commutative rings \cite{BM2019}. 
\item In contrast, the associative operad is well-understood to be self-dual in any characteristic. In this introduction, Koszul duality of operads is visible in the bar construction of \S\ref{subsec: bar def}: it sends an $A_\infty$-algebra to a dg-coalgebra, but an analogue sends an $L_\infty$-algebra (the Lie version of $A_\infty$) to a commutative dg-coalgebra (see e.g.\ \cite{LV2012}), over $\Q$. Because we work in arbitrary characteristic, we have written $\End_\F(\rho)$ instead of $\ad\rho$ to emphasize that we choose the associative algebra structure on the endomorphism ring, as opposed to its induced Lie algebra structure. Accordingly, our strategy to determine commutative deformation rings and other moduli spaces with commutative coefficients is to stay in the associative (non-commutative) world at least until Koszul duality is applied, and then finally abelianize at the end. 
\item Consequently, we do not work with representations valued in general algebraic groups, instead focusing on matrix-valued representations. For the cohomology controlling the deformations of algebraic group-valued representations is intrinsically valued in a Lie algebra that is not naturally induced by an associative algebra.
\item We also work in constant characteristic. Some additional Bockstein-type map is needed to control deformations to mixed characteristic. 
\end{itemize}

In contrast, Galatius--Venkatesh \cite{GV2018} are able to work with algebraic groups $\cG$ other than $\GL_n$ and mixed characteristic coefficient rings. In order to do this, instead of using commutative dg-algebras, they work with simplicial commutative rings as coefficient rings. They are able to calculate the cotangent complex of their derived deformation problem, which is the analogue of our $H^\bullet(G, \End_\F(\rho))$, and show that it matches the Andr\'e--Quillen cohomology of the derived deformation ring \cite[Lem.\ 5.10]{GV2018}. But as far as we are aware, only the more recent work of Brantner--Mathew, formulating partition Lie algebras, has the potential to supply structure on the cotangent complex that retains the information of the deformation problem. What we gain from a more limited choice of setting than \cite{GV2018} is that there is a well-established and integer-indexed additional structure on $H^\bullet(G, \End_\F(\rho))$ -- the $A_\infty$-structure -- that controls our deformation ring. 

Finally, we want to highlight the extreme contrast in the terminal results of the present paper compared to those of Galatius--Venkatesh \cite{GV2018}. The main results of this paper express the influence of $H^2(G,\End_\F(\rho))$ on obstructions to (classical) deformations. In particular, the results here are trivial when $H^1(G,\End_\F(\rho))$ is zero, as there are then no non-trivial classical deformations.  In contrast, the derived deformation problem to which the terminal results of \cite{GV2018} applies is expected to have no non-trivial classical deformations, yet still have non-trivial derived deformations. 

For another approach, producing an \emph{analytic} derived moduli space of representations of a profinite group, see \cite{antonio2017}. 

\subsection{Non-commutative deformation theory}
\label{subsec: NCDT}

As we have mentioned above, Part 2 has a new result in non-commutative deformation theory that is applied in Part 3, along with some results of \cite{WE2018}, to prove the main theorems stated in this introduction. The main results of Part 2 are Theorem \ref{thm: r-pointed a-inf} and Corollary \ref{cor: r-pointed dual}. We want to make some comments about how these results are related to other work in non-commutative deformation theory. 

In non-commutative deformation theory, the content of these results addresses what is called the ``deformation theory of $r$-points.'' Here we have an associative $\F$-algebra $E$. A ``point'' of $E$ is a maximal ideal $\m$ of $E$, and we extend $\F$ if necessary so that $E/\m \cong M_d(\F)$. We now take $r$ points, thought of as surjective representations $\rho_i : E \ra M_{d_i}(\F)$ for $i = 1, \dotsc, r$ cutting out distinct maximal ideals. One principal distinction from commutative deformation theory is that distinct points can have extensions between them, i.e.\ $\Ext^1_E(\rho_i, \rho_j)$ can be non-trivial when $\rho_i \not\simeq \rho_i$. This does not happen in the commutative setting. Let $\rho := \bigoplus_{i=1}^r \rho_i$ as usual. In doing non-commutative deformation theory, we are interested in determining the completion $E^\wedge_\rho$ of $E$ at the kernel of $\rho$. 

Our main result on non-commutative deformation theory (Theorem \ref{thm: r-pointed a-inf}) follows on previous work of Segal \cite[\S2]{segal2008}, which in turn is a development of work of Laudal \cite{laudal2002}. Segal proves a result that is very similar to Corollary \ref{cor: r-pointed dual}(2): in \cite[Thm.\ 2.14]{segal2008}, he proves that what we call $R^\mathrm{nc}_\rho$ is isomorphic to what we call $R$ (in the notation of Corollary \ref{cor: r-pointed dual}). The difference is that we keep track of data that \emph{determines an isomorphism} between $R^\mathrm{nc}_\rho$ and $R$, and determines a presentation for $E^\wedge_\rho$ in terms of cohomology. This is the data of a homotopy retract between the Hochschild cochain complex $C^\bullet(E,\End_\F(\rho))$ and its cohomology, which we explain in \S\ref{subsec: dg to A-inf}. We also carefully keep track of the some choices of idempotents that we use to remove Segal's assumption that $d_i$ (the dimensions of the simple summands of $\rho$) equals 1 for all $i$. Thus we have identified data that determines a presentation of $E^\wedge_\rho$ and the deformation functor, while Segal's approach identifies the isomorphism class of $E^\wedge_\rho$ and its deformation functor. Laudal \cite{laudal2002} also characterizes the isomorphism class of $E^\wedge_\rho$ in terms of Massey products. Thus $E^\wedge_\rho$ is described by Laudal in an inductive way, as we discuss more in \S\ref{subsec: massey express}. 

As we carry this out, we are careful in \S\ref{subsec: gauge} to make sure that notions of non-commutative gauge equivalence of Maurer--Cartan elements correspond to conjugacy classes of representations. This is well-understood in characteristic zero (see e.g.\ \cite[Defns.\ 2.3 and 2.9]{segal2008}), but appears less often in the literature in arbitrary characteristic because it is often expressed as an exponential. For this purpose, we found Prout\'e's study \cite{proute2011} of twisting morphisms very useful (and also \cite{CL2011}). The statement of the decomposition theorem (Theorem \ref{thm: decomp}) for $A_\infty$-algebras by Chuang--Lazarev \cite{CL2017} was also very helpful. 

There is a multitude of additional work along these lines in representation theory and non-commutative geometry, of which we can mention only a couple more. 
\begin{itemize}[leftmargin=2em]
\item The monograph of Le Bruyn \cite{lebruyn2008} summarizes results from non-commutative geometry that are related to the content of Part 2, as well as the Cayley--Hamilton algebra theory and pseudodeformation theory dealt with in Part 3. Le Bruyn especially focuses on the representation-unobstructed case (in the terminology of Corollary \ref{cor: RU case}). This is especially relevant (see \S\ref{subsec: GIT intro} and \S\ref{subsec: invariants quivers}) for the study of the ring $R^1_D$ that appears in the main theorem on the pseudodeformation ring (Theorem \ref{thm: main pseudo}). 
\item Keller \cite{keller2001, keller2002, keller2006} explains the connections between $A_\infty$-algebras and categories of representations. These have much of the same content as our results or Segal's results \cite[\S2]{segal2008}, but are not stated in terms of deformation theory. See especially \cite[\S2]{keller2002}. 
\item Maurer--Cartan elements and their gauge equivalences were promoted by the work of Kontsevich, e.g.\ \cite{kontsevich2003}. For an introduction to non-commutative deformations and $A_\infty$-algebras similar to our approach, see e.g.\ \cite{KS2009}. 
\end{itemize}

\subsection{The Kuranishi map}
\label{subsec: kuranishi}

To conclude our discussion of complements, we discuss one more perspective on our presentation of the deformation ring in terms of cohomology in Theorem \ref{thm: main A-inf}. Such results have been sought after in the context of the variation of flat connections on manifolds. In this case, there are functions whose analytic germ is the denominator in the expression of the deformation ring of Theorem \ref{thm: main A-inf}. Indeed, this germ has a natural extension to an analytic function in the neighborhood of the origin in the appropriate cohomology vector space $H^1(\End (\rho))$, known as the Kuranishi obstruction map; see e.g.\ \cite[Ch.\ 12]{MMR1994}. 

Likewise, in \cite{GM1988}, Goldman--Millson relate moduli spaces of flat connections to moduli spaces of representations. If $G$ is the fundamental group of a compact K\"ahler manifold, the comparison between dg-Lie algebras $C^\bullet(G, \End_\C(\rho))$ and the dg-Lie algebra controlling deformations of the flat connection associated to $\rho$ is exploited by Goldman--Millson to prove that these deformations are \emph{formal} (in the sense of Remark \ref{rem: cup case}). This has a consequence that the passage from $C^\bullet(G, \End_\C(\rho))$ to the graded Lie algebra $H^\bullet(G, \End_\C(\rho))$ loses no information and the presentation for the deformation space as in Theorem \ref{thm: main A-inf} is quadratic. 

\subsection{Acknowledgements}
\label{subsec: acknow}

I would like to thank Akshay Venkatesh for his interest and encouragement in this project from its early stages. I would also like to thank the students of Math 202a in Spring 2014 at Brandeis University for their interest as they heard an initial version of the ideas of Part 2. It is also a pleasure to thank Preston Wake for his collaboration, which helped to hone the ideas of this paper through our related joint works \cite{WWE3, WWE4}. 

It is a pleasure to thank Mark Behrens, John Francis, S{\o}ren Galatius, and Kirsten Wickelgren for helpful conversations about homotopy algebras. Likewise, I thank Bill Goldman and Danny Ruberman for providing perspective on the Kuranishi picture and Ed Segal for comments about non-commutative geometry. I would also like to thank Jo\"el Bella\"iche, Ga\"etan Chenevier, Michael Harris, Romyar Sharifi, and Adam Topaz for helpful conversations and comments. 

During work on this paper, the author was supported by an AMS-Simons Travel Grant, Engineering and Physical Sciences Research Council grant EP/L025485/1, and research support from Brandeis University and from Imperial College London's Mathematics platform grant. 

\subsection{Notation and terminology}
\label{subsec: notation}

$G$ denotes a profinite group. $\bF$ denotes a finite field of characteristic $p$, in which $\F[G]$ has the standard profinite topology with completion $\F\lb G\rb$. The topology of the codomain of functions with domain $G$ or $\F[G]$ is always from a presentation as a finitely generated (left) module over a topological $\F$-algebra. These topological $\F$-algebras are
\begin{itemize}[leftmargin=2em]
	\item $\cA_\F$, the category of Artinian local associative $\F$-algebras with residue field $\F$, equipped with the discrete topology;
	\item $\cC_\F$, the full subcategory of $\cA_\F$ consisting of commutative objects;
	\item the categories of limits $\hat\cA_\F$ and $\hat \cC_\F$, with the resulting profinite topology;
	\item $\Aff_\F$, the category of topologically finitely generated $\F$-algebras; this is the opposite category to the category of Noetherian affine $\Spf \F$-formal schemes (see \cite[\S10.1]{ega1}).
\end{itemize}

We let $E$ denote an associative $\F$-algebra. Often we consider the case $E = \F[G]$ or variants of this, in which case $E$ is topological as discussed above. 

We let $(V, \rho)$ denote a finite-dimensional representation of $E$, that is, a finite-dimensional $\bF$-vector space $V$ with a left $\F$-linear action $\rho$ of $E$. We write $\End_\F(\rho)$ for the adjoint representation of $\rho$, which is an $E$-bimodule. (We write ``$R$-bimodule'' as shorthand for ``$(R,R)$-bimodule.'') In contrast, we use $\End_\F(V)$ for the same $\F$-vector space, but in this case emphasizing its $\F$-algebra structure which receives the homomorphism $\rho$. The difference is only a matter of emphasis. We also write ``$\rho$'' when we identify $V \cong \F^{\oplus d}$, writing $\rho$ as a homomorphism
\[
\rho: E \lra M_d(\F) \quad \text{or} \quad \rho: G \lra \GL_d(\F)
\]
in this case. 

Write $\F[\epsilon_n]$ for $\F[\epsilon]/(\epsilon^{n+1}) \in \cC_\F$. Given a homomorphism $\rho : G \ra \GL_d(\bF)$, we use the term ``$n$th-order \emph{lift}'' to describe a homomorphism $G \ra \GL_d(\bF[\epsilon]/\epsilon^{n+1})$ that reduces to $\rho$ modulo $\epsilon$. In contrast, an ``$n$th-order \emph{deformation}'' refers to an orbit of lifts under the adjoint action of $\ker(\GL_d(\bF[\epsilon]/\epsilon^{n+1})) \ra \GL_d(\bF)$. We use the same terms for lifts of representations of $E$.

The term \emph{$\F^r$-algebra} refers to an algebra in the category of $\F^r$-bimodules, where $r \in \bZ_{\geq 1}$. In particular, this does \emph{not} refer to algebras receiving a map from $\F^r$ to the center; see \S\ref{subsec: r-pointed data} for more on this. We write $\uotimes$ for the tensor product in the category of $\F^r$-bimodules, which is also the tensor product for $\F^r$-algebras. All of this reduces to the usual setting of $\F$-algebras when $r=1$. So we will refer to $\F^r$-algebras for the rest of this introduction to notation. 

Graded objects are graded by $\Z$. Derivations and differentials on a graded object have degree $+1$. A complex is a graded object with a differential: we generally use the notation $(C,d_C)$ for a complex, where $d_C^i : C^i \ra C^{i+1}$. We may denote it by $C$ when the context is clear. Likewise, we may denote the cocycle and coboundary subobjects as $C^i \supset Z^i \supset B^i$ when the context $(C,d_C)$ is understood. Suspension ``$\Sigma$'' on a complex produces the complex where $\Sigma C^i = C^{i+1}$ and $d_{\Sigma C} = -d_C$. 

We write \emph{dg-algebra} for a differential graded algebra. These are complexes equipped with an associative graded multiplication satisfying the Leibniz rule, and are denoted $(C,d_C, m_{2,C})$, or ``$C$'' for short. These are most often dg-$\F^r$-algebras. 

We refer to \emph{augmented} dg-$\F^r$-algebras $C$, meaning that there is a augmentation map $C \rsurj \F^r$. A \emph{complete} dg-$\F^r$-algebra is an augmented dg-$\F^r$-algebra that is complete with respect to the kernel of the augmentation map. A \emph{free} complete algebra (resp.\ graded algebra, resp.\ dg-algebra) on a (resp.\ graded, resp.\ dg) $\F^r$-bimodule $V$ is the (resp.\ graded, resp.\ dg) tensor algebra
\[
\hat T_{\F^r} V := \prod_{n \geq 0} V^{\uotimes n}
\]
(resp.\ equipped with the differential produced by extension via the Leibniz rule). 

We use the following notation for categories of dg-algebras.
\begin{itemize}[leftmargin=2em]
	\item $\cA_{\F^r}^\mathrm{dg}$ finite-dimensional augmented dg-$\F^r$-algebras
	\item $\hat \cA_{\F^r}^\mathrm{dg}$ limits of finite-dimensional augmented dg-$\F^r$-algebras, such as $\hat T_{\F^r}^\mathrm{dg}$ when $V$ has finite dimension as an $\F$-vector space. These are complete. 
\end{itemize}

Undecorated tensor products ``$\otimes$'' are assumed to be over $\F$. Likewise, we use ``$(-)^*$'' to denote the $\F$-linear dual of a $\F$-linear object or morphism. This applies naturally to objects with $\F^r$-bimodule structure as well. On a graded $\F$-vector space $C$, it is applied graded-piecewise by default: $(C^*)^i := (C^{-i})^*$. We refer to this as a \emph{graded-dual} object when we wish to emphasize this. The duality operation on a complexes produces a complex $(C^*, d_{C^*})$ where $d_{C^*}^i := (d_C^{-i-1})^*$. As is standard when working with tensor products of morphisms of graded vector spaces, the Koszul sign rule 
\[
(f \otimes f')(x \otimes x') = (-1)^{|f'| |x|} f(x) \otimes f'(x')
\]
is in force.

Coalgebras inherit all of the notions above: codifferential, coaugmentation, cocomplete, cofree, etc., in the standard way. 

Given a graded $\F$-vector space $C$, we will often refer to the graded vector space denoted ``$\Sigma C^*$.'' This is shorthand for $(\Sigma C)^*$: it is the graded-dual of the suspension. 

\part{$A_\infty$-algebras and deformation theory}

In Part 2, we develop the theory of $A_\infty$-algebras, use them to produce presentations of non-commutative deformation rings, and discuss their relationship with Massey products. 

\section{$A_\infty$-algebras}
\label{sec:A-inf}

In this section we recall the definition of an $A_\infty$-algebra and discuss the relationship between $A_\infty$-algebras and dg-algebras. We also discuss the bar equivalence between $A_\infty$-algebras and cocomplete cofree dg-coalgebras, which will provide a useful perspective on $A_\infty$-algebras in the sequel. 

\subsection{Defining $A_\infty$-algebras}
\label{subsec: A-inf defn}

Recall that undecorated tensor products ``$\otimes$'' are over the base field $\F$. 

\begin{defn}
\label{defn:A-infinity}
An \emph{$A_\infty$-algebra} over $\F$ is a pair $(A, (m_n)_{n \geq 1})$ consisting of a graded $\F$-vector space $A = \bigoplus_{i \in \bZ} A^i$ and a sequence of homogenous degree $2-n$ maps
\[
m_n: A^{\otimes n} \lra A, \quad n \geq 1
\]
such that 
\begin{equation}
\label{eq:m_conds}
\sum (-1)^{r+st} m_u ( 1^{\otimes r} \otimes m_s \otimes 1^{\otimes t}) = 0,
\end{equation}
where the sum ranges over all decompositions $n = r+s+t$ into non-negative integers, with the conventions that $u = r+1+t$ and  $m_0 = 0$. We also use $m$ as shorthand for $(m_n)_{n \geq 1}$. 

A \emph{morphism} $f: A \ra A'$ of $A_\infty$-algebras $(A, m)$, $(A', m')$ is a sequence of maps
\[
f_n : A^{\otimes n} \lra A', \quad n \geq 1
\]
of homogenous degree $1-n$ such that, for all $n \geq 1$, 
\begin{equation}
\label{eq:f_conds}
\sum (-1)^{r+st} f_u(1^{\otimes r} \otimes m_s \otimes 1^{\otimes t}) = \sum (-1)^s m'_r(f_{i_1} \otimes f_{i_2} \otimes \dotsm \otimes f_{i_r})
\end{equation}
where the first sum runs over all decompositions $n = r+s+t$ into non-negative integers, we maintain $u = r+1+t$, and the second sum runs over all $1 \leq r \leq n$ and all decompositions $n = i_1 + \dotsm + i_r$ into positive integers, where $s = \sum_{j=1}^{r-1} j (i_j-1)$. 

In particular, the relations above imply that $m_1^2 = 0$, i.e.\ $(A, m_1)$ is a complex. They also imply that $f_1$ is a morphism of complexes $(A,m_1) \ra (A', m'_1)$. A morphism $f$ of $A_\infty$-algebras is called a \emph{quasi-isomorphism} when $f_1$ is a quasi-isomorphism of complexes. It is called an \emph{isomorphism} when $f_1$ is an isomorphism of complexes; indeed, $f$ is an isomorphism if and only if it has an inverse.
\end{defn}

An $A_\infty$-structure records homotopies that make $m_2$ ``associative up to homotopy,'' in contrast to a dg-algebras, which are (strictly) associative. 

\begin{eg}
\label{eg:A-inf}
The following examples illustrate the relationship between this definition and differential graded algebras. Let $(A, m)$ be an $A_\infty$-algebra.

\begin{enumerate}[leftmargin=2em]
\item We have noted that \eqref{eq:m_conds} states that $(A, m_1)$ is a complex.
\item For $n=2$, \eqref{eq:m_conds} states that $m_2(m_1 \otimes 1 + 1 \otimes m) = m_1m_2$, the Leibniz rule. 
\item For $n = 3$, \eqref{eq:m_conds} states that 
\[
m_1m_3 + m_3(m_1 \otimes 1 \otimes 1 + 1 \otimes m_1 \otimes 1 + 1 \otimes 1 \otimes m_1)
= m_2(1 \otimes m_2 - m_2 \otimes 1). 
\]
We see that the left hand side vanishes when $m_3 = 0$ and the right hand side is an \emph{associator,} as it measures the failure $m_2$ to be associative. The left hand side is a chain homotopy of maps $A^{\otimes 3} \ra A$ that is the boundary of $m_3$, so this relation expresses that $m_2$ is associative up to homotopy. 
\item The subsequent maps $m_n$ for $m \geq 4$ record associativity up to homotopy among the composition of $n$ elements. 
\item If $(B,d,m_2)$ is a dg-algebra with differential $d: B \ra B$ of degree 1 and multiplication $m_2: B \otimes B \ra B$ (of degree 0), then assigning $m'_1 := d$, $m'_2 := m_2$, and $m'_n = 0$ for $n \geq 3$ results in an $A_\infty$-algebra $(B,m')$. Likewise, when $(A,m)$ is an $A_\infty$-algebra where $m_n = 0$ for $n \geq 3$, then $(A, m_1, m_2)$ is a dg-algebra. 
\end{enumerate}

Now we discuss a morphism $f = (f_n)_{n \geq 1} : (A, m) \ra (A', m')$. 
\begin{enumerate}[resume, leftmargin=2em]
\item We have noted that \eqref{eq:f_conds} states that $f_1m_1 = m'_1f_1$, i.e.\ $f_1$ is a morphism of degree zero of the complexes.
\item For $n = 2$, a rearrangement of \eqref{eq:f_conds} states that 
\[
m'_1f_2 + f_2(m_1 \otimes 1 + 1 \otimes m_1) = f_1 m_2 - m'_2(f_1 \otimes f_1)
\]
as maps $A^{\otimes 2} \ra A'$. That is, the right hand side expresses the failure of $f_1$ to be multiplicative with respect to $m_2$, and $f_2: A^{\otimes 2} \ra A'$ is a chain homotopy whose boundary is this failure. 
\item Consider $A_\infty$-algebra morphisms $f: A \ra A'$ where 
\begin{itemize}[leftmargin=2em]
\item $(A', m')$ is an $A_\infty$-algebra arising from a dg-algebra (i.e.\ $m'_n = 0$ for $n \geq 3$) and 
\item $m_1 = 0$ (making $(A, m)$ a ``minimal'' $A_\infty$-algebra).
\end{itemize}
If one also assumes that $m_i = 0$ for $2 \leq i \leq n-1$, then \eqref{eq:f_conds} states
\begin{equation}
\label{eq:A-inf to Massey}
f_1 m_n = m'_1f_n  + m'_2(f_1 \otimes f_{n-1} - f_2 \otimes f_{n-2} + \dotsm + (-1)^n f_{n-1} \otimes f_1)
\end{equation}
Equation \eqref{eq:A-inf to Massey} will be relevant for comparison with Massey products in \S\ref{subsec: A-inf to Massey}. 
\end{enumerate}
\end{eg}

\begin{rem}
If the reader finds the relations \eqref{eq:m_conds} and \eqref{eq:f_conds} unenlightening, the more compact equivalent formulation of these conditions in terms of the bar complex may be helpful. For this, see \S\ref{subsec: bar}. 
\end{rem}

\subsection{The relationship between dg-algebras and $A_\infty$-algebras}
\label{subsec: dg to A-inf}

Homotopy retracts relate dg-algebras and $A_\infty$-algebras. 

\begin{defn}
\label{defn: HR}
Let $(A,d_A)$, $(C,d_C)$ be complexes. We call $(A,d_A)$ a \emph{homotopy retract} of $(C,d_C)$ when they are equipped with maps
\[
\xymatrix{
C \ar@(dl,ul)^h \ar@<1ex>[r]^p & A \ar@<1ex>[l]^i
}
\]
such that $p$ and $i$ are morphisms of complexes, $h : C \ra \Sigma^{-1} C$ is a morphism of graded vector spaces, $\mathrm{id}_C - ip = d_Ch + hd_C$, and $i$ is a quasi-isomorphism. 
\end{defn}

The following example of a homotopy retract will be used extensively. 
\begin{eg}
\label{eg:H-sections}
Let $(C,d_C)$ be a cochain complex. Then $(H^\bullet(C), 0)$ is a homotopy retract of $(C,d_C)$, when it is equipped with maps $i, p$ set up in the following way. For each $n$, choose a section $i^n: H^n(C) \ra C^n$ of the standard map from $\ker(d_C\vert_{C^n})$ to $H^n(C)$ and a section $h^n : d_C(C^n) \ra C^n$ of $d_C\vert_{C^n} : C^n \ra C^{n+1}$. These sections produce a decomposition $C^n \risom \im(d_C\vert_{C^{n-1}}) \oplus i^n(H^n(C)) \oplus h^n(d_C(C^n))$, which we write as 
\begin{equation}
\label{eq:LPWZ decomp}
C^n = B^n(C) \oplus \tilde H^n(C) \oplus L^n(C) = B^n \oplus \tilde H^n \oplus L^n.
\end{equation}
Think of the three summands as 
\begin{itemize}
\item co\textit{B}oundaries, 
\item cocycles lifting $H^n(C)$, and 
\item \textit{L}ifts of coboundaries to a cochain inducing it. 
\end{itemize}
Then let $p^n : C^n \ra H^n(C)$ be the projection killing the summands complementing $\tilde H^n$. Combine the above data into maps of complexes $p$ and $i$. Finally, extend $h^n$ to $C^{n+1}$ by killing the summands complementing $B^{n+1}$, so that the $h^n$ together produce a map $h : C \ra \Sigma^{-1} C$ of graded vector spaces. By design, $p$, $i$, and $h$ make $(H^*(C), 0)$ a homotopy retract of $(C,d_C)$. 

There is also a complimentary complex to $H^\bullet(C)$, which we call $K = K^\bullet(C)$. It consists of
\begin{equation}
\label{eq: acyclic K}
K^n = B^n \oplus L^n
\end{equation}
with the restricted differential $d_K = d_C\vert_K$. It is acyclic and contractible, decomposing into $L^n \risom B^{n+1}$ plus the zero map out of $B^n$. 
\end{eg}

When $(A,d_A)$ is a homotopy retract of $(C,d_C)$, there is an isomorphism $H^\bullet(A) \risom H^\bullet(C)$ induced by $i$. Consequently, when $(C,d_C)$ admits the extra structure of a dg-algebra $(C,d_C, m_C)$, there is an induced graded algebra structure on $H^\bullet(A)$. Moreover, there is a natural lift of this graded algebra structure to a graded linear map $A \otimes A \ra A$, namely, 
\[
m_2(x \otimes y) := pm_C(i(x) \otimes i(y)), \quad \text{ for } x,y \in A. 
\]
It satisfies the Leibniz rule with respect to $d_A$ but is not necessarily associative. The $m_2$ map is, however, associative up to a homotopy that can be expressed in terms of the homotopy retract structure. This homotopy is a map $m_3 : A^{\otimes 3} \ra A$ of homogenous degree $-1$ that satisfies the relation \eqref{eq:m_conds} for $n=3$, see \cite[Lem.\ 9.4.2]{LV2012}. Iterating this procedure yields the following result. 
\begin{thm}[{Kontsevich--Soibelman \cite{KS2000}}]
\label{thm:KS}
Let $(C, (m'_n))$ be an $A_\infty$-algebra such that the complex $(C,m'_1)$ is a homotopy retract of $(A, d_A)$ via
\[
%\xymatrix{
%(C, m'_1)  \ar@<1ex>[r]^p \ar@(dl,ul)^h & (A,d_A) \ar@<1ex>[l]^i 
%}
\xymatrix{
*[r]{\;(C, m'_1)\;} \ar@<0.5ex>[rr]^{p} \ar@(dl,ul)[]^{h} 
  && *[l]{\;(A,d_A)\,} \ar@<0.5ex>[ll]^{i} 
}
\]
The retract determines an $A_\infty$-algebra structure $(m_n)$ on $(A, d_A)$ with $m_1 = d_A$ and a quasi-isomorphism of $A_\infty$-algebras $f = (f_n) : (A, (m_n)) \risom (C, (m'_n))$ such that $f_1 = i$. 
\end{thm}
The equality $f_1 = i$ will be expressed as ``$f = (f_n)_{n \geq 1}$ is a quasi-isomorphism of $A_\infty$-algebras extending the quasi-isomorphism of complexes $i$.''

\begin{proof}
See \cite[Thm.\ 9.4.14]{LV2012} and the references therein. 
\end{proof}

We will especially commonly apply the theorem in the following case. In this statement, recall that any dg-algebra $(C,d_C,m_C)$ is an $A_\infty$-algebra in the sense of Example \ref{eg:A-inf}(4). 
\begin{cor}[{Kadeishvili \cite{kadeishvili1982}}]
\label{cor:A-inf on H}
Let $(C,d_C, m_{2,C})$ be a dg-algebra with graded algebra of cohomology $(H^\bullet(C), 0, \bar m_{2,C})$. Then any choice of homotopy retract 
\[
%\xymatrix{
%(C, d_C)  \ar@(dl,ul) \ar@<1ex>[r]^p &  (H^\bullet(C),0) \ar@<1ex>[l]^i 
%}
\xymatrix{
*[r]{\;(C, d_C)\;} \ar@<0.5ex>[rr]^{p} \ar@(dl,ul)[]^{h} 
  && *[l]{\;(H^\bullet(C),0)\,} \ar@<0.5ex>[ll]^{i} 
}\]
as in Example \ref{eg:H-sections} determines an $A_\infty$-structure $(m_n)_{n \geq 1}$ on $H^\bullet(C)$ and a quasi-isomorphism of $A_\infty$-algebras $f = (f_i)_{i \geq 1} : (H^\bullet(C),(m_n)) \ra (C, d_C, m_{2,C})$ such that 
\begin{itemize}
\item $m_1 = 0$, 
\item $m_2 = \bar m_C$, and 
\item $f$ extends $i$, that is, $f_1 = i$.
\end{itemize}
\end{cor}

\begin{proof}
See \cite[Cor.\ 9.4.8]{LV2012}. We sketch this proof in the following examples, because we require the explicit expression of $f$ in the sequel. 
\end{proof}

\begin{eg}
\label{eg:KFT1}
For $n \geq 3$, let $PBT_n$ be the set of planar binary rooted trees with $n$ leaves, as in \cite[App.\ C]{LV2012}. For each $T \in PBT_n$, label each leaf of $T$ with $i$, each intermediate branch of $T$ by $h$, each vertex by $m_C$, and the root of $T$ by $p$.  With the $j$th leaf corresponding to the $j$th tensor factor of $H^*(C)^{\otimes n}$, each $T \in PBT_n$ determines a homogenous degree $2-n$ map $H^*(C)^{\otimes n} \ra H^*(C)$ (in the sense of e.g.\ \cite[Thm.\ 10.3.8]{LV2012}). We set $m_n$ to be the sum over all $T$, where the summand for $T$ is given sign $(-1)^{s_T+1}$, where $s_T$ is the number of leaves lying above the left of the two branches arriving at the root. 

The construction of $f = (f_n)$ is similar. We already have $f_1 = i$. For $n \geq 2$, we modify the labeling of $PBT_n$ used to produce $m_n$ by labeling the root by $h$ instead of by $p$. As a result, each $T \in PBT_n$ gives rise to a homogenous degree $n-1$ map $H^*(C)^{\otimes n} \ra C$. We assign $f_n$ to the sum of these maps over all $T \in PBT_n$, with the sign convention as above. 
\end{eg}

\begin{eg}
\label{eg:KFT2}
For an explicit formula realizing Kadeishvili's theorem but not referring to $PBT_n$, we reproduce \cite[p.\ 2021]{LPWZ2009}, a formula credited to Merkulov \cite{merkulov1999}. Crucially for these formulas, the homotopy retract data induce an identification of $H^n(C)$ with a subspace $\tilde H^n \subset C^n$ as in \eqref{eq:LPWZ decomp}. We will now define versions of $m_n$ on $C$ instead of $H^\bullet(C)$. We denote these by $\tilde m_n$. Initially, we set $\tilde m_1 = d_C$, $\tilde m_2 = m_C$, as usual. 

For $n \geq 3$ we set define a homogenous degree $2-n$ map $\tilde m_n : C^{\otimes m} \ra C$ by 
\begin{equation}
\label{eq:tilde m}
\tilde m_n = \sum_{s+t=n; s,t \geq 1} (-1)^{s+1} m_C ((h\circ \tilde m_s) \otimes (h\circ \tilde m_t)), 
\end{equation}
where we formally define $h\circ \tilde m_1$ to be $-\mathrm{id}_C$. 

Now, for $n \geq 1$, we may define $m_n : H^\bullet(C)^{\otimes n} \ra H^\bullet(C)$ as $p \circ \tilde m_n \circ i^{\otimes n}$, where $i$ stands in for the inclusion $\tilde H \rinj C$. Likewise, let $f_n := -(h \circ \tilde m_n) \circ i^{\otimes n}$, resulting in $f_1 = i$. 
\end{eg}

One may check by examining diagrams in $PBT_n$ that the formulations of $m_n$ (resp.\ $f_n$) in Examples \ref{eg:KFT1} and \ref{eg:KFT2} are equal. 

\begin{defn}
Let $(A, m)$ be an $A_\infty$-algebra. It is called \emph{minimal} provided that $m_1 = 0$. A minimal $A_\infty$-algebra equipped with a quasi-isomorphism from $A$ is called a \emph{minimal model} of $A$. 
\end{defn}

The $A_\infty$-algebras $(A, m)$ produced in Examples \ref{eg:KFT1} and \ref{eg:KFT2} are minimal $A_\infty$-algebras. The morphism $f$ produced in each example is the minimal model, as guaranteed by Corollary \ref{cor:A-inf on H}. 

\subsection{Relationship with the minimal model}

In this section, we explain the extent to which the minimal model induced by homotopy retracts is unique. We also decompose an $A_\infty$-algebra into the sum of its minimal model and a linearly contractible factor. To begin,  we record this useful fact. 
\begin{lem}
\label{lem: p extends}
In the setting of Theorem \ref{thm:KS}, the projection map $p$ extends to a quasi-isomorphism of $A_\infty$-algebras. That is, there exists a quasi-isomorphism $g = (g_n)_{n \geq 1}: (C, m') \ra (A, m)$ such that $g_1 = p$. 
\end{lem}

\begin{proof}
See \cite[Thm.\ 3.9(2)]{CL2017}.
\end{proof}

This is the uniqueness statement we will use. 
\begin{prop}
\label{prop: uniqueness over retract}
Let $(C,d_C, m_{2,C})$ be a dg-algebra with graded algebra of cohomology $(H^\bullet(C), 0, \bar m_{2,C})$. Let $(i,p), (i',p')$ be two choices of homotopy retract as in Corollary \ref{cor:A-inf on H}, resulting in two pairs 
\[
(f, m)  \text{ and } (f', m').
\]
Then the morphism of $A_\infty$-algebras $h := g \circ f$ (where $g$ is as in Lemma \ref{lem: p extends}) 
\[
h = (h_n)_{n \geq 1}: (H^\bullet(C), m) \lrisom (H^\bullet(C), m'). 
\]
is an isomorphism, and $h_1$ is the identity map on the complex $H^\bullet(C)$. 
\end{prop}

\begin{proof}
The claim follows from Lemma \ref{lem: p extends} and Corollary \ref{cor:A-inf on H}. 
\end{proof}

To prepare for the statement of the following theorem, recall that the homotopy retract on $(C,d_C)$ given in Example \ref{eg:H-sections} produced the linearly contractible subcomplex $(K, d_K) \subset (C,d_C)$ defined in \eqref{eq: acyclic K}. The \emph{trivial $A_\infty$-structure} on $K$, which we also denote by $(K,d_K)$, consists of $m_1 = d_K$ along with $m_n = 0$ for $n \geq 2$. 

\begin{thm}[Decomposition theorem]
\label{thm: decomp}
Let $(C, m')$ be an $A_\infty$-algebra. Let the cohomology $(H=H^\bullet(C), 0)$ of the complex $(C,m'_1)$ be given the structure of a homotopy retract via
\[
%\xymatrix{
%(C, m'_1)  \ar@<1ex>[r]^p \ar@(dl,ul)^h & (A,d_A) \ar@<1ex>[l]^i 
%}
\xymatrix{
*[r]{\;(C, m'_1)\;} \ar@<0.5ex>[rr]^{p} \ar@(dl,ul)[]^{h} 
  && *[l]{\;(H,0)\,} \ar@<0.5ex>[ll]^{i} 
}
\]
in the form of Example \ref{eg:H-sections}. Let $(H, m)$ be the minimal $A_\infty$-algebra structure on $H$ induced in Theorem \ref{thm:KS} by this data, and let $(K, d_K)$ represent a trivial $A_\infty$-algebra. 

Then there exists an isomorphism of $A_\infty$-algebras
\[
\chi: (C, m') \lrisom (H, m) \oplus (K, d_K)
\]
such that 
\begin{enumerate}[label=(\roman*), leftmargin=2em]
\item $\chi_1 : C \risom H \oplus K$ is the identity isomorphism of complexes.
\item The projection of $\chi$ onto $H$ is equal to the $A_\infty$-quasi-isomorphism $g : C \risom H$ of Lemma \ref{lem: p extends}
\item The restriction to $H$ of the inverse isomorphism $\chi^{-1} : H \oplus K \ra C$ is equal to the $A_\infty$-quasi-isomorphism $f : H \risom C$ of Theorem \ref{thm:KS}. 
\end{enumerate}
\end{thm}

\begin{proof}
This is the case of \cite[Thm.\ 3.14]{CL2017} where the differential on $H$ is trivial. Moreover, an explicit formula for $\chi$ is given there. 
\end{proof}

\subsection{The bar equivalence}
\label{subsec: bar}

In this section, we set up a dual formulation of $A_\infty$-algebras. 

Recall that $\Sigma$ denotes the suspension operation on complexes $(A, d_A)$, so $(\Sigma A)^i = A^{i+1}$. We write $s$ for the canonical homogeneous degree $-1$ map $s : A \ra \Sigma A$ and $\omega:= s^{-1}$ for its inverse. Recall that $d_{\Sigma A} = -s \circ d_A \circ \omega$. 

\begin{defn}[Bar construction]
	\label{defn: bar}
	Let $A, A'$ be graded vector spaces. For $n \geq 1$, let $m_n : A^{\otimes n} \ra A$ have homogeneous degree $2-n$. Given this data, the \emph{bar construction} consists of 
	\begin{itemize}[leftmargin=2em]
		\item The homogenous degree $1$ map $b_n : (\Sigma A)^{\otimes n} \ra \Sigma A$ given by $b_n := -s \circ m_n \circ \omega^{\otimes n}$, for $n \geq 1$,
		\item The coderivation $b = b_A$ on the cofree cocomplete coalgebra 
		\[
		\hat T^\mathrm{co}_\F (\Sigma A) := \bigoplus_{i\geq 0} (\Sigma A)^{\otimes i}
		\]
		determined by extending $\bigoplus_{n \geq 1} b_n: \hat T^\mathrm{co} (\Sigma A) \ra \Sigma A$ to a coderivation via the co-Leibniz rule.
	\end{itemize}
	
	We write $\B(A)$ or $\B(A, m)$ for the data $(\hat T^\mathrm{co}_\F (\Sigma A), b_A)$.
	
	Likewise, we define the analogue of the bar construction for morphisms. Given the data $f = (f_n)_{n \geq 1}$ where $f_n : A^{\otimes n} \ra A'$ has homogenous degree $1-n$, we define
	\begin{itemize}[leftmargin=2em]
		\item the homogeneous degree $0$ map $g_n:= (\Sigma A)^{\otimes n} \ra \Sigma A'$ given by $g_n = s \circ f_i \circ \omega^{\otimes n}$, for $n \geq 1$
		\item the morphism $g^*$ of free graded cocomplete coalgebras $\hat T^\mathrm{co}_\F (\Sigma A) \ra \hat T^\mathrm{co} (\Sigma A')$ determined by extending $\bigoplus_{n \geq 1} g_n : \hat T^\mathrm{co}_\F (\Sigma A) \ra \Sigma A'$ co-multiplicatively. 
	\end{itemize}
	
	We write $\B(f): \B(A') \ra \B(A)$  for $g^*$. 
\end{defn}

Note that we have not imposed any conditions on the graded linear maps $m_n$ (resp.
 $f_n$) in the construction above. The following theorem explains exactly when they satisfy the $A_\infty$-compatibility conditions making $(A,m)$ an $A_\infty$-algebra (resp.\ making $f$ a morphism of $A_\infty$-algebras). 

\begin{thm}
\label{thm: bar}
Let $A, A'$ be graded vector spaces. For $n \geq 1$, let $m_n : A^{\otimes n} \ra A$ and $m'_n : A'^{\otimes n} \ra A'$ (resp.\ $f_n : A^{\otimes n} \ra A'$) be linear of homogeneous degree $2-n$ (resp.\ $1-n$). 
\begin{enumerate}[leftmargin=2em]
\item $(A, (m_n)_{n\geq 1})$ is an $A_\infty$-algebra if and only if the bar construction $(\hat T^\mathrm{co}_\F (\Sigma A), b)$ is a dg-coalgebra, i.e.\ $b^2 = 0$. 
\item Assuming that $(A, (m_n)_{n \geq 1})$ and $(A', m'_n)_{n \geq 1})$ are $A_\infty$-algebras,  $f := (f_n)_{n \geq 1}$ is a morphism of $A_\infty$-algebras $f : A \ra A'$ if and only if $\B(f) : \hat T^\mathrm{co}_\F (\Sigma A) \ra \hat T^\mathrm{co}_\F (\Sigma A')$ is a morphism of dg-coalgebras (i.e.\ $\B(f)$ commutes with the codifferentials $b^*_A$ on $\hat T^\mathrm{co}_\F (\Sigma A)$ and $b^*_{A'}$ on $\hat T^\mathrm{co}_\F (\Sigma A')$. 
\end{enumerate}
\end{thm}

\begin{proof}
See e.g.\  \cite[Lem.\ 9.2.2 and \S9.2.11]{LV2012}. 
\end{proof}

The following amplification of the theorem is also true. 
\begin{cor}[Bar equivalence]
\label{cor: bar equiv}
The bar construction defines a functor from $A_\infty$-algebras to the full subcategory dg-coalgebras determined by the cofree cocomplete objects. This functor is an isomorphism of categories. 
\end{cor}
\begin{proof}
The first statement follows directly from Theorem \ref{thm: bar}. The second statement follows from the calculations of the bar construction (Definition \ref{defn: bar}): we see that they do not lose any information and that a natural inverse to the bar construction exists. 
\end{proof}

When the graded factors $A^n \subset A$ are finite-dimensional over $\F$, it amounts to the same thing to work with the dg-algebra that is  $\F$-dual to the dg-coalgebra $\B(A)$. This is the dg-algebra that appears in \S\ref{subsec: bar def}. We will use the notation therein for this complete free dg-algebra: 
\begin{defn}
Let $(A,m)$ be an $A_\infty$-$\F$-algebra. Assume that $A^n$ is finite-dimensional for all $n \in \Z$. Then we write 
\[
\B^*(A, m) := (\hat T_\F \Sigma A^*, m^*, \pi)
\]
for the free complete dg-algebra that is graded-dual to the cofree cocomplete dg-coalgebra $\B(A,m)$. 
\end{defn}

Note that we write $m^*$ for the differential that is dual to the codifferential $b$. 

\subsection{The Maurer--Cartan functor}
\label{subsec: MC functor intro}

We change notation: now $(A,d_A, m_{2,A})$ denotes a test object in the category of dg-algebras, while $(C,d_C, m_{2,C})$ denotes a dg-algebra receiving an $A_\infty$ quasi-isomorphism $f : (H, m) \ra (C,d_C, m_{2,C})$ from a minimal $A_\infty$-algebra $(H,)$. We follow \cite{CL2011}. 

For simplicity, we assume that $A$ is finite-dimensional and complete with maximal ideal $\m_A$, i.e.\ $A \in \cA_\F^\mathrm{dg}$, which suffices for applications. We denote by $A^* = (A^*, d^*_A, m^*_{2,A})$ the natural  dg-coalgebra that is $\F$-dual to $A$, with maximal ideal $\m^*_A$. Recall that $\m_A$ is nilpotent. 

We put an ``$A$-linear'' $A_\infty$-algebra structure $m^A$ on $H \otimes A$, defined by 
\begin{equation}
\label{eq: A-linear m}
m_n^A = \mathrm{mult}_n \otimes m_n : (A \otimes H)^{\otimes n} \cong A^{\otimes n} \otimes H^{\otimes n} \to A \otimes H,
\end{equation}
where $\mathrm{mult}_n$ is the usual $n$-fold multiplication map $A^{\otimes n} \to A$. Likewise, we obtain $b^A_n$. 

\begin{defn}[{Cf.\ \cite[\S7.6]{keller2001}}]
\label{defn:HMC}
Let $(H, m)$ be an $A_\infty$-algebra. Given a dg-algebra $A \in \cA_\F^\mathrm{dg}$, a \emph{Maurer--Cartan element} for $H$ valued in $A$ is some $\xi \in (\m_A \otimes H)^1$ such that
\begin{equation}
\label{eq: HMC}
(d_A \otimes \mathrm{id}_H)(\xi) + \sum_{n =1}^\infty (-1)^{\frac{n(n+1)}{2}}m^A_n(\xi) = 0. 
\end{equation}
Denote the set of such elements by $\mathrm{MC}(H,A)$.
\end{defn}

\begin{rem}
The reader may encounter an alternate definition of the homotopy Maurer--Cartan equation in terms of the $b^A_n$, which has the advantage of lacking signs. Then one must consider Maurer--Cartan elements in the degree zero piece of the suspension. They are $\xi' \in (\m_A \otimes \Sigma H)^0$ such that 
\begin{equation}
\label{eq: HMC shift}
(d_A \otimes \mathrm{id}_{\Sigma H})(\xi') + \sum_{n =1}^\infty b^A_n(\xi') = 0. 
\end{equation}
\end{rem}

\begin{rem}
For applications, we will restrict to \emph{classical} complete algebras $A$, i.e.\ $A \in \cA_\F$. In this case, Maurer--Cartan elements come from $\m_A \otimes \Sigma (H^1)$, and the leftmost term of the equations \eqref{eq: HMC} and \eqref{eq: HMC shift} may be dropped. 
\end{rem}

\begin{rem}
\label{rem: MC classical}
In the case that an $A_\infty$-algebra $C$ arises from a dg-algebra $(C,d_C, m_{2,C})$ and $A \in \cA_\F$ is classical, the Maurer--Cartan equation takes on its classical formulation 
\[
d_C(\xi) + m_{2,C}(\xi \otimes \xi) = 0.
\]
Sometimes the equation of Definition \ref{defn:HMC} is called the \emph{homotopy} Maurer--Cartan equation to distinguish it from this case. 
\end{rem}

\begin{prop}
\label{prop:MC-corep}
Let $(H, m)$ be an $A_\infty$-algebra. The association 
\[
\cA^\mathrm{dg}_\F \ni A  \mapsto \mathrm{MC}(H,A) \in \underline{\text{Sets}}
\]
is functorial in $A$ and corepresentable by the bar construction of $H$; that is, 
\[
\mathrm{MC}(H,A) = \Hom_\text{dg-co}(A^*, \B(H)). 
\]
Moreover, if $H^n$ is finite-dimensional for all $n \in \Z$, then $\mathrm{MC}(H,A)$ is representable by the dual of the bar construction of $H$,  that is, 
\[
\mathrm{MC}(H,A) = \Hom_\text{dg}(\B^*(H), A).
\]
\end{prop}

\begin{proof}
By calculation, e.g.\ \cite[Prop.\ 2.2]{CL2011}. 
\end{proof}

\begin{eg}
\label{eg: fin dim classical hull}
Assume that $H^n$ is finite-dimensional for all $n$. Because of our interest in working with classical coefficient rings, we want to determine the classical complete $\F$-algebra $R$ such that for all classical augmented algebras $A = (A, 0, m_{2,A})$, we have
\[
\Hom_\mathrm{dg}(\B^*(H), A) = \Hom_\mathrm{alg}(R,A).
\]
This $R$ is the \emph{classical hull} of $\B^*(H) = (\hat T (\Sigma H)^*, m^*, \pi)$, which we set up in \S\ref{subsec: classical hull}.
\end{eg}

\subsection{The gauge action}
\label{subsec: gauge}

Here we recall basic facts about the gauge action on Maurer--Cartan sets. 

\begin{rem}
	We emphasize that the following version of the gauge action is well-defined over a field $\F$ of any characteristic. One reason we highlight this is that many accounts of the gauge action are expressed in the setting of Lie algebras, where the gauge action originated. There, the action is best expressed as an exponential and therefore only makes sense in characteristic zero. This action readily generalizes to the associative setting, as we will see here; but only some of the literature about this action has denominator-free expressions. 
\end{rem}

\begin{defn}[Unital associative gauge action]
Let $(C,d_C, m_{2,C})$ be a unital dg-$\F$-algebra. Let $\beta \in C^1$. Then the gauge action of $\gamma \in (C^0)^\times$ on $\beta$ is 
\[
\gamma \cdot \beta := \gamma \beta \gamma^{-1} - d_C(\gamma) \gamma^{-1}. 
\]
\end{defn}

For non-unital dg-algebras, one can use the same formula to produce the following augmented gauge action. 
\begin{defn}[Augmented associative gauge action]
\label{defn: aug gauge action}
Let $(C,d_C, m_{2,C})$ be a non-unital complete dg-$\F$-algebra. Then the gauge action of $\gamma \in C^0$ on $\beta \in  C^1$ is the unital gauge action of $(1-\gamma)^{-1}$, which is 
\[
\gamma \cdot \beta := \beta - (1-\gamma)^{-1}(\beta  \gamma - \gamma \beta + d_C(\gamma)),
\]
where $(1-\gamma)^{-1}$ is interpreted as the standard power series and 1 is interpreted as the identity action on $C$. 
\end{defn}

\begin{prop}
The gauge action in the expression above preserves the Maurer--Cartan subset of $C^1$ (consisting of elements such that $d\beta + \beta^2 = 0$). 
\end{prop}

\begin{proof}
For lack of a reference, we supply the calculation in the unital case, from which the augmented case can be derived. 
Given $\gamma \in (C^0)^\times$ and a Maurer--Cartan element $\beta \in C^1$, we calculate that
\begin{align*}
d(\gamma \cdot \beta)  & = d(\gamma \beta \gamma^{-1} - d(\gamma)\gamma^{-1}) \\
& = d(\gamma)\beta\gamma^{-1} - \gamma d(\beta)\gamma^{-1} - \gamma\beta d(\gamma^{-1}) + d(\gamma) d(\gamma^{-1})
\end{align*}
while
\[
(\gamma \cdot \beta)^2 =  - d(\gamma) \beta \gamma^{-1} + \gamma \beta^2 \gamma^{-1} - \gamma \beta\gamma^{-1} d(\gamma) \gamma^{-1}  + d(\gamma) \gamma^{-1} d(\gamma)\gamma^{-1}. 
\]
Pairing the terms of these two expressions using the written order of the summands, and in light of the formula $d(\gamma^{-1}) = -\gamma^{-1}d(\gamma)\gamma^{-1}$, we see that each pair has sum zero. 
\end{proof}

Recall that $\mathrm{MC}(C, A)$ denotes the Maurer--Cartan elements in $(C \otimes \m_A)^1$, when $(C,m)$ is an $A_\infty$-algebra over $\F$ (perhaps arising from a dg-algebra), $A$ is a dg-$\F$-algebra, and $A$ is complete. The $A_\infty$-algebra structure on $C \otimes \m_A$ is as in \eqref{eq: A-linear m}. We have the Maurer--Cartan equation of \eqref{eq: HMC}. In this setting, we formulate ``strict'' gauge equivalence in analogy with deformation theory: one conjugates by the multiplicative group $1 + M_d(\m_A)$, as opposed to the entirety of $M_d(A)^\times$. 
\begin{defn}[Strict gauge equivalence]
\label{defn: strict gauge}
Let $C$ be a dg-$\F$-algebra and let $A \in \hat\cA_\F^\mathrm{dg}$. We say that $\beta,\beta' \in (C \otimes \m_A)^1$ are \emph{strictly gauge equivalent} when they are in the same orbit under the augmented gauge action of $(C \otimes A)^0$. 

Let $(C, m)$ be an $A_\infty$-algebra, and let $A \in \cA_\F^\mathrm{dg}$. We say that $\beta, \beta' \in (C \otimes \m_A)^1$ are gauge equivalent via $\gamma \in (C \otimes \m_A)^0$ when 
\[
\beta - \beta' = (1 \otimes d_A)\gamma + \sum_{n=1}^\infty
 \sum_{j=1}^n (-1)^j m_n^A(\beta^{\otimes j-1} \otimes \gamma \otimes \beta'^{\otimes i-j})
\]

Denote by $\overline{\mathrm{MC}}(C,A)$ the set of strict equivalence classes of the Maurer--Cartan set $\mathrm{MC}(C,A)$. 
\end{defn}

One may check that the $A_\infty$-version of strict gauge equivalence reduces to the dg-algebra case of Definition \ref{defn: aug gauge action}, provided that the $A_\infty$-algebra arises from a dg-algebra. 

It is shown in \cite[Cor.\ 4.17]{proute2011} that gauge equivalence is an equivalence relation under a ``connectedness'' assumption. The point of this assumption is that certain expressions can be inverted, but since we presently work in the complete case (unlike \emph{loc.\ cit.}), these assumptions can be dropped. 

\begin{rem}
For an expression of an $A_\infty$ strict gauge action with denominators, see \cite[Defn.\ 2.9]{segal2008}. The denominator-free version above comes from \cite[Defn.\ 4.5]{proute2011}, where it is expressed dually in $C$. 
\end{rem}

\section{Non-commutative deformations of a point} 
\label{sec:assoc pt}

The goal of this section is to prove a non-commutative generalization of theorems of \S\ref{subsec: results1} --- the determination of a deformation ring in terms of $A_\infty$-structure --- about deformations of an absolutely irreducible representation.

\subsection{Non-commutative deformation theory}
\label{subsec:adt setup}

We use the conventions of \S\ref{subsec: notation}, which set up the deformation theory of a representation
\[
\rho: E \ra \End_\F(V) 
\]
of an associative $\F$-algebra $E$ with finite dimension $d := \dim_\F V$. We especially use the coefficient categories $\cA_\F$ and $\hat \cA_\F$ described there, consisting of objects $(A, \m_A)$. When a $\F$-basis for $V$ is chosen, we will write $\rho : E \ra M_d(\F)$. 

\begin{defn}
\label{defn: lift assoc}
Let $A \in \cA_\F$. A \emph{lift of $\rho$ over $A$} is a homomorphism $\rho_A : E \ra \End_\F(V) \otimes A$ such that $\rho_A \otimes_A \F = \rho$. 

A \emph{deformation of $\rho$ over $A$} is an equivalence class of lifts $\rho_A : E \ra \End_\F(V) \otimes A$ under the equivalence relation of conjugation (i.e.\ inner automorphism) by an element of $\End_\F(V) \otimes A$ whose reduction modulo $\End_\F(V) \otimes \m_A$ is $\mathrm{id}_V \in \End_\F(V)$ (any such element is a unit). 

We define the \emph{lifting functor of $\rho$} (resp.\ the \emph{deformation functor of $\rho$}), denoted $\Def^{\mathrm{nc},\square}_\rho$ (resp.\ $\Def^\mathrm{nc}_\rho$), as the functor from $\cA_\F$ to the category of sets sending $A$ to the set of lifts (resp.\ deformations) of $\rho$ over $A$. 
\end{defn}

To relate $\Def^{\mathrm{nc},\square}_\rho$ to homological invariants, we introduce the Hochschild cochain complex. 

\begin{defn}
\label{defn: Hochs}
Let $E$ be an associative $\F$-algebra. Let $M$ be an $E$-bimodule. The \emph{Hochschild cochain complex}, denoted $C^\bullet(E,M)$, is determined by the $E$-bimodules 
\[
C^\bullet(E,M) := \bigoplus_{i \geq 0} C^i(E,M), \qquad C^i(E,M) := \Hom_\F(E^{\otimes i}, M). 
\]
The differential $d = d^i : C^i(E,M) \ra C^{i+1}(E,M)$ is determined by 
\begin{align*}
d^i(f)(x_1, \dotsc, x_{i+1}) = &\  x_1 f(x_2, \dotsc, x_{i+1}) + (-1)^{i+1} f(x_1, \dotsc, x_i) x_{i+1} \\
+ &\ \sum_{j = 1}^i (-1)^j f(x_1, \dotsc, x_j x_{j+1}, \dotsc, x_{i+1}).
\end{align*}
It is standard to check that $d^{i+1} \circ d^i = 0$. 

We denote by $H^\bullet(E,M)$ the cohomology graded vector space of $C^\bullet(E,M)$, which is called \emph{Hochschild cohomology}. 
\end{defn}

\begin{lem}
\label{lem: Hochs of algebra}
If $M$ has the structure of an associative $\F$-algebra, then the graded multiplication on $C^\bullet(E,M)$ induced by
\begin{align*}
C^i(E,M) \otimes &\ C^j(E,M) \lra C^{i+j}(E,M) \\
f \otimes g &\mapsto \big[ (x_1, \dotsc, x_{i+j}) \mapsto  f(x_1, \dotsc x_i) \cdot g(x_{i+1}, \dotsc, x_{i+j}) \big] 
\end{align*}
is a dg-algebra. 
\end{lem}
\begin{proof}
This is a standard computation. 
\end{proof}

$\Def^{\mathrm{nc},\square}_\rho$ has the following corepresentability. Indeed, the $E$-bimodule structure on $\End_\F(\rho)$ is the natural one, 
\begin{equation}
\label{eq: EE bimodule structure}
((x,y) \cdot f)(v) = \rho(x) \cdot f(\rho(y) \cdot v) \quad \text{for } x,y \in E, v \in V, f \in \End_\F(\rho). 
\end{equation}
For the statement of the proposition, we interpret $\beta \in C^1(E,\End_\F(\rho)) \otimes \m_A$ as a function $\beta : E \ra \End_\F(\rho) \otimes \m_A$. 
\begin{thm}
\label{thm: lift corep}
We have the dg-$\F$-algebra $C = C^\bullet(E, \End_\F(\rho))$. For $A \in \cA_\F$, there is the following natural bijection between $A$-valued Maurer--Cartan elements for $C$ and lifts of $\rho$ to $A$. That is, $\mathrm{MC}(C,A) \risom \Def^{\mathrm{nc}, \square}_\rho(A)$ via
\[
\xi \mapsto \big(\rho \oplus \xi : E \ra \End_\F(V) \otimes A\big). 
\]
In particular, $\B(C)$ corepresents $\Def^{\mathrm{nc}, \square}_\rho$. 
\end{thm}

\begin{proof}
Using Definition \ref{defn:HMC} and Remark \ref{rem: MC classical}, we calculate that an element 
\[
\xi \in \m_A \otimes C^1(E, \End_\F(\rho)) \cong \Hom_\F(E, \End_\F(\rho) \otimes \m_A)
\]
is Maurer--Cartan if and only if it obeys the relation 
\[
\xi(x_1 \cdot x_2) = \xi(x_1) \cdot \xi(x_2) + \rho(x_1) \cdot \xi(x_2) + \xi(x_1) \cdot \rho(x_2). 
\]
From this, one readily observes that an element of $\Hom_\F(E, \End_\F(V) \otimes \m_A)$ is Maurer--Cartan if and only if 
\[
\rho \oplus \xi : E \ra \End_\F(V) \otimes A \cong (\End_\F(V) \oplus \End_\F(V) \otimes \m_A)
\]
is a $\F$-algebra homomorphism. 

The corepresentability claim follows from Proposition \ref{prop:MC-corep}. 
\end{proof}

We now turn to deformations. We write
\begin{equation}
\label{eq: def Erho}
E^\wedge_\rho := \varprojlim_i E/\ker(\rho)^i.
\end{equation}
First, we notice that the natural map $E \ra E^\wedge_\rho$ is a deformation of $\rho$. Indeed, because all deformations of matrix algebras are known to be trivial (or apply \cite[Thm.\ 1.2]{laudal2002}), we get that $E/\ker(\rho)^i$ is isomorphic to a matrix algebra $M_d(A_i)$ for some $A_i \in \cA_\F$. In the limit, we choose an isomorphism
\begin{equation}
\label{eq: isom lift}
E^\wedge_\rho \simeq M_d(R^\mathrm{nc}_\rho) \simeq \End_\F(V) \otimes R^\mathrm{nc}_\rho
\end{equation}
for a chosen $R^\mathrm{nc}_\rho \in \hat \cA_\F$, well-defined up to inner automorphism. This choice is a lift of $\rho$ to $R^\mathrm{nc}_\rho$. So it is fair to call $E \ra E^\wedge_\rho$ is a deformation of $\rho$ valued in $R^\mathrm{nc}_\rho$. 

The isomorphism \eqref{eq: isom lift} realizes a Morita equivalence between $E^\wedge_\rho$ and $R^\mathrm{nc}_\rho$, as follows. Selecting the idempotent $e^{11} \in M_d(R^\mathrm{nc}_\rho) \simeq E^\wedge_\rho$ (the matrix with $1$ concentrated in the $(1,1)$ coordinate) via isomorphisms above, the Morita equivalence of categories is explicitly given by
\begin{align}
\label{eq: morita}
\begin{split}
E^\wedge_\rho\text{-Mod} &\lrisom R^\mathrm{nc}_\rho\text{-Mod} \\
W &\mapsto e^{11}W \\
V \otimes Y &\mapsfrom Y
\end{split}
\end{align}

Now the representability of $\Def^\mathrm{nc}$ follows from the explicit Morita equivalence. 
\begin{prop}
\label{prop:def rep}
Assume that $\rho : E \ra \End_\F(V)$ is absolutely irreducible and let $R^\mathrm{nc}_\rho \in \hat\cA_\F$ as above. Then $\Def^{\mathrm{nc}}_\rho$ is isomorphic to the inner automorphism quotient of the $\Hom$-functor on $\cA_\F$ of $R^{\mathrm{nc}}_\rho$. That is, we have a functorial isomorphism
\[
\Def^{\mathrm{nc}}_\rho(A) \lrisom \Hom_\F(R^{\mathrm{nc}}_\rho, A)/\sim_A,
\]
where $\sim_A$ indicates the equivalence relation of inner $\F$-algebra automorphisms of $A$. This isomorphism is given by applying $(-) \otimes \End_\F(V)$ to (the domain and codomain of) a representative $\eta : R^\mathrm{nc}_\rho \ra A$ of an element of $\Hom_\F(R^\mathrm{nc}_\rho, A)/\sim_A$. 
\end{prop}

\begin{rem}
We see that when we restrict the test coefficients of $\Def^\mathrm{nc}_\rho$ from $\cA_\F$ to $\cC_\F$, which we call $\Def_\rho$, the abelianization of $R^\mathrm{nc}_\rho$ represents $\Def_\rho$. 
\end{rem}

Finally, we prove that $\Def^\mathrm{nc}_\rho$ is representable by $C$ when taken up to strict gauge equivalence. The salient point is that strict gauge equivalence amounts to strict isomorphism. 
\begin{prop}
\label{prop: MC-gauge of C}
We have the dg-$\F$-algebra $C = C^\bullet(E, \End_\F(\rho))$. For $A \in \cA_\F$, there is the following natural bijection between $A$-valued Maurer--Cartan elements for $C$ and lifts of $\rho$ to $A$. That is, $\overline{\mathrm{MC}}(C,A) \risom \Def^{\mathrm{nc}}_\rho(A)$.
\end{prop}
\begin{proof}
In light of Theorem \ref{thm: lift corep}, it remains to prove that the strict gauge action of $\gamma \in C^0 \otimes \m_A$ on $\xi \in \mathrm{MC}(C,A)$ amounts to conjugation of $\rho \oplus \xi$ by $(1-\gamma)$. Indeed, we calculate the conjugation 
\begin{align*}
(1-\gamma)^{-1}(\rho \oplus \xi)(1-\gamma) &= \sum_{i=0}^\infty \gamma^i(\rho + \xi - \rho\gamma - \xi\gamma) \\
& = \rho + \xi + \sum_{i=0}^\infty \gamma^i(\gamma \xi - \xi \gamma - d\gamma) \\
& = \rho \oplus \xi',
\end{align*}
where 
\[
\xi' := \xi - (1-\gamma)^{-1}([\xi,\gamma] + d\gamma). 
\]
This is exactly the strict case of the augmented gauge action of Definition \ref{defn: aug gauge action}.
\end{proof}

\subsection{Non-commutative deformation theory determined by $A_\infty$-structure}
\label{subsec:adt a-inf}

We aim for an analogue of Theorem \ref{thm: lift corep}, giving a homological expression for $\Def^\mathrm{nc}_\rho$. This analogue will be formulated in terms of an $A_\infty$-structure on Hochschild cohomology. We use the objects defined above, $E$, $\rho$, and $C$, assuming that $\rho$ is absolutely irreducible. Let 
\[
H := H^\bullet(C) = H^\bullet(E, \End_\F(\rho))
\]
denote Hochschild cohomology of $\End_\F(\rho)$. 

We choose a homotopy retract structure on $(H, 0)$ relative to $(C, d_C)$ as in Example \ref{eg:H-sections} and apply the results of \S\ref{sec:A-inf}, producing 
\begin{itemize}[leftmargin=2em]
\item a minimal $A_\infty$-structure on $H = H^\bullet(G, \End_\F(\rho))$, denoted 
\[
(H,m) = (H, (m_n)_{n \geq 2}),
\]
extending its canonical graded algebra structure $m_2$. This comes along with 
\item a quasi-isomorphism $f : H \ra C$ of $A_\infty$-algebras (Corollary \ref{cor:A-inf on H}) and
\item an isomorphism $\chi : C \ra H \oplus K$ of $A_\infty$-algebras, where the projection to $H$ is a left inverse and right quasi-inverse to $f$ (Theorem \ref{thm: decomp}), and $(K,d_K)$ is a trivial $A_\infty$-algebra 
\end{itemize}

We assemble the following two facts. 
\begin{lem}
\label{lem: contractible MC}
For a trivial $A_\infty$-algebra $(K,d_K)$, one has $\overline{\mathrm{MC}}(K,-) = \ast$ and
\[
\mathrm{MC}(K,A) \cong B^1 \otimes \m_A
\]
for $A \in \cA_\F$. For the particular trivial $A_\infty$-algebra $(K,d_K)$ produced from $C = C^\bullet(E, \End_\F(\rho))$, along with a homotopy retract on $C$ given in Example \ref{eg:H-sections}, we have
\[
\mathrm{MC}(K, A) \cong \frac{\End_\F(\rho)}{\mathrm{diag}(\F)} \otimes \m_A,
\]
where $\mathrm{diag}(\F)$ denotes the scalar matrix subfield $\F \rinj \End_\F(\rho)$. 
\end{lem}

\begin{proof}
Firstly, recall that $m_{K,1} : B^i \oplus L^i \ra B^{i+1} \oplus L^{i+1}$ arises by restriction of the differential $d_C^i : C^i \ra C^{i+1}$. It is the zero map on $B^i$ along with the isomorphism $L^i \risom B^{i+1}$. Considering the case $i=1$ and recalling that $m_{K,n} = 0$ for $n \geq 2$, we have the calculation of $\mathrm{MC}(K,A)$. Considering the case $i=0$, we see that the gauge action is a torsor, hence $\overline{\mathrm{MC}}(K,-)$ is a point. The final claim follows from the canonical isomorphism $B^1 \cong C^0/\ker(d_C^0)$, noting that $C^0 \cong \End_\F(\rho)$ and $d_C^0$ kills exactly the scalar matrices. 
\end{proof}

\begin{lem}
\label{lem: connected gauge}
Assume that $H^i$ is finite-dimensional for all $i \in \Z$, so that $\mathrm{MC}(H,-)$ is pro-represented on $\cA_\F$ by the classical hull $R$ of $\B^*(H) \in \hat\cA_\F^\mathrm{dg}$ (see Example \ref{eg: fin dim classical hull}). For $A \in \cA_\F$, this pro-representability maps strict equivalence classes in $\mathrm{MC}(H,A)$ isomorphically onto inner automorphism classes in $\Hom_\F(R, A)$. 
\end{lem}

The key point is that $H$ is \emph{connected}, in the sense that $H^0 \cong \F$, arising from the center of $C$. 

\begin{proof}
Let $\gamma \in H^0 \otimes \m_A$ and let $\beta, \beta' \in H^1 \otimes \m_A$. Because $H^0 \cong \F$ and it arises from the center of $C^0$, the higher multiplications are trivial on $H^0$. That is, when $\gamma' \in H^0$ and $\delta_i \in H^i$ for $1 \leq i < n$, then $m_n(\gamma' \otimes \delta_1 \otimes \dotsm \otimes \delta_{n-1}) = 0$ for $n=1$ and $n \geq 3$, and for any other tensor-permutation of the arguments. Therefore the formula expressing $A_\infty$-strict gauge equivalence (Definition \ref{defn: strict gauge}) between $\beta$ and $\beta'$ via $\gamma$ reduces to 
\[
\beta - \beta' = -\gamma \beta' + \beta \gamma, \quad \text{and thus} \quad \beta(1-\gamma) = (1-\gamma)\beta'
\]
(where we are implicitly using $m_2^A$ as multiplication). This is the relation of conjugation by $(1-\gamma)$, and only the $\m_A$ tensor factor of $H^0 \otimes \m_A$ conjugates non-trivially. 
\end{proof}

With these two lemmas in place, the idea is to use the isomorphism of Maurer--Cartan sets and their compatible gauge relations under the $A_\infty$-isomorphism $\chi : C \risom H \oplus K$. 

\begin{thm}
	\label{thm: a-inf main}
	Let $E,\rho, C, H$ be as above. Choose a homotopy retract structure on $(H, 0)$ relative to $(C, d_C)$ as in Example \ref{eg:H-sections}, which gives the additional data  $(H, m)$, $f$, and $\chi$. For $A \in \cA_\F$, this choice determines isomorphisms 
\[
\overline{\mathrm{MC}}(H,A) \underset{\bar f_*}{\lrisom} \overline{\mathrm{MC}}(C,A) \lrisom \Def^\mathrm{nc}_\rho(A),
\]
and they are functorial in $A$. 
\end{thm}
In words, the theorem states that $\Def^\mathrm{nc}_\rho$ is corepresented by the $A_\infty$-algebra $(H, m)$ up to gauge equivalence. We also see that for Maurer--Cartan elements of $H$ in $A$, gauge equivalence amounts to inner automorphism in $A$. We write $\bar f_*$ in the statement to indicate the map on gauge equivalence classes induced from the map of Maurer--Cartan sets 
\[
f_* : \mathrm{MC}(H,-) \ra \mathrm{MC}(C,-) \cong \Def^{\mathrm{nc}, \square}_\rho(-)
\]
induced by $f : H \ra C$. 

\begin{proof}
By Theorem \ref{thm: decomp}, we have available an $A_\infty$-isomorphism
\[
\chi: (C, d_C, m_{2,C}) \lrisom (H, m) \oplus (K, d_K). 
\]
Clearly this induces isomorphisms of Maurer--Cartan functors 
\[
\mathrm{MC}(C, -) \risom \mathrm{MC}(H \oplus K, -), \quad 
\overline{\mathrm{MC}}(C, -) \risom \overline{\mathrm{MC}}(H \oplus K, -). 
\]
Combining Proposition \ref{prop: MC-gauge of C} with Lemmas \ref{lem: contractible MC} and \ref{lem: connected gauge}, and noting that gauge equivalence decomposes along the decomposition $H \oplus K$, the claim follows immediately. 
\end{proof}

\begin{rem}
We see in the statement of Theorem \ref{thm: a-inf main} an instance of ``the homotopy invariance of the Maurer--Cartan functor'': after gauge equivalence, it is a quasi-isomorphism invariant of $A_\infty$-algebras. This is well-known but rarely stated in arbitrary characteristic. In this generality, it can be derived from Theorem \ref{thm: decomp} from \cite{CL2017}, Lemma \ref{lem: contractible MC}, and a generalization of Lemma \ref{lem: connected gauge} for minimal $A_\infty$-algebras that we do not require here.
\end{rem}

We are interested in amplifying Theorem \ref{thm: a-inf main} to give a cohomological presentation for $R^\mathrm{nc}_\rho$ and explicit formulas for representations associated to elements of $\mathrm{MC}(H,A)$. The data $f, \chi$ induced by the homotopy retract is suited for this. 

Under the assumption that $H^i$ is finite-dimensional for all $i \geq 0$, recall that 
\begin{equation}
\label{eq: def R}
R := \frac{\hat T_\F (\Sigma H^1)^*}{(m^*((\Sigma H^2)^*))} \in \hat\cA_\F
\end{equation}
denote the classical hull of the dg-algebra $\B^*(H)$ set up in Example \ref{eg: fin dim classical hull}. 

\begin{cor}
\label{cor: irred case A-inf}
Let $E,\rho, C, H$ be as above. Choose a homotopy retract structure on $(H, 0)$ relative to $(C, d_C)$ as in Example \ref{eg:H-sections}, which gives the additional data  $(H, m)$, $f$, and $\chi$. Under the additional assumption that $H^i$ is finite-dimensional for all $i \geq 0$, these data 
\begin{enumerate}[leftmargin=2em]
\item determine an isomorphism 
\[
\rho^u: E^\wedge_\rho \lrisom \End_\F(V) \otimes R
\]
given by, for $x \in E$, 
\[
x \mapsto \rho(x) + \sum_{i =1}^\infty (\underline{e} \mapsto (f_i(\underline{e}))(x)) \in \End_\F(V) \otimes R
\]
where $\underline{e}$ is a generic element of $(\Sigma H^1)^{\otimes i} = \Sigma H^1(E, \End_\F(V))^{\otimes i}$. 
\item Upon the additional choice of an idempotent $e^{11} \in E^\wedge_\rho$ used to define $R^\mathrm{nc}_\rho$, $\rho^u$ induces an isomorphism
\[
R^\mathrm{nc}_\rho \lrisom R.
\]
\end{enumerate}
\end{cor}

We give an explanation of the notation $(\underline{e} \mapsto (f_i(\underline{e}))(x))$. By definition of $f = (f_n)_{n \geq 1}$, we find $f_i(\underline{e}) \in C^1(E, \End_\F(V)) = \Hom_\F(E, \End_\F(V))$, which is a function that can be evaluated on $x \in E$. So, altogether, $(\underline{e} \mapsto (f_i(\underline{e}))(x))$ is an element of $(\Sigma H^1(E, \End_\F(V))^*)^{\otimes i} \otimes \End_\F(V)$. This determines an element of $\End_\F(V) \otimes R$ via the surjection $\hat T (\Sigma H^1)^* \rsurj R$ of \eqref{eq: def R}. 

\begin{rem}
Theorem \ref{thm: main A-inf} follows more-or-less directly from Corollary \ref{cor: irred case A-inf}. 
\end{rem}

\begin{proof}
First we produce $\rho^u$ and justify its formula. We claim that $\rho^u$ is the $R$-valued lift of $\rho$ arising from the map $\B(f) : \B(H) \ra \B(C)$. Indeed, recall from Theorem \ref{thm: lift corep} that $\B(C)$ corepresents the lifting functor $\Def^{\mathrm{nc}, \square}_\rho$ on Artinian augmented $\F$-algebras. Then recall that that $R$ is a limit of such algebras, and it is also the classical hull of the dual dg-algebra to $\B(H)$. To prove the claim, note that the sum of $f_i : (H^1)^{\otimes i} \ra C^1$ over $i \geq 1$ determines an element of 
\[
\prod_{i \geq 1} (\Sigma H^1)^* \otimes C^1.
\]
This element reduces to the Maurer--Cartan element $\xi_R$ in $C^1 \otimes \m_R$ arising from $\B(H) \ra \B(C)$. Finally, as in the proof of Theorem \ref{thm: lift corep}, $\xi_R$ is a $\F$-linear map from $E$ to $\End_\F(V) \otimes \m_R$ that determines a homomorphism $\rho \oplus \xi_R : E \ra \End_\F(V) \otimes R$ that appears in the formula in (1). We observe that its codomain is local and complete, so $\rho \oplus \xi_R$ induces the map $\rho^u$. 

Now that we know that $\rho^u$ is a homomorphism, we produce
\begin{equation}
\label{eq: Rnc to R}
e \rho^u e : R^\mathrm{nc}_\rho = eE^\wedge_\rho e \lra R \cong \rho^u(e)(\End_\F(V) \otimes R)\rho^u(e).
\end{equation}
The following collection of facts implies that $e \rho^u e$ is an isomorphism. Firstly, note that $R$ pro-represents $\mathrm{MC}(H,-)$ on $\cA_\F$ by its definition. Next, Theorem \ref{thm: a-inf main} draws an isomorphism $\overline{\mathrm{MC}}(H, -) \risom \Def^\mathrm{nc}_\rho$. Proposition \ref{prop:def rep} shows that $R^\mathrm{nc}_\rho$ represents $\Def^\mathrm{nc}_\rho$ after inner automorphism of the coefficients. Lemma \ref{lem: connected gauge} shows that the projection $\mathrm{MC}(H, -) \rsurj \overline{\mathrm{MC}}(H, -)$ amounts to inner automorphism classes in the coefficients in $\cA_\F$. Therefore $R^\mathrm{nc}_\rho$ and $R$ pro-represent the functors on $\cA_\F$ that are identified via the map $\bar f_*$ of Theorem \ref{thm: a-inf main} up to inner automorphism. Because the formula for $\rho^u$ realizes the map $f_*$ discussed after Theorem \ref{thm: a-inf main}, we see that $e \rho^u e$ is compatible with this isomorphism of functors, up to inner automorphism. Therefore $e \rho^u e$ is itself an isomorphism. 
\end{proof}

\section{Non-commutative deformations of multiple points}
\label{sec:assoc mult pts}

The goal of this section is to generalize the results of \S\ref{sec:assoc pt} to the case where we are deforming multiple points. Stated representation-theoretically, we are deforming a semi-simple representation with distinct simple summands. We will follow the approach of \cite{segal2008} in order to study this problem with multiple-pointed coefficients: see \S1.3 and \S2 of \emph{loc.\ cit}.           

\subsection{Setting up the data for multiple points}
\label{subsec: r-pointed data}

We adapt the notation established at the outset of \S\ref{subsec:adt setup}. Now $\rho : E \ra \End_\F(V)$ is a semi-simple representation on a $\F$-vector space $V$. We write $\rho \cong \bigoplus_{i=1}^r \rho_i$, where $\rho_i : E \ra \End_\F(V_i)$ and we fix an isomorphism $V \risom \bigoplus_{i=1}^r V_i$. We assume that the summands $\rho_i$ are absolutely irreducible and pairwise non-isomorphic. 

As in \cite[\S1.3]{segal2008}, we use coefficient algebras on $r$ points. We write $\F^r$ for the $r$-times product algebra $\F \times \dotsm \times \F$. 

\begin{defn}
\label{defn: r}
For $1 \leq i \leq r$, we write $1_i \in \F^r$ for the element with $1$ concentrated in the $i$th coordinate. For a $\F^r$-bimodule $M$ and $1 \leq i,j \leq r$, we write $M_{ij}$ for $1_i \cdot M \cdot 1_j$, so that $M \cong \bigoplus_{i,j} M_{ij}$. 

Let $\Alg^r_\F$ denote the \emph{category of $\F^r$-algebras}, that is, associative unital algebra objects in the category of $\F^r$-bimodules. We write $\F^r \in \Alg^r_\F$ for the standard $\Alg^r_\F$ structure on the ring $\F^r$. Then we observe that $\F^r$ is a unit for the symmetric monoidal (tensor) product $\uotimes$ in $\Alg_\F^r$ by assigning to $A, A' \in \Alg_\F^r$ the vector space
\[
(A \uotimes A')_{ij} := A_{ij} \otimes A'_{ij},
\] 
with coordinate-wise multiplication. Note that $\uotimes$ is the underlying tensor product in the category of $\F^r$-bimodules. In contrast, given $\F^r$-bimodules $M,N$, we use $M \otimes_{\F^r} N$ to denote the usual tensor product, using the left $\F^r$-module structure on $M$ and the right $\F^r$-module structure on $N$ to produce the $\F^r$-bimodule structure of $M \otimes_{\F^r} N$. Finally, the undecorated symbol ``$\otimes$'' is understood to be over $\F$, as usual. 

An \emph{augmentation} of $A \in \Alg^r_\F$ is a morphism $A \ra \F^r$ in $\Alg^r_\F$. Let $\cA_\F^r$ denote the category of augmented $\F^r$-algebras that have finite $\F$-dimension. We write $\m_A \subset A$ for the augmentation ideal of $A \in \cA_\F^r$, so $A/\m_A \risom \F^r$. 

Given a $\F^r$-bimodule $M$, we have the completed tensor product $\F^r$-algebra
\[
\hat T_{\F^r} M := \prod_{i \geq 0} M^{\otimes_{\F^r} i}. 
\]

We understand $\End_\F(V)$ to be a $\F^r$-algebra by sending $1_i$ to the projection operator from $V$ to $V_i$. We will also use the $\F^r$-subalgebra of $\End_\F(V)$, 
\[
\End_\F^r(V) := \bigoplus_{i=1}^r 1_i \End_\F(V) 1_i \cong \bigoplus_{i=1}^r \End_\F(V_i). 
\]
\end{defn}

\begin{warn}
Note that the present notion of $\F^r$-algebra is not the same as ``an associative ring receiving a homomorphism from $\F^r$ to its center.'' An $\F^r$ algebra does receive a canonical homomorphism from $\F^r$, but it is not central. See \cite[\S1.3]{segal2008} for equivalent formulations of $\Alg^r_\F$. Most useful in the sequel is the following alternate formulation: an associative $\F$-algebra with an ordered complete set of $r$ orthogonal idempotents. 
\end{warn}

\subsection{Dg-algebras, $A_\infty$-algebras, and representability over multiple points} 
\label{subsec: r-pointed versions}

In Theorem \ref{thm: lift corep}, we found a bijection between associative lifts and Maurer--Cartan elements for the Hochschild complex. We have the following $r$-pointed generalizations of the objects. 

\subsubsection{Dg-algebras}
A dg-$\F^r$-algebra amounts to a $\F$-linear dg-category on $r$ objects (labeled by $\{1, \dotsc, r\}$), or, equivalently, additional $r$-pointed structure on a dg-algebra over $\F$. It will suffice to consider the example we are concerned with: morphisms in this category are the Hochschild cochain complexes
\[
\Hom(j, i) := C^\bullet(E, \Hom_\F(\rho_j, \rho_i)) \quad \text{for } 1 \leq i,j \leq r
\]
(the $E$-bimodule structure of $\Hom_\F(\rho_j, \rho_i)$ is just like \eqref{eq: EE bimodule structure}), where the composition of morphisms arises from 
\[
\Hom_\F(\rho_k, \rho_j) \otimes \Hom_\F(\rho_j, \rho_i) \ra \Hom_\F(\rho_k, \rho_i).
\]
This composition is compatible with the Hochschild differential for the same reason as Lemma \ref{lem: Hochs of algebra}. Indeed, it follows from applying the statement of Lemma \ref{lem: Hochs of algebra} (verbatim) to the Hochschild complex of $\End_\F(\rho)$, and then using its $\F^r$-algebra structure to deduce the compatibility for the dg-category. 

 We leave the notion of morphisms to the reader. 
 
\subsubsection{$A_\infty$-algebras}
 Similarly to dg-algebras, we may view an $A_\infty$-$\F^r$-algebra as a $\F$-linear $A_\infty$-category on $r$ objects. That is, $\Hom(j,i)$ is a complex with differential $m_1$, and for $n \geq 2$ and any finite sequence $i_0, \dotsc, i_n$ in $\{1, \dotsc, r\}$, there is a $\F$-linear composition law $m_n$ on
 \begin{equation}
\label{eq: r-pointed m_n}
 m_n : \Hom(i_0, i_1) \otimes \dotsm \otimes \Hom(i_{n-1}, i_n) \lra \Hom(i_0, i_n) \text{ of degree } 2-n. \end{equation}
 The $m=(m_n)_{n \geq 1}$ are required to satisfy the compatibility conditions of \eqref{eq:m_conds}. We will mainly discuss $A_\infty$-$\F^r$-algebra structures on $H^\bullet(E, \End_\F(V))$; namely, 
 \[
 \Hom(j,i) = H^\bullet(E, \Hom_\F(\rho_j, \rho_i)). 
 \]
 The composition $m_n$ on $H^\bullet(E, \End_\F(V))$ is a sum of maps of the form \eqref{eq: r-pointed m_n} by applying the direct sum decomposition $\End_\F(\rho) \cong \bigoplus_{1 \leq i,j\leq r} \End_\F(\rho_j, \rho_i)$. 
 
 We leave the notion of morphisms to the reader. 
\subsubsection{Bar construction}
The bar construction involves taking linear duals and suspensions, all of which naturally respects $\F^r$-structure. The bar equivalence of Corollary \ref{cor: bar equiv} also generalizes, giving an isomorphism of categories between $A_\infty$-$\F^r$-algebras and cofree cocomplete (over $\F^r$, i.e.\ coaugmented over $\F^r$) dg-$\F^r$-coalgebras. 

The cofree cocomplete dg-$\F^r$-coalgebra corresponding to an $A_\infty$-$\F^r$-algebra $(H, m)$ is the data of a codifferential on $\hat T^\mathrm{co}_{\F^r} \Sigma H \cong \bigoplus_{i \geq 0} \Sigma H^{\otimes_{\F^r} i}$. When $H^i$ has finite $\F$-dimension for all $i \in \Z$, then we form the dual dg-algebra
\[
(\hat T_{\F^r} \Sigma H^*, m^*, \pi). 
\]

\subsubsection{Maurer--Cartan functor} 
For an $A_\infty$-$\F^r$-algebra $(H, (m_n)_{n \geq 1})$ and any $A \in \cA^r_\F$, a Maurer--Cartan element for $H$ valued in $A$ is some $\xi \in (\m_A \uotimes \Sigma H)^0$ such that the Maurer--Cartan equation \eqref{eq: HMC} holds. The functor of Maurer--Cartan elements is corepresentable over $\F^r$, in direct analogy to Proposition \ref{prop:MC-corep}. 

\subsubsection{Kadeishvili's theorem and the decomposition theorem}
\label{sssec: r-pointed KT}
Next, in order to discuss deformations, we need an $r$-pointed version of Kadeishvili's theorem (Corollary \ref{cor:A-inf on H}). The key point is that the homotopy retract relating the complex $C$ and its cohomology $H$ should respect $\F^r$-bimodule structure. To emphasize this, we state the $r$-pointed generalization of Definition \ref{defn: HR}.
\begin{defn}
Let $(A,d_A)$, $(C,d_C)$ be complexes of $\F^r$-bimodules. We call $(A,d_A)$ a \emph{$r$-pointed homotopy retract} of $(C,d_C)$ when they are equipped with maps
\[
\xymatrix{
C \ar@(dl,ul)^h \ar@<1ex>[r]^p & A \ar@<1ex>[l]^i
}
\]
such that $p$ and $i$ are morphisms of complexes of $\F^r$-bimodules, $h : C \ra \Sigma^{-1} C$ is a morphism of graded $\F^r$-bimodules, $\mathrm{id}_C - ip = d_Ch + hd_C$, and $i$ is a quasi-isomorphism of complexes of $\F^r$-bimodules. 
\end{defn}

Once this is done, Corollary \ref{cor:A-inf on H} and Theorem \ref{thm: decomp} apply to $A_\infty$-$\F^r$-algebras. 
\begin{eg}
\label{eg: r-pointed retract}
To illustrate this for the dg-$\F^r$-algebra $C = C^\bullet(E, \End_\F(V))$ and its cohomology $H$, the point is that the retract datum 
\[
i: (H, 0) \lra (C,d_C)
\]
must lift cohomology classes in $H^\bullet(E, \Hom_\F(\rho_j, \rho_i))$ to a cocycle in the $\F$-subspace 
\[
C^\bullet(E, \Hom_\F(\rho_j, \rho_i)) \subset C^\bullet(E, \End_\F(V)). 
\]
Just as in Example \ref{eg:H-sections}, this induces a direct sum decomposition of $C^\bullet(E, \Hom_\F(\rho_j, \rho_i))$ into coboundaries, cocycles complementing coboundaries, and cochains complementing cocycles. 

Then one may readily deduce from Example \ref{eg:KFT1} or \ref{eg:KFT2} that the resulting 
\begin{itemize}[leftmargin=2em]
\item $A_\infty$-structure $m$ on $H$,
\item $A_\infty$-quasi-isomorphism $f: H \ra C$, and
\item $A_\infty$-isomorphism $\chi : C \ra H \oplus K$
\end{itemize}
respect $\F^r$-structure, using the formulas for $m$, $f$, and $\chi$. 
\end{eg}

\subsection{Deformation theory of $r$ points}
\label{subsec: r pointed defs}

We begin by setting up an an $r$-pointed version of the 1-pointed lifting functor $\Def^{\mathrm{nc}, \square}_\rho$ and 1-pointed deformation functor $\Def^{\mathrm{nc}}_\rho$ that were defined in Definition \ref{defn: lift assoc}. 

\begin{rem}
\label{rem: r conj options}
We comment on the appropriate notion of $r$-pointed notion of strict equivalence: there are two possible options, and we will show that they are equivalent (Proposition \ref{prop: segal auts}) when applied to the appropriate notions of lift. At the least, conjugation of a lift $\rho_A$ of $\rho$ should preserve the lifting property. The largest subgroup of $(\End_\F(V) \uotimes A)^\times$ that does this is $\F^r + \End_\F(V) \uotimes \m_A$, where $\F^r \rinj \End_\F(V)$ arises from its $\F^r$-algebra structure. Within this subgroup, we can also insist on preserving $\F^r$-structure when conjugating, i.e.\ we demand an inner automorphism of $\End_\F(V) \uotimes A$ as an $\F^r$-algebra. This subgroup is 
\[
\F^r + \End_\F^r(V) \uotimes \m_A \subset \F^r + \End_\F(V) \uotimes \m_A. 
\]
The smaller one is more naturally $r$-pointed. However, we are forced to use the larger relation because $E$ has no natural $r$-pointed structure. 
\end{rem}

\begin{defn}
Let $A \in \cA_\F^r$. A \emph{lift of $\rho$ over $A$} is a $\F$-algebra homomorphism $\rho_A : E \ra \End_\F(V) \uotimes A$ such that $\rho_A \uotimes_A \F^r = \rho$. 

A \emph{deformation of $\rho$ over $A$} is an equivalence class of lifts $\rho_A : E \ra \End_\F(V) \uotimes A$ under the equivalence relation of conjugation by $\F^r + \End_\F(V) \uotimes \m_A$. 

We define the \emph{lifting functor of $\rho$} (resp.\ the \emph{deformation functor of $\rho$}) on $\cA_\F^r$, denoted $\Def^{\mathrm{nc},\square}_\rho$ (resp.\ $\Def^\mathrm{nc}_\rho$), as the functor from $\cA_\F^r$ to the category of sets sending $A$ to the set of lifts (resp.\ deformations) of $\rho$ over $A$. 
\end{defn}

To produce $r$-pointed notions of lift and deformation, we begin with $E \ra E^\wedge_\rho$, defined in \eqref{eq: def Erho}. 

Let $\bar e_i \in \End_\F(V_i)$ be a projection operator onto a 1-dimensional subspace of $V_i$, and let $\bar e = \sum_1^r \bar e_i \in \End_\F^r(V)$. Choose orthogonal idempotent lifts $e_i \in E^\wedge_\rho$ of $\bar e_i$ via $\rho_i : E^\wedge_\rho \rsurj \End_\F(V_i)$. Letting $e = \sum_{i=1}^r e_i$, we get a Morita equivalence between $E^\wedge_\rho$ and 
\[
R^\mathrm{nc,r}_\rho := e E^\wedge_\rho e \in \hat \cA_\F^r,
\]
where the $\F^r$-algebra structure on $R^{\mathrm{nc},r}_\rho$ is determined by $(e_i)_{i=1}^r$. The inverse equivalence on algebras is realized by 
\begin{equation}
\label{eq: r isom lift}
E^\wedge_\rho \simeq (M_{d_i \times d_j}((R^{\mathrm{nc},r}_\rho)_{i,j}))_{i,j} \simeq \End_\F(V) \uotimes R^{\mathrm{nc},r}_\rho;
\end{equation}
for a proof of this, apply \cite[Thm.\ 1.2]{laudal2002} in the limit on $E/\ker(\rho)^i$. 

Now that we have a choice \eqref{eq: r isom lift} of $\F^r$-algebra structure on $E^\wedge_\rho$, we can set up $r$-pointed lifting and deformation functors. 

\begin{defn}
Let $A \in \cA_\F^r$. A \emph{$r$-lift of $\rho$ over $A$} is a $\F^r$-algebra homomorphism $\rho_A : E^\wedge_\rho \ra \End_\F(V) \uotimes A$ such that $\rho_A \uotimes_A \F^r = \rho$. 

A \emph{$r$-deformation of $\rho$ over $A$} is an equivalence class of $r$-lifts $\rho_A : E^\wedge_\rho \ra \End_\F(V) \uotimes A$ under the equivalence relation of conjugation by $\F^r + \End_\F^r(V) \uotimes \m_A$. 

We define the \emph{lifting functor of $\rho$} (resp.\ the \emph{deformation functor of $\rho$}) on $\cA_\F^r$, denoted $\Def^{\mathrm{nc},\square,r}_\rho$ (resp.\ $\Def^{\mathrm{nc},r}_\rho$), as the functor from $\cA_\F^r$ to the category of sets sending $A$ to the set of lifts (resp.\ deformations) of $\rho$ over $A$. 
\end{defn}

While there is clearly a natural proper inclusion of lifting functors on $\cA_\F^r$ 
\[
\Def^{\mathrm{nc}, \square,r}_\rho \rinj \Def^{\mathrm{nc}, \square}_\rho, 
\]
the two notions of deformation are equivalent. 
\begin{prop}[Segal]
\label{prop: segal auts}
There is a natural isomorphism of functors on $\cA_\F^r$ 
\[
\Def^{\mathrm{nc},r}_\rho \lrisom \Def^{\mathrm{nc}}_\rho. 
\]
\end{prop}

\begin{proof}
This is \cite[Prop.\ 2.11 and Lem.\ 2.12]{segal2008}. 
\end{proof}

Having set up the deformation theory, we examine what arises from the natural $r$-pointed structures on the Hochschild cochain dg-$\F^r$-algebra $C = C^\bullet(E, \End_\F(V))$. We immediately find the following $r$-pointed version of Theorem \ref{thm: lift corep}, where the left isomorphism is simply the representability of the Maurer--Cartan functor. 
\begin{thm} 
\label{thm: r lift corep}
Let $A \in \cA^r_\F$. Let  be the Hochschild cochain dg-$\F^r$-algebra. There are canonical isomorphisms 
\begin{equation}
\label{eq:MC lift rep mult}
\Hom_{\F^r\text{-dgca}}(A^\vee, \B(C)) \lrisom \mathrm{MC}_{\F^r}(C,A) \lrisom \Def^{\mathrm{nc}, \square}_\rho(A).
\end{equation}
\end{thm}

Similarly, in analogy to Proposition \ref{prop:def rep} and its proof, we can find a tautological construction representing the deformation functor in terms of $E^\wedge_\rho$ the choice of idempotent made to construct $R^{\mathrm{nc},r}_\rho$ and give $E^\wedge_\rho$ the structure of a $\F^r$-algebra. We apply the Morita equivalence of categories explicitly given by 
\begin{align}
\label{eq: r morita}
\begin{split}
E^\wedge_\rho\text{-Mod} &\lrisom R^{\mathrm{nc},r}_\rho\text{-Mod} \\
W &\mapsto eW \\
V \otimes_{\F^r} Y &\mapsfrom Y
\end{split}
\end{align}
generalizing the case $r=1$ of \eqref{eq: morita}. 

\begin{prop}
\label{prop: r def rep}
$\Def^{\mathrm{nc}, r}_\rho$ is isomorphic to the $\F$-inner automorphism quotient of the $\Hom_{\F^r}$-functor on $\cA_\F^r$ of $R^{\mathrm{nc},r}_\rho$. That is, there is a functorial isomorphism 
\[
\Def^{\mathrm{nc}, r}_\rho(A) \lrisom \Hom_{\F^r}(R^{\mathrm{nc},r}_\rho, A)/\sim_{A,r}
\]
over $A \in \cA_\F^r$, where $\sim_{A,r}$ indicates the equivalence relation of inner $\F^r$-algebra isomorphisms of $A$. This isomorphism is given by applying $(-) \uotimes \End_\F(V)$ to a representative $\eta : R^{\mathrm{nc},r}_\rho \ra A$ of an element of $\Hom_{\F^r}(R^{\mathrm{nc},r}_\rho, A)/\sim_{A,r}$. 
\end{prop}

Finally, we prove that $\Def^\mathrm{nc}_\rho$ is representable by $C$ when taken up to gauge equivalence and $\F^r$-conjugation. The proof is exactly as in the 1-pointed version of Proposition \ref{prop: MC-gauge of C}. The one point of difference is that $\F^r$ is no longer central as it was when $r=1$, so the strict gauge action of conjugation by $1 + \End_\F(\rho) \uotimes \m_A$ must be followed by conjugation by $(\F^r)^\times \cong \Aut_E(\rho)$. 

\begin{prop}
\label{prop: r MC-gauge of C}
We have the dg-$\F^r$-algebra $C = C^\bullet(E, \End_\F(\rho))$. For $A \in \cA_\F^r$, there is the following natural bijection between $A$-valued Maurer--Cartan elements for $C$ and lifts of $\rho$ to $A$. That is, $\overline{\mathrm{MC}}(C,A)/(\F^r)^\times \risom \Def^{\mathrm{nc}}_\rho(A)$.
\end{prop}

\subsection{$A_\infty$-algebras and deformations over multiple points} 

We aim for an analogue of Theorem \ref{thm: r lift corep} expressing $\Def^{\mathrm{nc}, r}_\rho$ in terms of cohomological data. We use the setup for the $r$-pointed Kadeishvili's theorem and decomposition theorem from \S\ref{sssec: r-pointed KT}. After assembling a few more lemmas, we apply very similar arguments to the 1-pointed case to prove the main theorems. 

From above, we have $E$, $\rho$, and $C$. Let 
\[
H := H^\bullet(C) = H^\bullet(E, \End_\F(\rho))
\]
denote Hochschild cohomology of $\End_\F(\rho)$. We choose an $r$-pointed homotopy retract structure on $(H, 0)$ relative to $(C, d_C)$, obtaining all of the decompositions and structures listed in Example \ref{eg: r-pointed retract}. We use all of the notation therein. 

The following $r$-pointed analogues of Lemmas \ref{lem: contractible MC} and \ref{lem: connected gauge} are needed. 

\begin{lem}
\label{lem: r contractible MC}
For a trivial $A_\infty$-algebra $(K,d_K)$, one has $\overline{\mathrm{MC}}(K,-) = \ast$ and
\[
\mathrm{MC}(K,A) \cong B^1 \uotimes \m_A
\]
for $A \in \cA_\F^r$. For the particular trivial $A_\infty$-$\F^r$-algebra $(K,d_K)$ produced from $C = C^\bullet(E, \End_\F(\rho))$, along with a homotopy retract on $C$ given in Example \ref{eg: r-pointed retract}, we have
\[
\mathrm{MC}(K, A) \cong \frac{\End_\F(\rho)}{\mathrm{diag}(\F^r)} \uotimes \m_A,
\]
where $\mathrm{diag}(\F^r)$ denotes the product of scalar matrices in $\End_\F^r(\rho)$. 
\end{lem}
\begin{proof}
Same as that of Lemma \ref{lem: contractible MC}. 
\end{proof}

\begin{lem}
\label{lem: r connected gauge}
Assume that $H^i$ is finite-dimensional for all $i \in \Z$, so that $\mathrm{MC}(H,-)$ is pro-represented on $\cA_\F$ by the classical hull $R$ of $\B^*(H) \in \cA_\F^\mathrm{dg}$ (see Example \ref{eg: fin dim classical hull}). For $A \in \cA_\F$, this pro-representability maps strict equivalence classes in $\mathrm{MC}(H,A)$ isomorphically onto strict $\F^r$-inner automorphism classes in $\Hom_\F(R, A)$. 
\end{lem}

Just as $\F^r$-inner automorphism of $A \in \cA^r_\F$ means conjugation by $\F^r \uotimes A \cong \bigoplus_{i=1}^r A_{i,i}$, strict $\F^r$-inner automorphism refers to conjugation by $1 + \F^r \uotimes \m_A$. 

\begin{proof}
The key point is that $H$ is itself augmented as an $A_\infty$-$\F^r$-algebra, in the sense that $H^0 \cong \F^r$, arising from the center of $C$. The proof then proceeds with the same calculations as in Lemma \ref{lem: connected gauge}. We arrive at conjugation by $1-\gamma$, where 
\[
\gamma \in H^0 \uotimes \m_A \cong \F^r \uotimes \m_A \cong \bigoplus_{i=1}^r (\m_A)_{i,i}. \qedhere
\]
\end{proof}

Now we deduce an $r$-pointed analogue of Theorem \ref{thm: a-inf main}. 

\begin{thm}
	\label{thm: r-pointed a-inf}
Let the data $E$, $\rho$, $C$, $H$ be as above. Choose an $r$-pointed homotopy retract between $H$ and $C$ as in Example \ref{eg: r-pointed retract}, inducing  $m$, $f$, and $\chi$. For any $A \in \cA_\F^r$, these data determine functorial isomorphisms
\[
\overline{\mathrm{MC}}_{\F^r}(H,A)/(\F^r)^\times \underset{\bar f_*}{\lrisom} \overline{\mathrm{MC}}_{\F^r}(C,A)/(\F^r)^\times \lrisom \Def^\mathrm{nc}_\rho(A). 
\]
\end{thm}

As in the 1-pointed version, we write $\bar f_*$ to indicate the map on gauge equivalence classes induced from the map of Maurer--Cartan sets
\[
f_* : \mathrm{MC}(H,-) \ra \mathrm{MC}(C,-) \cong \Def^{\mathrm{nc}, \square}_\rho(-)
\]
induced by $f : H \ra C$. 

\begin{proof}
Analogously to the proof of Theorem \ref{thm: a-inf main}, this follows straightforwardly from the $r$-pointed version of the decomposition Theorem \ref{thm: decomp} discussed in \S\ref{sssec: r-pointed KT}, the conclusions about the gauge action in Lemmas \ref{lem: r contractible MC} and \ref{lem: r connected gauge}, and the rightmost isomorphism of the theorem statement from Proposition \ref{prop: r MC-gauge of C}
\end{proof}

Under the assumption that $H^i$ is has finite $\F$-dimension for all $i \in \Z$, we consider the augmented $\F^r$-algebra 
\begin{equation}
\label{eq: r pointed hull}
R = \frac{\hat T_{\F^r} (\Sigma H^1)^*}{(m^*((\Sigma H^2)^*))},
\end{equation}
which is the classical hull of the augmented dg-$\F^r$-algebra $\B^*(H)$ that is dual to the dg-$\F^r$-coalgebra produced by the bar construction (Fact \ref{fact: bar}). This directly generalizes the 1-pointed expression of \eqref{eq: def R}. 

Then, we prove an explicit relationship between $E^\wedge_\rho$ and $R$. 
\begin{cor}
	\label{cor: r-pointed dual}
We assume the setting of Theorem \ref{thm: r-pointed a-inf}, so that we have the data $E$, $\rho$, $C$, $H$, $m$, $f$, and $\chi$. Under the additional assumption that $H^i$ is finite-dimensional for all $i \geq 0$, these data 
\begin{enumerate}[leftmargin=2em]
\item determine an isomorphism of $\F$-algebras 
\[
\rho^u: E^\wedge_\rho \lrisom \End_\F(V) \uotimes R
\]
given by, for $x \in E$, 
\[
x \mapsto \rho(x) + \sum_{i =1}^\infty (\underline{e} \mapsto (f_i(\underline{e}))(x)) \in \End_\F(V) \otimes R
\]
where $\underline{e}$ is a generic element of $(\Sigma H^1)^{\otimes i} = \Sigma H^1(E, \End_\F(V))^{\otimes i}$. 
\item Upon the additional choice of an idempotents used to give $E^\wedge_\rho$ an $\F^r$-algebra structure and define $R^{\mathrm{nc},r}_\rho$ (see \eqref{eq: r isom lift}), $\rho^u$ induces an isomorphism in $\hat\cA_\F^r$ 
\[
R^\mathrm{nc}_\rho \lrisom R.
\]
\end{enumerate}
\end{cor}

The notation $(\underline{e} \mapsto (f_i(\underline{e}))(\gamma))$ has mostly the same as in Corollary \ref{cor: irred case A-inf}, with the sole exception that it is correct here to use $\uotimes$: $f$ will send $\underline{e}_{i,j}$ to $\Hom_\F(V_j, V_i)$, as we insisted that the homotopy retract respect $\F^r$-structure. 

\begin{rem}
	As discussed in \S\ref{subsec: NCDT}, Corollary \ref{cor: r-pointed dual} refines results of Segal \cite[Thm.\ 2.14]{segal2008} and Laudal \cite{laudal2002}. 
\end{rem}

\begin{proof}
The proof proceeds just as the proof of Corollary \ref{cor: irred case A-inf}. The formula for $\rho^u$ is identical, and respects $\F^r$-structure because the retract structures and (therefore) the $A_\infty$-structures and homomorphisms do so. 

We deduce the isomorphism (2), from which (1) follows. The choice of idempotents yields $e\rho^u e : R^{\mathrm{nc},r}_\rho \ra R$ exactly as in \eqref{eq: Rnc to R}. Theorem \ref{thm: r-pointed a-inf} draws an isomorphism $\overline{\mathrm{MC}}(H, -)/(\F^r)^\times \risom \Def^\mathrm{nc}_\rho$. Proposition \ref{prop: r def rep} shows that $R^{\mathrm{nc},r}_\rho$ represents $\Def^{\mathrm{nc},r}_\rho$ up to  $\F^r$-inner automorphism of the coefficients. By Proposition \ref{prop: segal auts}, there is an isomorphism of functors $\Def^{\mathrm{nc},r}_\rho \risom \Def^\mathrm{nc}_\rho$. Lemma \ref{lem: r connected gauge} shows that the projection $\mathrm{MC}(H, -) \rsurj \overline{\mathrm{MC}}(H, -)$ amounts to $\F^r$-inner automorphism classes in the coefficients in $\cA_\F^r$. Recall that $R$ pro-represents $\mathrm{MC}(H,-)$ on $\cA_\F^r$ by its definition. Putting together these isomorphisms, we deduce that $R^\mathrm{nc}_\rho$ and $R$ pro-represent functors on $\cA_\F^r$ that are isomorphic via the map $\bar f_*$ of Theorem \ref{thm: r-pointed a-inf} (up to $\F^r$-inner automorphism). Because the formula for $\rho^u$ realizes the map $f_*$ discussed after Theorem \ref{thm: r-pointed a-inf}, we see that $e \rho^u e$ is compatible with this isomorphism of functors, up to inner automorphism. Therefore $e \rho^u e$ is itself an isomorphism. 
\end{proof}

\section{Massey products}
\label{sec: massey}

The point of this section is to introduce Massey products, in preparation for the explanation of \S\ref{sec: lifts-massey} of the relationship between lifts of representations and Massey products. Here, we focus on the relationship between Massey products and $A_\infty$-products; mainly, we follow \cite{LPWZ2009}. 

\begin{rem}
Massey products were first introduced in topology by Massey and Massey--Uehara \cite{massey1958,MU1957}. For introductions relatively similar to our approach, see Kraines \cite{kraines1966}, May \cite{may1969}, and Dwyer \cite[\S2]{dwyer1975}. 
\end{rem}

\subsection{Massey products in dg-algebras}

In this section, we depart from the notation of other sections in order to write $C = (C^\bullet ,\partial, \smile)$ for a dg-$\F$-algebra, possibly non-unital as usual. Let $H = H^\bullet(C)$ be its cohomology. A Massey product of degree $n$ is a multi-valued cohomology operation $H^{\otimes n} \ra H$ of cohomological degree $2-n$. They are not always defined: each value arises from a defining system. Presently, we will introduce these notions in detail. 

The second Massey product $\langle \,, \rangle : H^{\otimes 2} \ra H$ is unambiguously and unconditionally defined: is the reduction of $\smile$ modulo coboundaries, i.e.\ the cup product. 

Further Massey products are defined as follows. We establish the notation $\bar \sigma = (-1)^{i+1}\sigma$ for $\sigma \in C^i$ for this general definition. Note that in our main case of interest where $d=1$, we have $\bar \sigma = \sigma$. We also remind the reader that we write $C^i \supset Z^i \supset B^i$ for the cocycle and coboundary subobjects. 

\begin{rem}
There are at least two sign conventions used for Massey products. We follow May \cite{may1969} and \cite{LPWZ2009}, in contrast to \cite{kraines1966} and \cite{LV2012}. 
\end{rem} 

\begin{defn}
	\label{defn:massey_product}
Let $n \geq 3$. Let $I_n$ be the set of pairs of integers $(i,j)$ such that $1 \leq i \leq j \leq n$ and $(i,j) \neq (1,n)$. 

	Let $\sigma_i \in Z^{d_i} \subset C^{d_i}$ be cocycles for $1 \leq i \leq n$. For $(i,j) \in I \cup \{(1,n)\}$, let 
	\[
	d(i,j) = -(i-j) + \sum_{k=i}^j d_i.  
	\]
	We say that a set $\fS=\{\sigma(i,j) \in C^{d(i,j)} : (i,j) \in I\}$ is a \emph{defining system for the $n$th Massey product} $\dia{\sigma_1,\dots, \sigma_n}$ if
	\begin{enumerate}[leftmargin=2em]
		\item $\sigma(i,i)=\sigma_i$ for all $i=1,\dots n$, and
		\item $\displaystyle \partial\sigma(i,j) = \sum_{k=i}^{j-1} \bar\sigma(i,k)\smile \sigma(k+1,j)$ for all $(i,j) \in I_n$ such that $i < j$.
	\end{enumerate}
	When $\fS$ is a defining system for $\dia{\sigma_1,\dots,\sigma_n}$, we note that 
	\[
	c(\fS) := \sum_{k=1}^{n-1} \bar\sigma(1,k)\smile \sigma(k+1,n) 
	\]
	is an element of $Z^{d(1,n)+1}$ and we let $\dia{\sigma_1,\dots, \sigma_n}_\fS \in H^{d(1,n)+1}$ be the class of $c(\fS)$. We let
	\[
	\dia{\sigma_1,\dotsc, \sigma_n} = \{\dia{\sigma_1,\dots, \sigma_n}_\fS  \} \subset H^{d(1,n)+1}
	\]
	where $\fS$ ranges over all defining systems. It may be empty. 
	
	Call $\dia{\sigma_1,\dots, \sigma_n}$ \emph{defined} if it is non-empty (i.e.\ there exists a defining system), and say that $\dia{\sigma_1,\dots, \sigma_n}$ \emph{contains zero} if $0 \in \dia{\sigma_1,\dots, \sigma_n}$.
\end{defn}

It is known that the set $\dia{\sigma_1,\dots, \sigma_n}$ only depends on the cohomology classes of $\sigma_1, \dots, \sigma_n$ \cite[Thm.\ 3]{kraines1966}. Also, note that $d(1,n)+1 = 2-n+ \sum_{k=1}^n d_i$, confirming that the $n$th Massey product has cohomological degree $2-n$ (as a multi-valued map). Finally, note that $\dia{\sigma_1, \dotsc, \sigma_n}$ is defined only if 
\begin{itemize}[leftmargin=2em]
\item all lower-degree Massey products on proper sub-words of $\sigma_1\sigma_2\dotsm\sigma_n$ are defined, and
\item all of these contain zero. 
\end{itemize}
In the sequel, we have $d_i = 1$ for all $i$; therefore all such Massey products are valued in $H^2$. 

\subsection{Massey powers in dg-algebras}
\label{subsec: Massey powers}

Let $C = (C, \partial, \smile)$ continue to represent a dg-algebra. The following (non-standard) notion of \emph{Massey power} will be useful for our applications. Due to our attention to this application, we only discuss Massey powers of elements of $C^1$. 

\begin{defn}
	\label{defn: Massey powers}
	Let $\tau \in Z^1$, and let $\tau_1 := \tau, \tau_2, \dotsc, \tau_{n-1} \in C^1$, where $n \geq 3$. We say that $\fT:=\{\tau_1, \dots, \tau_{n-1}\}$ \emph{is a defining system $\fS$ for the $n$th Massey power} $\dia{\tau}^n$ if the set
	\[
	\fS = \fS(\fT) := \{\sigma(i,j)=\tau_{j-i+1} : 1 \le i \le j \le n, (i,j) \ne (1,n)\}
	\]
	is a defining system for the Massey product $\dia{\tau, \dots, \tau}$ (with $\tau$ repeated $n$ times). A defining system (for the product) arising in this manner is called \emph{symmetric}. If $\fT$ is a defining system for the Massey power $\dia{\tau}^n$, then we let $\dia{\tau}^n_\fT :=\dia{\tau,\dots,\tau}_{\fS(\fT)}$, and we let $c(\fT):=c(\fS)$. We let 
	\[
	\dia{a}^n = \{ \dia{a}^n_\fT\} \subset H^2
	\]
	where $\fT$ ranges over defining systems for the Massey powers. 
\end{defn}

We make the following important observations. 
\begin{itemize}[leftmargin=2em]
\item  $\dia{\tau}^n \subset \dia{\tau, \dots, \tau}$, properly in general. 

\item $\fT=\{\tau_1, \dots, \tau_{n-1}\} \subset C^1$ is a defining system for the Massey power $\dia{\tau}^n$ if and only if $\tau_1=\tau$ and, for all $i=1, \dots, n-1$, we have
\begin{equation}
\label{eq:massey power reln}
d\tau_i = \sum_{j=1}^{i-1} \tau_j \smile \tau_{i-j}.
\end{equation}
\item for such $\fT$, we have
\[
c(\fT) = \sum_{j=1}^{n-1} \tau_j \smile \tau_{n-j}.
\]
\item The cohomology classes of $\dia{a}^n$ do not depend on the choice of $a$ within its cohomology class. 
\end{itemize}

\subsection{The relationship between Massey products and $A_\infty$-products}
\label{subsec: A-inf to Massey}

Recall from Definition \ref{defn:A-infinity} that an $A_\infty$-algebra structure $m$ on a graded $\F$-vector space $H$ consists of maps $m_n : H^{\otimes n} \ra H$ of homogeneous degree $2-n$. Recall also from Corollary \ref{cor:A-inf on H} that when $C = (C^\bullet, \partial, \smile)$ is a dg-$\F$-algebra, then there are various compatible choices of $A_\infty$-algebra structure $m$ on its cohomology $H = H^\bullet(C)$. For example, homotopy retracts between $H$ and $C$ induce such an $m$, according to Example \ref{eg:KFT2}. This suggests a relationship between Massey products and $A_\infty$-products. Following \cite{LPWZ2009}, we discuss the relationship between these two notions. 

\begin{prop}
\label{prop:A-inf vs Massey}
Let $(C, \partial, \smile)$ be a dg-algebra with cohomology $H = H^*(C)$. We fix cohomology classes $a_i\in H^{d_i}$ for $1 \leq i \leq n$, where $n \geq 3$. 

Choose in addition the data of a homotopy retract
\[
\xymatrix{
*[r]{\;(C, d_C)\;} \ar@<0.5ex>[rr]^{p} \ar@(dl,ul)[]^{h} 
  && *[l]{\;(H,0)\,} \ar@<0.5ex>[ll]^{f} 
}
\]
as in Example \ref{eg:H-sections}, specifying an $A_\infty$-algebra $(H, (m_n)_{n\geq 2})$ and a quasi-isomorphism $f : (C, \partial, \smile) \ra (H, (m_n)_{n \geq 2})$ as in Example \ref{eg:KFT1}. 
\begin{enumerate}[label=(\alph*), leftmargin=2em]
\item Assume that for all $i$, $2 \leq i \leq n-1$ and all sub-$i$-tuples $(a_j, \dotsc, a_{j+i-1})$ of $(a_1, \dotsc, a_n)$, it is the case that $m_i(a_j \otimes \dotsm \otimes a_{j+i-1}) = 0$. 
\item Define 
\begin{equation}
\label{eq:massey to A-inf}
a(i,j) = f_{j-i+1}(a_i \otimes \dotsm \otimes a_j).
\end{equation}
\end{enumerate}
Then $(-1)^b m_n(a_1 \otimes \dotsm \otimes a_n)$ is the element 
\begin{equation}
\label{eq:mass A-inf equality}
\langle a_1, \dotsc, a_n \rangle_D = \sum_{i=1}^{n-1} \bar a(i,n-i) \smile a(n-i, i) 
\end{equation}
of the Massey product $\langle a_1, \dotsc, a_n \rangle$ arising from the defining system $D=\{a(i,j): 1 \le i \le j \le n, (i,j) \ne (1,n)\}$, where
\[
b = 1 + d_{n-1} + d_{n-3} + \dotsm 
\]
is a sum with final term $d_1$ or $d_2$. 
\end{prop}

We remark that condition (a) does depend on the choice of retract. 

\begin{proof}
Using induction on $n$, we will show that the proposition follows directly from part (8) of Example \ref{eg:A-inf}. As \emph{loc.\ cit.} notes, assumption (a) implies that the relation \eqref{eq:A-inf to Massey} holds when evaluated on $a_1 \otimes \dotsm \otimes a_n$. Using the induction step, assumption (a) also implies that $\partial f_i(a_j \otimes \dotsm \otimes a_{j+i-1})$ is equal to 
\[
-m_{2,C}
(f_1 \otimes f_{i-1} - f_2 \otimes f_{i-2} + \dotsm + (-1)^i f_{i-1} \otimes f_1)(a_j \otimes \dotsm \otimes a_{j+i-1})
\]
(where we use $m_{2,C}$ to represent the $\smile$ map applied to tensors here). Using definition (b) for the $a(i,j)$, we see that \eqref{eq:A-inf to Massey} states that $\langle a_1, \dotsc, a_n \rangle_D$ is a member of the cohomology class $m_n(a_1 \otimes \dotsm \otimes a_n)$, as desired, up to some sign. From \cite[Thm.\ 3.1]{LPWZ2009}, this sign is given by $b$. 
\end{proof}

\begin{eg}
\label{eg:massey-Ainf-deg1}
In particular, for our case of interest where $d_i = 1$ for all $i$, $b = (-1)^{(n+1)(n+2)/2}$. This coincides with the signs in the homotopy Maurer--Cartan equation \eqref{eq: HMC}. Moreover, when $d_i = 1$ for all $i$, this formula for $b$ extends to the case $n=2$ (from $n\geq 3$ as in the statement of Proposition \ref{prop:A-inf vs Massey}), where we have set the Massey product equal to the cup product.
\end{eg}

The natural converse to Proposition \ref{prop:A-inf vs Massey} is not true, because there exist defining systems that cannot arise from a fixed set of $(f_i)$ as in \eqref{eq:massey to A-inf}. For example, consider a cup product $\langle a, a\rangle = 0$, and a defining system $D$ for the triple Massey product $\langle a,a,a\rangle_D = a(1,1) \smile a(2,3) + a(1,2) \smile a(3,3)$. If $a(1,2) \neq a(2,3)$, then \eqref{eq:massey to A-inf} is not possible. Of course, sufficient conditions such that a Massey product arises as in \ref{prop:A-inf vs Massey}(b) could be made clear. 

For our purposes, it suffices to produce conditions guaranteeing that Massey powers arise from an $A_\infty$-structure arising from a homotopy retract, as follows. 

\begin{prop}
Let $(C,\partial, \smile)$ be a dg-algebra and choose $\tau \in H^1(C)$. Let $\fT = \{\tau_1, \tau_2, \dotsc, \tau_{n-1}\}$ be a defining system for the Massey power $\langle \tau\rangle^n$. The following claims are equivalent.
\begin{enumerate}[leftmargin=2em]
\item Then there exists a choice of retract of $(C, \partial)$ by $(H^*(C), 0)$ as in Examples \ref{eg:H-sections} such that $\tau_i = f_i(a^{\otimes i})$.
\item There exists a section $h^2 : B^2(C) \ra C^1$ of $\partial\vert_{C^1}$ such that $h^2(\sum_{j=1}^{i-1} \tau_j \smile \tau_{i-j}) = \tau_i$. In this case, $m_n(a^{\otimes n}) = \langle \tau \rangle^n_\fT$. 
\end{enumerate}
\end{prop}

\begin{proof}
Given $h^2$ as in (2), one defines $i^1 : H^1(C) \ra C^1$ so that $\tau_1 = i^1(\tau)$ and defines $p^1 : C^1 \ra H^1(C)$ so that it kills $h^2(B^2(C))$ (which is possible because $h^2(B^2(C))$ is linearly disjoint from $\ker(d^1 : C^1 \ra B^2)$). Clearly these $h^2, i^1, p^1$ can be extended to a retract $h, i,p$ between $H$ and $C$. The converse is clear, in view of the Massey power defining system identities \eqref{eq:massey power reln}. 
\end{proof}

\section{Massey products and non-commutative deformations}
\label{sec: lifts-massey}

The main point of this section is to flesh out the introductory illustration, in \S\ref{subsec: illustrate}, that the Massey products and Massey powers of \S\ref{sec: massey} in the Hochschild cohomology of $\End_\F(\rho)$ are intrinsically related to lifts of $\rho$. More specifically, we will illustrate that the computations involved in Massey products and their defining systems are exactly those calculations that one needs to inductively choose and compute lifts $\F$-linearly. We are motivated in part by known applications in number theory of cup products -- especially Nekov\'a\v{r}'s height pairings (see \cite[11.5.5]{nekovar2006}) -- and an expectation that higher Massey products may be involved in analogous applications. 

\begin{eg}
Massey powers have been used to calculate an invariant that controls congruences of modular forms, answering a question of Mazur. This is a main theorem of the author's joint work with Wake \cite{WWE3}; see \S\ref{subsec: Mazur Eis}. 
\end{eg}

\begin{rem}
\label{rem: A-inf vs Massey}
On the same token, we expect that the reader will conclude from these computations that the approach using $A_\infty$-algebras in in \S\S\ref{sec:assoc pt}-\ref{sec:assoc mult pts} is more efficient, thanks to its rigidity. By ``rigidity,'' we are referring to Proposition \ref{prop:A-inf vs Massey}'s statement that the choice of a homotopy retract ``chooses all Massey products in advance.'' (Indeed, it fixes all $A_\infty$-products and also induces the data of Massey defining systems of degree $n+1$, when the appropriate degree $\leq n$ products vanish.) 
\end{rem}

\begin{rem}
Logically speaking, one could combine 
\begin{itemize}[leftmargin=2em]
\item the results relating $A_\infty$-algebras to non-commutative deformations in \S\S\ref{sec:assoc pt}-\ref{sec:assoc mult pts} with
\item the content of \S\ref{subsec: A-inf to Massey}, which shows that any $A_\infty$-product can be computed in terms of a Massey product (but the converse relation is not true) 
\end{itemize}
to prove a theorem computing non-commutative deformations in terms of Massey products, along the lines of \cite{laudal2002}. But we do not leverage this implication; instead, in this section, we directly link Massey products to deformations. 
\end{rem}

At the end of this section, there are remarks on perspectives from which Massey powers and products are superior to $A_\infty$-products. 

\subsection{Iterated extensions of representations}
\label{subsec: iter}

Let $\rho_i : E \ra M_{d_i}(\F)$, $1 \leq i \leq n$, be a sequence of representations of $E$. Let 
\[
\sigma_{i+1,i} \in Z^1(E, \Hom_\F(\rho_i, \rho_{i+1})) \quad \text{for } 1 \leq i < n
\]
be representatives of extension classes. Let $d = \sum_{i=1}^n d_i$. Assume that there exists a representation $\eta_{n} : E \ra M_d(\F)$ that realizes the $\sigma_{i+1}$ below the diagonal, in the sense that there exist $\sigma_{i,j}$ such that 
\begin{equation}
\label{eq: iter basis}
\eta_{n} =  
\begin{pmatrix}
	\rho_1 & & & & \\
	 \sigma_{2,1} & \rho_2 & & &  \\
	 \sigma_{3,1} & \sigma_{3,2} & \rho_3& &  \\
	 \vdots & & \ddots & \ddots  &  \\
	 \sigma_{n,1}& \dotsm & & \sigma_{n,n-1} & \rho_{n}
\end{pmatrix}.
\end{equation}
We form a dg-$\F$-algebra out of the Hochschild complex valued in the $E$-bimodule
\[
\bigoplus_{1 \leq i < j \leq n} \Hom_\F(\rho_i, \rho_j)
\]
with the natural compositions of homomorphisms making this $E$-bimodule a non-unital $\F$-algebra. Now we use Massey products in this dg-algebra. One may readily compute that $\eta_{n}$ is a homomorphism if and only if $\fS = \{\sigma(i,j)\} = \{\sigma_{j,i}\}$ is a defining system for the $n$th Massey product $\langle \sigma_{2,1}, \dotsc, \sigma_{n,n-1}\rangle$ and
\[
d\sigma_{n,1} = c(\fS).
\]
In other words, we have an equivalence, as follows. 
\begin{prop}
\label{prop: matric vanishing}
There exists an $\eta_n$ realizing the $\sigma_{i+1,i}$ below the diagonal if and only if the Massey product $\langle \sigma_{2,1}, \dotsc, \sigma_{n,n-1}\rangle$ is defined and contains zero. 
\end{prop}
This idea of such a connection between defining systems and extensions is due to May \cite{may1969}.

This may be taken to be a condition on iterated extensions of representations: the condition that $\fS$ is a defining system is equivalent to the existence of ``overlapping'' homomorphisms 
\[
\eta_{1,n-1} = 
\begin{pmatrix}
	\rho_1 & & &  \\
	 \sigma_{2,1} & \rho_2 & &   \\
%	 \sigma_{3,1} & \sigma_{3,2} & \rho_3& &  \\
	 \vdots & \ddots & \ddots &   \\
	 \sigma_{n-1,1}& \dotsm  & \sigma_{n-1,n-2} & \rho_{n-1}
\end{pmatrix},
\eta_{2,n} = 
\begin{pmatrix}
	\rho_2 & & &  \\
	 \sigma_{3,2} & \rho_3 & &   \\
%	 \sigma_{4,2} & \sigma_{4,3} & \rho_4& &  \\
	 \vdots & \ddots & \ddots  &  \\
	 \sigma_{n,1}& \dotsm &  \sigma_{n,n-1} & \rho_{n}
\end{pmatrix}
\]
Given this, the element $\langle \sigma_{2,1}, \dotsc, \sigma_{n-1,n}\rangle_\fS = c(\fS)$ of $\langle \sigma_{2,1}, \dotsc, \sigma_{n-1,n}\rangle$ vanishes (in cohomology) if and only if there exists a common extension $\eta_n$ of $\eta_{1,n-1}$ and $\eta_{2,n}$ exists. For $\eta_n$ exists if and only if there exists $\sigma_{n,1} \in C^1(E, \Hom_\F(\rho_1, \rho_n)$ such that $d\sigma_{n,1} = c(\fS)$. 

\subsection{Lifts of representations}
\label{subsec: explicit lifts}

We will start with a representation $\rho : E \ra M_d(\F)$ as in \S\S\ref{sec:assoc pt}-\ref{sec:assoc mult pts}. We will especially use the following coefficient algebras: for $n \geq 0$, write by $\F[\varep_n]$ for $A[\varep]/(\varep^{n+1}) \in \cC_\F$. 

An \emph{$n$th-order lift of $\rho$} is a lift $\rho_n$ of $\rho$ (as in Definition \ref{defn: lift assoc}) to $\F[\varep_n]$. We associate to $\rho_n$ an expression as a homomorphism to $M_{nd}(\F)$, extending the standard basis $(a_i)_{i=1}^d$ of $\F^{\oplus d}$ to a $\F$-basis of $\F[\varep_n]^{\oplus d}$ consisting of elements $(\varep^j a_i : 1 \leq i \leq d, 0 \leq j \leq n\}$ with the ordering by $j$ and then by $i$. We arrive at the matrix realization 
\begin{equation}
\label{eq: ep basis}
\rho_n = 
\begin{pmatrix}
	\rho & \sigma_1 & \sigma_2 & \dotsm & \sigma_n \\
	 & \rho & \sigma_1 & & \vdots  \\
	 & & \ddots & \ddots &  \\
	 & & & \rho & \sigma_1 \\
	 & & & & \rho
\end{pmatrix} : E \ra M_{nd}(\F)..
\end{equation}
We will render this as 
\[
\rho_n = \rho +  \sum_{i = 1}^n \sigma_i \varep^i : E \ra M_d(\F[\varep_n]),
\]
where $\sigma_i$ is a function $\sigma_i : E \ra M_d(\F) \cong \End_\F(\rho)$. 

Let $C = C^\bullet(E, \End_\F(\rho))$, so $\sigma_i \in C^1$. Because $\rho_n$ is a homomorphism, one readily observes that $\sigma_1$ lies in $Z^1$. More generally, the relations 
\begin{equation}
\label{eq:ut_cond}
d\sigma_i = \sum_{j=1}^{i-1} \sigma_j \smile \sigma_{i-j} \text{ for } 1 \leq i \leq n
\end{equation}
are satisfied if and only if the corresponding expression for $\rho_n$  a homomorphism. 

We may apply the connection between these conditions and Massey products from \S\ref{subsec: iter}. Next, we observe that these are Massey powers, using the symmetry visible by comparing \eqref{eq: ep basis} to \eqref{eq: iter basis}. Namely, the set $\fT = \{\sigma_1, \dotsc, \sigma_n\}$ satisfies \eqref{eq:ut_cond}, and therefore it constitutes a defining system for the $(n+1)$st Massey power $\dia{\sigma_1}_\fT^{n+1}$. We will simply denote this Massey power defining system $\fT$ by $\rho_n$ when it will not cause confusion. Thus the resulting Massey power is written $\langle \sigma_1 \rangle^{n+1}_{\rho_n}$, the cohomology class of the cocycle $c(\rho_n)$. 

If we increment $n$ to $n+1$, the new relation of \eqref{eq:ut_cond} is 
\[
d\sigma_{n+1} = \sum_{j=1}^n \sigma_j \smile \sigma_{n-j+1} =: c(\rho_n),
\]
so there exists some $\sigma_{n+1}$ satisfying \eqref{eq:ut_cond} if and only if $c(\rho_n)$ is a 2-coboundary. We summarize this discussion as follows. 

\begin{prop}
\label{prop: massey lifting}
For $n \geq 1$, let $\rho_n$ be an $n$th-order lift of $\rho$, defining cochains $\sigma_i \in C^1(E, \End_\F(\rho))$ for $1 \leq i \leq n$ as above. Then $\sigma_1 \in Z^1(E, \End_\F(\rho))$ and the Massey power $\langle \sigma_1 \rangle^{n+1}_{\rho_n}$ vanishes if and only if there exists an $(n+1)$st order deformation $\rho_{n+1} = \rho + \sum_{i=1}^{n+1} \sigma_i \varep^i$ extending $\rho_n$. In this case, we have an equality of 2-coboundaries $d\sigma_{n+1} = c(\rho_n)$ and the set of possible $\sigma_{n+1}$ is a $Z^1(E, \End_\F(\rho))$-torsor.  
\end{prop}

Varying over possible lifts extending a first order lift $\rho_1$ of $\rho$, we have the following immediate consequence, analogous to Proposition \ref{prop: matric vanishing}. 
\begin{cor}
Let $\rho_1 = \rho + \sigma_1\varep$ be a first order lift of $\rho$. Then $\rho_1$ extends to an $n$th-order lift if and only if the Massey power $\langle \sigma_1 \rangle^n \subset H^2(E, \End_A(\rho))$ is defined and contains zero. 
\end{cor}

\subsection{Expression of moduli spaces using Massey products}
\label{subsec: massey express}

In light of Proposition \ref{prop: massey lifting}, it is clear, in principle, that Massey products control $\Def^\mathrm{nc}_\rho$. In this section, we explain how to compute \emph{universal} lifts in terms of elements of Massey powers. We will only do this up to the point of illustrating the technique --- and illustrating its limitations compared to the $A_\infty$-based expression --- as we have already explained that $A_\infty$-structures give rise to Massey products and $A_\infty$-structures control $\Def^\mathrm{nc}_\rho$ (Remark \ref{rem: A-inf vs Massey}). 

We observe that there is a universal first order lift $\rho_1^u$ of $\rho$. It is induced by the ``universal 1-cocycle'' defined by 
\begin{equation}
\label{eq:universal 1-cocycle}
\sigma_1^u : E \lra M_d(Z^1(E,\End_\F(\rho))^*), \quad \gamma \mapsto (\sigma_1 \mapsto \sigma_1(\gamma)_{i,j})_{i,j}
\end{equation}
for $\gamma \in E$, $\sigma_1 \in Z^1(E, \End_\bF(\rho)$, and $(i,j)$ denoting the matrix coordinate. There is a left and a right $E$-action on $M_d(Z^1(E,\End_\bF(\rho))^*) \cong M_d(\bF) \otimes Z^1(E,\End_\bF(\rho))^*$ given by the usual left and right actions of $E$ on $M_d(\F)$ via $\rho$ and the multiplication map of $M_d(\F)$. 

Letting $\F[M]$ denote the square-zero $\F$-algebra extension of $\F$ by an $\F$-vector space $M$, we have
\begin{equation}
\label{eq:u 1st order lift}
\rho_1^u = \rho + \varep \sigma_1^u : E \lra M_d(\bF[Z^1(E,\End_\bF(\rho))^*]). 
\end{equation}
This lift is universal in the sense that for any other first-order lift $\rho_A$, there is a unique $\F$-linear map $E^1(E,\End_\F(\rho))^* \ra \m_A$ such that $\rho^1_u \otimes_{\bF[Z^1(E, \End_\bF(\rho))^*]} A = \rho_A$. This follows from the fact that $\rho_A - \rho \otimes_\F A : E \ra M_d(\m_A)$ is a cocycle, i.e.\ valued in $Z^1(E, \End_\F(\rho) \otimes_\F \m_A)$. 

For the remainder of this section we produce a universal lift of any order, applying the computations of the previous section. 

We establish notation for the sake of concision: write $T$ for $Z^1(E, \End_\bF(\rho))^*$ and let $\bF[T_n] := \bigoplus_{i=0}^n T^{\otimes i}$ be the free associative $\bF$-algebra on $T$ truncated at degree $n$, in analogy with $\bF[\varep_n]$. 
Inductively, we construct Massey powers of $\sigma_1^u \in Z^1(E,\End_\F(\rho)) \otimes T$. The base case is the cup product, which is the cohomology class of the unambiguously defined 2-cocycle 
\[
\sigma_1^u \cup \sigma_1^u \in H^2(E,\End_\bF(\rho)) \otimes T^{\otimes 2}.
\] 
Let $I_2 \subset T^{\otimes 2}$ be the minimal subspace such that $\sigma_1^u \cup \sigma_1^u$ vanishes modulo $I_2$. Then we can solve \eqref{eq:ut_cond} modulo $I_2$, i.e.\ there exists $\sigma_2^u \in C^1(E, M_d(T^{\otimes 2}/I_2))$ such that 
\[
d\sigma_2^u \equiv \sigma_1^u \cup \sigma_1^u \pmod{I_2}
\]
and the set of possible choices for $\sigma_2^u$ is a torsor under $Z^1(E, \End_\bF(\rho)) \otimes_\bF T^{\otimes 2}/I_2$. We get a second order lift 
\[
\rho_2^u = \rho +  \sigma_1^u + \sigma_2^u : E \lra M_d(\bF[T_2]/I_2). 
\]
In fact, untangling dualities, we see that $I_2$ is generated by the image in $T^{\otimes 2}$ of the image of the map 
\[
\cup^* : H^2(E, \End_\bF(\rho))^* \lra Z^1(E, \End_\bF(\rho))^* \otimes Z^1(E, \End_\bF(\rho))^* \cong T^{\otimes 2}
\]
that is dual to the cup product map.

The inductive step from order $n$ to order $n+1$ is to start with an $n$th-order lift $\rho_n^u = \rho + \sum_{i=1}^n \sigma_n^u$ of $\rho$ with coefficients in 
\[
\frac{\bF[T_n]}{(I_2, I_3, \dotsc, I_n)}
\]
and calculate Massey power $\langle \sigma_1^u \rangle^{n+1}_{\rho_n^u}$ valued in $H^2(E, \End_\bF(\rho)) \otimes_\bF T^{\oplus n+1}/I'_{n+1}$, where $I'_{n+1}$ is the ideal of $\F[T^{n+1}]$ generated by the image of $(I_2, I_3, \dotsc, I_n) \subset \F[T_n]$ under the (non-multiplicative) natural map $\F[T_n] \rinj \F[T_{n+1}]$. Then define $I_{n+1}$ to be the minimal submodule of $T^{\otimes n+1}$ containing the degree $n+1$ projection of $I'_{n+1}$ and such that $\langle \sigma_1 \rangle^{n+1}_{\rho_n}$ vanishes modulo $I_{n+1}$. As in the case $n=1$, we now have $\sigma_{n+1}^u$ valued in $T^{\otimes n}/I_{n+1}$ and $\rho_{n+1}^u$ valued in $\bF[T_{n+1}]/(I_2, \dotsc, I_{n+1})$. 

Because $I_n$ is concentrated in degree $n$, we have limits 
\[
R^\square_\rho := \varprojlim_n \frac{\bF[T_n]}{(I_2, I_3, \dotsc, I_n)}, \quad \rho^u := \varprojlim_n \rho_n^u :  E \ra M_d(R^\square_\rho) 
\]
where $(R^\square_\rho, \m^\square_\rho)$ is a complete local $\bF$-algebra quotient of $\hat T_\bF T$. 

We summarize the construction above and state its universal property. 
\begin{thm}
\label{thm:lifting ring}
	For any local Artinian $\bF$-algebra with residue field $\bF$ and lift $\rho_A$ of $\rho$ valued in $A$, there exists a unique local $\bF$-algebra homomorphism $R^\square_\rho \ra A$ such that $\rho_A = \rho^u \otimes_{R^\square_\rho} A$. That is, $\Def^\square_\rho = \Spf R^\square_\rho$ and $R^\square_\rho$ pro-represents $\Def^\square_\rho$. Moreover, we have a canonical isomorphism $(\m^\square_\rho/(\m^\square_\rho)^2)^* \cong Z^1(\End_\F(\rho))$. 
\end{thm}

We omit the proof, since it amounts to the same argument as the proof of Corollary \ref{cor: irred case A-inf}, but is more complicated due to the inductive construction of $R^\square_\rho$. 

\begin{rem}
\label{rem: Massey over S}
One advantage of Massey powers over $A_\infty$-products is that the base coefficient ring $\F$ may be replaced with a general commutative ring $S$. Indeed, all of the calculations of ideals $I_n$ make sense in this case, and the ideals $I_n$ may have non-trivial $S$-part. For example, when $S = \Z/p^2$ for some prime $p$, it is possible for
 \begin{itemize}
\item a non-trivial first-order lift to exist modulo $p^2$, i.e., over $S[\epsilon_1]$, 
\item an extension to a second-order lift to exist only over $\Z/p$, i.e., over $R := S[\epsilon]/(p\epsilon^3, \epsilon^4)$, 
\item and no extension to a third-order lift whatsoever. 
\end{itemize}
If this particular first-order lift is unique modulo $p$, then the universal deformation ring is $R$. For an example of an application of Massey products over $S = \bZ/p^n\Z$, see \cite{WWE3}. 
\end{rem}

\begin{rem}
One obstruction to applying the technology of $A_\infty$-algebras, as in \S\ref{sec:A-inf}, over a general commutative base ring $S$ in place of $\F$ is that homotopy retracts as in Example \ref{eg:H-sections} may not exist. A replacement for Kadeishvili's theorem (Corollary \ref{cor:A-inf on H}) is needed: see \cite{sagave2010}. 
\end{rem}

\part{Moduli of Galois representations and pseudorepresentations}

In this part, we apply the results of Part 2 to cases of interest in number theory. First we set up the theory for representations of a profinite group. Continuity of representations is taken to be implicit, and we now discuss only commutative coefficients. Then we adapt this theory for cases of interest in number theory: representations of a profinite group satisfying some additional condition. 

\section{Moduli of representations of a profinite group}
\label{sec: profinite moduli}

We recall moduli spaces of representations of a profinite group. These moduli spaces were set up in the author's previous work \cite{WE2018}, which we now recall. In contrast with \emph{loc.\ cit.}, we work in constant positive characteristic. Thus the initial coefficient ring is a finite field. We write $p$ for the characteristic. 

\subsection{Connected components biject with residual semi-simplification}
\label{subsec: residual SS}

The main result of \cite[\S3]{WE2018} is that the moduli of (integral) $p$-adic representations of a profinite group is the disjoint union of connected components parameterized precisely by the residual semi-simplification. We set up a precise meaning of the term ``residual semi-simplification.'' 
\begin{defn}
\label{defn: resid SS}
Let $G$ be a profinite group and let $p$ be a prime. A \emph{residual semi-simplification} is an equivalence class of semi-simple representations of $G$ valued in a finite field $\F$ of characteristic $p$ such that each simple summand is absolutely irreducible. The equivalence relation is isomorphism of representations, or, equivalently, isomorphism after change of coefficients via $\otimes_\F \overline{\F}$. 

A residual semi-simplification is called \emph{multiplicity-free} if there are no non-trivial isomorphisms among the simple summands. 
\end{defn}

\begin{rem}
Residual semi-simplifications are in bijection with \emph{residual pseudorepresentations} in \cite[Def.\ 3.4]{WE2018}. 
\end{rem}

Fix a representative $\rho : G \ra \GL_d(\F)$ of a residual semi-simplification; we take $\F_\rho$ to be the smallest possible base field such that a residual semi-simplification is defined over $\F_\rho$ and each irreducible factor is absolutely irreducible. We set up the equal-characteristic moduli of deformations of all residual representations with residual semi-simplification $\rho$. Without loss of generality, we take $\rho$ to be in block diagonal form in $\GL_d$, with diagonal summands 
\[
\rho \cong \bigoplus_{i=1}^r \rho_i.
\]
Here $\rho_i : G \ra \GL_{d_i}(\F)$ are the absolutely irreducible factors of $\rho$. Given this data, we write $\GL(\rho)$ for the corresponding Levi sub-$\F$-algebraic group 
\[
\GL(\rho) := \prod_{i=1}^r \GL_{d_i} \rinj \GL_d.
\]
Write $\PGL(\rho)$ for the quotient group of $\GL(\rho)$ by the center of $\GL_d$. 

The natural equal-characteristic category of coefficient rings are \emph{topologically finitely generated $\F_p$-algebras}, i.e.\ topological quotients of algebras of the form $\F_p\lb x_1, \dotsc, x_n\rb \langle y_1, \dotsc, y_m\rangle$. These algebras have the $(x_1, \dotsc, x_n)$-adic topology, and the angle brackets refer to restricted power series in this topology. We denote this category by $\Aff_{\F_p}$.  Equivalently, $\Aff_{\F_p}^\mathrm{op}$ is the category of finite type affine $\Spf \F_p$-formal schemes. It is natural to replace $\F_p$ by $\F$ when we impose the condition that a representation has residual semi-simplification $\rho$, as follows. 

\begin{defn}
\label{defn: Rep spaces}
Let $\Rep^\square_d$ denote the functor on $A \in \Aff_{\F_p}$ valued in sets, given by 
\[
\Rep^\square_d(A) = \{ \rho_A: G \ra \GL_d(A)\}.
\]
Likewise, we have the quotient groupoid by the adjoint action of $\PGL_d$, 
\[
\Rep_d := [\Rep^\square_d / \PGL_d],
\]
which represents the deformation groupoid defined just as in the definition of $\Rep^\square_\rho$, but with $\GL_d(A)$ replaced by the units of a $d$-dimensional Azumaya algebra over $A$ (see \cite[\S2.1]{WE2018}). 

Let $\Rep^\square_{\rho}$ denote the subfunctor of $\Rep^\square_d \times_{\F_p} \F_\rho$ given by 
\[
\Rep^\square_{\rho}(A) = \{ \rho_A: G \ra \GL_d(A) \mid \text{for all } f: A \ra \overline{\F}_\rho, (\rho_A \otimes_{A,f} \overline{\F}_\rho)^{ss} \simeq \rho \otimes_{\F_\rho} \overline{\F}_\rho\}
\]
for $A \in \Aff_{\F_\rho}$. Here $\simeq$ indicates being in the same orbit under the adjoint action of $\PGL_d(\overline{\F}_\rho)$. 
\end{defn}

A main result of \cite{WE2018} is that the mixed-characteristic versions of such spaces are representable in the category of $\Spf \Z_p$-formal schemes. We state this result specialized to equal-characteristic. 
\begin{thm}[{\cite[\S3.1]{WE2018}}]
\label{thm: algfamS3 main}
$\Rep^\square_{\rho}$ is representable by a topologically finite type affine $\Spf \F$-formal scheme $\Spf S_{\rho}$. We have
\[
\Rep^\square_d \times_{\F_p} \overline{\F}_p = \coprod_{\rho} \Rep^\square_{\rho} \times_{\F_\rho} \overline{\F}_\rho,
\]
where $\rho$ varies over all residual semi-simplifications of dimension $d$ and $\F_\rho$ denotes the coefficient field of $\rho$. 
\end{thm}
In additional to representability, the main upshot is that in order to understand the entire moduli space $\Rep^\square_d$, we may study it one residual semi-simplification $\rho$ at a time. Similarly, this theorem implies algebraicity and decomposition into connected components parameterized by $\rho$ for the stack quotients $\Rep_d$. 

In the case that the residual semi-simplification $\rho$ is irreducible, both $\Rep_\rho^\square$ and $\Rep_\rho$ are represented by the formal spectra of complete Noetherian local rings with residue field $\F_\rho$. These are the usual deformation rings for representations of profinite groups whose study was initiated by Mazur \cite{mazur1989}. We proceed with results that will be useful in order to study the case where $\rho$ is not irreducible. 

\subsection{Theory of pseudorepresentations, Cayley--Hamilton algebras, and generalized matrix algebras}
\label{subsec: review RR case}

We review the theory of pseudorepresentations due to Chenevier \cite{chen2014}. Because the review of \cite[\S2]{WE2018} is precisely what we need, we refer the reader there. Here, we recall only notation and selected parts of definitions.
\begin{itemize}[leftmargin=2em]
\item $D : E \ra A$ denotes a \emph{pseudorepresentation}. Using this notation implies that $E$ is an associative unital $A$-algebra, where $A$ is a commutative ring. This $D$ has a dimension $d \in \Z_{\geq 1}$, and is a functor from commutative $A$-algebras $B$ to functions $D_B : E \otimes_A B \ra B$ that are homogeneous of degree $d$ in $B$. 
\item When $G$ is a group, $D : G \ra A$ is notation for a pseudorepresentation $D : A[G] \ra A$. 
\item For any commutative $A$-algebra $B$, and $x \in E \otimes_A B$, there is a characteristic polynomial $\chi_D(x,t) \in B[t]$. 
\item Given $D : E \ra A$, there is a notion of a kernel two-sided ideal and Cayley--Hamilton two-sided ideal of $E$, 
\[
E \supset \ker(D) \supset \mathrm{CH}(D). 
\]
There exist canonical factorizations of $D$ through $E/\ker(D)$ and $E/\mathrm{CH}(D)$. 
\item A pseudorepresentation is called \emph{Cayley--Hamilton} when $\mathrm{CH}(D) = 0$. Equivalently, for all commutative $A$-algebras $B$ and all $x \in E \otimes_A B$, $\chi_D(x,x) = 0$. That is, $x$ satisfies its own characteristic polynomial $\chi_D(x,t) \in B[t]$. Collectively, such data $(E, A, D : E \ra A)$ is called a \emph{Cayley--Hamilton $A$-algebra}. 
\item Given a representation $\eta : G \ra M_d(A)$, there is an induced $d$-dimensional pseudorepresentation, denoted $\psi(\eta) : G \ra A$, given by composing $\eta$ with the determinant pseudorepresentation $\det : M_d(A) \ra A$. 
\item Similarly, if $E$ is equipped with a pseudorepresentation $D : E \ra A$ and $\eta$ is a homomorphism $G \ra E^\times$, then there is an induced pseudorepresentation $D \circ \eta : G \ra A$. We especially study the case where $(E, A, D)$ is a Cayley--Hamilton representation. Then we call the data
\[
(\eta : G \ra E^\times, E, A, D : E \ra A)
\]
a \emph{Cayley--Hamilton representation of $G$}, and we call $\psi(\eta) := D \circ \eta$ its \emph{induced pseudorepresentation} . 
\item A \emph{generalized matrix algebra over $A$} or \emph{$A$-GMA} is an associative $A$-algebra $E$ equipped with
\begin{itemize}[leftmargin=2em]
\item a complete orthogonal set of idempotents $(e_i)_{i=1}^r \subset E$,
\item $A$-algebra isomorphisms $e_i E e_i \risom M_{d_i}(A)$
\end{itemize}
that satisfy an extra condition. This notion, due in this form to Bella\"iche--Chenevier \cite[\S1]{BC2009}, was shown to admit a natural Cayley--Hamilton pseudorepresentation $D_\mathrm{GMA} : E \ra A$ in \cite[Prop. 2.23]{WE2018}. 
\item We call a Cayley--Hamilton representation $(\rho, A, E, D)$ a \emph{GMA representation} over $A$ when $E$ is also equipped with GMA data such that $D = D_\mathrm{GMA}$. 
\end{itemize}

The application of the tools above to the moduli of profinite groups is the main content of \cite[\S3]{WE2018}. We require a few more definitions and results about this situation, which we recall directly from \cite[\S3]{WE2018}. The results about deformation theory of pseudorepresentations are due to Chenevier \cite{chen2014}. However, here we work in constant characteristic $p$. 

We start with a fixed residual semi-simplification $\rho : G \ra \GL_d(\F_\rho)$, and write $\F = \F_\rho$. It induces a $d$-dimensional pseudorepresentation 
\[
D = \psi(\rho) : G \ra \F. 
\]
\begin{itemize}[leftmargin=2em]
\item There is a deformation functor sending $A \in \cC_\F$ to pseudorepresentations $D_A : G \ra A$ such that the reduction modulo $\m_A$
\[
G \buildrel{D_A}\over\lra A \lra \F
\]
is equal to $D$. Such a $D_A$ is called a \emph{pseudodeformation of $D$ to $A$}. This gives rise to a deformation functor on $\cC_A$ for $D$. 
\item There is a universal pseudodeformation ring $R_D$ representing the deformation functor for $D$ on the category $\cC_\F$. It also represents the extension of this deformation problem from $\cC_\F$ to $\Aff_\F$. Thus $R_D$ supports the universal deformation
\[
D^u : G \ra R_D
\]
of $D$. It is a complete local $\F$-algebra with residue field $\F$. When $G$ satisfies the $\Phi_p$ finiteness condition of \cite[\S1]{mazur1989}, $R_D$ is Noetherian. 
\item We say that a Cayley--Hamilton representation 
\[
(\eta : G \ra E^\times, A, E, D_E : E \ra A)
\]
where $A \in \cC_\F$ \textit{has residual pseudorepresentation $D$} when its induced pseudorepresentation of $G$ 
\[
\psi(\eta) = D_E \circ \eta : G \lra A
\]
is a deformation of $D$. For short, we say that $\rho$ is \emph{over $D$}. These notions have a sensible extension of coefficient algebras from $A \in \cC_\F$ to $A \in \Aff_\F$. 
\item There is a universal Cayley--Hamilton representation of $G$ over $D$, produced as follows. We let $E_D$ be the Cayley--Hamilton quotient of the universal pseudodeformation of $D$, that is, 
\[
E_D := \frac{R_D[G]}{\mathrm{CH}(D^u)}. 
\]
The theory of Cayley--Hamilton algebras recalled above implies that $D^u$ factors through $E_D$ as a Cayley--Hamilton pseudorepresentation; we denote the factorization by $D_{E_D}$. Thus we have a Cayley--Hamilton pseudorepresentation
\[
(\rho^u : G \ra E_D^\times, R_D, E_D, D_{E_D} : E \ra R_D)
\]
over $D$. Its induced pseudorepresentation $\psi(\rho^u) : G \ra R_D$ is equal to the universal pseudodeformation $D^u : G \ra R_D$ of $D$. 
\item When $G$ satisfies the $\Phi_p$ finiteness property, $E_D$ is finitely generated as a $R_D$-module. Therefore it is a Noetherian ring. 
\end{itemize}

\subsection{Application to moduli spaces of representations}
\label{subsec: apply to moduli}

In order to apply the theory of Cayley--Hamilton representations to the moduli spaces of representations $\Rep^\square_\rho$ and $\Rep_\rho$, we make the following observations and additional definitions, which come from \cite[\S3]{WE2018}. 

\begin{itemize}[leftmargin=2em]
\item When $D = \psi(\rho) : G \ra \F$, there is a natural functor $\psi^\square : \Rep^\square_\rho \ra \Spf R_D$ sending a representation $\eta : G \ra M_d(A)$ with residual semi-simplification $\rho$ to its induced pseudorepresentation $\psi(\eta) = \det \circ \, \eta : G \ra A$, which is a pseudodeformation of $D$. And $\psi^\square$ factors through 
\[
\psi : \Rep_\rho \ra \Spf R_D,
\]
which sends an Azumaya algebra valued representation of $G$ to the pseudorepresentation of $G$ that arises from composition with the Azumaya algebra's reduced norm.  
\item Given a $d$-dimensional Cayley--Hamilton algebra $(E, A, D: E \ra A)$, there exists an affine $A$-scheme $\Rep^\square_{E,D}$ of representations of $E$ that are \emph{compatible with $D$}. It is a functor on commutative $A$-algebras $B$ sending
\[
B \mapsto \{\eta : E \ra M_d(B) \mid \psi(\eta) := \det \circ \, \eta : E \ra B \text{ equals } D \otimes_A B : E \ra B\}. 
\]
Likewise, there is an moduli groupoid of Azumaya algebra-valued representations of $E$, which is represented by the stack quotient $[\Rep^\square_{E,D}/\PGL_d]$. 
\item Given a $d$-dimensional GMA $(E, A, D_\mathrm{GMA} : E \ra A)$ with idempotents $e_i$ each of dimension $d_i$, there exists a closed sub-$A$-scheme
\[
\Rep^\mathrm{GMA}_{E,D} \subset \Rep^\square_{E,D}
\]
of \emph{adapted representations}. The notion and moduli of adapted representations were first studied in \cite[\S1.3]{BC2009}. These are matrix algebra-valued representations that fix the data of idempotents, where we choose a diagonal data of idempotents in the matrix algebra. 
\item Let $Z(e_i)$ be the split torus in $\GL_d$ which centralizes the block diagonal subalgebra 
\[
\bigoplus_{i=1}^r e_i E e_i \cong \bigoplus_{i=1}^r M_{d_i}(A) \rinj M_d(A).
\]
This torus has a natural adjoint action on $\Rep^\mathrm{GMA}_{E,D}$, and its stack quotient admits an isomorphism
\[
[\Rep^\mathrm{GMA}_{E,D}/Z(e_i)] \cong \Rep_{E,D}. 
\]
\item Let $e^{11}_i$ denote the idempotent of $e_i E e_i \cong M_{d_i}(A)$ cutting out the $(1,1)$-coordinate of $M_{d_i}(A)$. Let $e^{11} = \sum_{i=1}^r e_i^{11}$. We then get a Morita-equivalent algebra
\begin{equation}
\label{eq: Morita GMA}
e^{11} E e^{11}
\end{equation}
that naturally admits the structure of a GMA: the idempotents are $e^{11}_i$ and $d_i = 1$ for all $1 \leq i \leq r$. We write $\cA_{i,j}$ for
\[
\cA_{i,j} := e^{11}_j E e^{11}_i = e^{11}_j (e^{11} E e^{11}) e^{11}_i.
\]

\item According to \cite[Prop.\ 1.3.9]{BC2009}, there is an expression for the $A$-algebra $S^\mathrm{GMA}_{E,D}$ representing the affine $A$-scheme $\Rep^\mathrm{GMA}_{E,D}$ in terms of the multiplication map on $e^{11}E e^{11}$ decomposed into its idempotent-based coordinates as
\[
\varphi_{i,j,k} : \cA_{i,j} \otimes_A \cA_{j,k} \ra \cA_{i,k}. 
\]
The expression for $S^\mathrm{GMA}_{E,D}$ is 
\begin{equation}
\label{eq: adapt ring}
S^\mathrm{GMA}_{E,D} \lrisom \frac{
\Sym^*_A \left(\bigoplus_{1 \leq i \neq j \leq r} \cA_{i,j}\right)
}{
\left(x \otimes y - \varphi(x \otimes y)\right)
},
\end{equation}
where the denominator stands for the ideal generated over varying $x \in \cA_{i,j}$, $y \in \cA_{j,k}$, and $\varphi = \varphi_{i,j,k}$, over varying $(i,j,k)$. 
\end{itemize}

We summarize the results about these objects given in \cite{WE2018}. 

\begin{thm}[{\cite[\S3]{WE2018}}]
\label{thm: algfam main}
Let $\rho$ be a residual semi-simplification with $D = \psi(\rho)$, and assume that $G$ satisfies the $\Phi_p$ finiteness condition. There are natural isomorphisms of $\Spf \F$-spaces
\[
\Rep^\square_\rho \cong \Rep_{E_D, D_{E_D}}^\square, \qquad 
\Rep_\rho \cong \Rep_{E_D, D_{E_D}}, 
\]
each of which admits a finite type module over $\Spec R_D$. That is, there is an isomorphism of the moduli of 
\begin{itemize}[leftmargin=2em]
\item representations of $G$ with residual semi-simplification $\rho$ and
\item representations of $E_D$ that are compatible with the pseudorepresentation $D^u : E_D \ra R_D$,
\end{itemize}
which is an isomorphism of $\Spf R_D$-formal spaces. 

Furthermore, assuming that $\rho$ is multiplicity-free,
\begin{enumerate}[leftmargin=2em]
\item  $E_D$ admits the structure $(e_i)_{i=1}^r$ of an $R_D$-GMA such that $D^u = D_\mathrm{GMA}$ and
\[
\Rep_\rho \cong [\Rep^\mathrm{GMA}_{E,D_{E_D}} / Z(e_i)]
\]
\item If we let $S^\mathrm{GMA}_\rho$ be the commutative $A$-algebra representing the affine scheme $\Rep^\mathrm{GMA}_{E,D_{E_D}}$, then the structure morphism $\psi_\mathrm{GMA} : \Rep^\mathrm{GMA}_{E,D_{E_D}} \ra \Spf R_D$ induces an isomorphism from $R_D$ to the invariant subring 
\[
R_D \lrisom (S^\mathrm{GMA}_{E,D})^{Z(e_i)}.
\]
\end{enumerate}

\end{thm}

\begin{proof}
The initial statements come from \cite[Thm.\ 3.7]{WE2018}. The GMA structure claimed in (1) comes from \cite[Thm.\ 2.22(ii)]{chen2014}. The rest of (1) is proved in \cite[Thm.\ 2.27]{WE2018}, and (2) is \cite[Thm.\ 3.8(4)]{WE2018}. 
\end{proof}

\begin{rem}
As discussed in \cite[\S3]{WE2018}, the result (2) means that $\Spec R_D$ is canonically isomorphic to the GIT quotient for the stack $\Rep_\rho$. Even when $\rho$ is not multiplicity-free, so that (2) is not known to hold, nonetheless it is proved in \textit{loc.\ cit.}\ that there is a map $R_D \ra (S^\square_\rho)^{\PGL_d}$ (where $\Rep^\square_\rho \cong \Spf S^\square_\rho$) that is very close to being an isomorphism. 
\end{rem}

\section{Presentations in terms of $A_\infty$-structure on group cohomology}
\label{sec: profinite presentations}

We express $\Rep^\square_\rho$ as a moduli space of representations of an algebra, so that we may apply the results of Part 2. 

\subsection{From Hochschild cohomology to group cohomology}
\label{subsec: hochs-group}

Given a left $G$-module $V$, we write $C^\bullet(G,V)$ for the cochain complex of inhomogeneous group cochains; see \S\ref{subsec: kad intro}, or, for a full introduction, see e.g.\ \cite{brown1982}. Because we are working over a field, group cohomology realizes the $\Ext$-functors in the category of $\F[G]$-modules. There is also a direct compatibility between group cohomology and Hochschild cohomology of left modules for the group algebra. 
\begin{prop}
	Let $V, W$ be left $\F[G]$-modules. Give $\Hom_\F(W,V)$ the natural induced $\F[G]$-bimodule structure (left via $V$, right via $W$). 
	\begin{enumerate}[leftmargin=2em]
		\item There is an isomorphism from the Hochschild complex to the group cochain complex
	\begin{align*}
	\theta^n : C^n(\F[G], \Hom_\F(W,V)) & \lrisom C^n(G, \Hom_\F(W,V)) \\
	(f : \F[G]^{\otimes n} \ra \Hom_\F(W,V)) & \mapsto f\vert_{G^{\times n}},
	\end{align*}
	using the natural embedding $G^{\times n} \rinj \F[G]^{\otimes n}$. 
	
	\item There are canonical isomorphisms of graded $\F$-vector spaces 
	\[
	H^\bullet(C^\bullet(\F[G], \Hom_\F(W,V))) \lrisom  H^\bullet(G, \Hom_\F(W, V)) \lrisom \Ext^\bullet_{\F[G]}(V,W). 
	\]
\end{enumerate}

\end{prop}

\begin{proof}
	Because $\F[G]^{\otimes n} \cong \F[G^{\times n}]$, $\theta$ is an isomorphism of graded vector spaces. The differentials are compatible under $\theta$ because the map from $\F[G]$-bimodule actions $\star = (\star_\mathrm{left}, \star_\mathrm{right})$ to left $G$-module actions
	\[
	g \cdot m := g \star_\mathrm{left} m \star_\mathrm{right} g^{-1}
	\]
	sends the given $\F[G]$-bimodule action on $\Hom_\F(W,V)$ to its standard left $G$-module action. From there, one observes that the formula for the Hochschild differential (Definition \ref{defn: Hochs}) is the same as the formula for the differential on inhomogeneous group cochains. 
	
The leftmost isomorphism of (2) is then clear. The right isomorphism relies on $C^n(G,\F)$ being a projective resolution for the $G$-module $\F$, and $\F$ having trivial homological dimension. 
\end{proof}

\begin{rem}
The arguments are valid for discrete modules with a continuous action of $G$, when $G$ is a profinite group, using the fact that $\F$ is finite. The key fact is that $C^n(G,\F)$ remains a resolution of $\F$; see e.g.\ \cite[Prop.\ 6.2.2]{RZ2010}. Correspondingly, the ambient categories are continuous finite discrete $G$-modules, resp.\ continuous finite discrete $\F[G]$-modules. 
\end{rem}

\subsection{Presentation of the completed group algebra}

Note that each absolutely irreducible factor $\rho_i$ of $\rho$ cuts out a maximal ideal of the completed group algebra $\F\lb G\rb$ (see e.g.\ \cite[\S5.3]{RZ2010} for the definition of $\F\lb G\rb$). We now let
\[
E := \F\lb G\rb
\]
and apply the theory of Part 2 as follows.
\begin{itemize}[leftmargin=2em]
\item We consider only \emph{continuous} representations of $E$, but we leave this implicit without stating it explicitly in the sequel. 
\item Likewise, we let $C = C^\bullet(E, \End_\F(\rho))$ denote the \emph{continuous} Hochschild cochain complex
\[
C^i(E, \End_\F(\rho)) := \Hom_\mathrm{cts}(E^{\otimes i}, \End_\F(\rho)). 
\]
It is a straightforward exercise to check that the differential and multiplication in $C$ preserves continuity. 
\end{itemize}

After setting up these two instances of continuity, there are no additional instances where we must impose it. For consider that the lift $\rho \oplus \xi$ of $\rho$ to $A \in \cA_\F$ associated to a Maurer--Cartan element $\xi \in \mathrm{MC}(C, A) \subset C^1 \otimes \m_A$ is obviously continuous. As all other representations are ultimately produced out of elements of $C^1$ along with formulas within $C$ that preserve continuity, namely, those of Example \ref{eg:KFT2}. We will implicitly always work in the continuous case in the sequel. 

\begin{thm}
\label{thm: profinite r-pointed dual}
Assume that $G$ satisfies the $\Phi_p$ finiteness condition. Choose an $r$-pointed homotopy retract between $C = C^\bullet(E, \End_\F(\rho))$ and its cohomology $H = H^\bullet(C)$, as in Example \ref{eg: r-pointed retract}. Choose also idempotents as in \eqref{eq: r isom lift}, including $e$. These choices induce an $A_\infty$-algebra structure $m$ on $H$ and complete $\F^r$-algebra isomorphisms
\begin{align*}
\rho^u : \F[G]^\wedge_{\rho} \lrisom \End_\F(V)\, \uotimes\,  \frac{\hat T_{\F^r} (\Sigma H^1)^*}{(m^*((\Sigma H^2)^*))},\\
e \rho^u e : R^\mathrm{nc}_\rho := e \F[G]^\wedge_{\rho} e \lrisom \frac{\hat T_{\F^r} (\Sigma H^1)^*}{(m^*((\Sigma H^2)^*))}
\end{align*}
\end{thm}
\begin{proof}
This is an application of Corollary \ref{cor: r-pointed dual}. 
\end{proof}

In order to concisely list the choices made in main theorems, we set up this 
\begin{defn}
	\label{defn: pres datum}
	Given a profinite group $G$ and a residual semi-simplification $\rho$, a \emph{presentation datum} for the moduli of representations of $G$ with residual semi-simplification $\rho$ is 
	\begin{itemize}[leftmargin=2em]
		\item A (ordered) basis for $\rho$ making the expression $\rho \cong \bigoplus_{i=1}^r \rho_i$ compatible with the ordering of the factors of the block diagonal subalgebra $\bigoplus_{i=1}^r M_{d_i}(\F) \rinj M_d(\F)$, where $\rho_i : E \ra M_{d_i}(\F)$. 
		\item A $r$-pointed homotopy retract between $H$ and $C$, as in Example \ref{eg: r-pointed retract}.
		\item A choice of $\F^r$-algebra structure on $\F[G]^\wedge_{\ker\rho}$ arising from choices of idempotents as in \eqref{eq: r isom lift}, compatible with the standard matrix idempotents in the codomains of the $\rho_i$. 
	\end{itemize}
\end{defn}

\subsection{Presentation of the moduli space $\Rep_\rho$}

Fix a residual semi-simplification $\rho$. Next we want to deduce a presentation for $\Rep_\rho$. We will do this in the case that $\rho$ is multiplicity-free. 

Because representations of $G$ parameterized by $\Rep_\rho$ are continuous, it follows that the induced $\F[G]$-action factors through $E := \F\lb G\rb$. Moreover, the condition that they have residual semi-simplification $\rho$ implies that they factor through the completion 
\[
E^\wedge_{\ker \rho}. 
\]
(Note that one can assume that $\rho$ is multiplicity-free without any loss of generality on the algebras $E^\wedge_{\ker\rho}$ we study here.) The only difference from the previous subsection is that we now consider coefficients in $M_d(A)$ for \emph{commutative} algebras $A \in \Aff_\F$, with its standard $\F^r$-algebra structure. Formerly, in place of $M_d(A)$, we considered the category of coefficient algebras $\cA_\F^r$. 

We require some definitions in order to state the presentation. Recall from \S\ref{subsec: GIT intro} the notion of a simple closed path $\gamma$, the set of simple closed paths $SCP(\bold r)$, and tensor module of $\Ext^1$-modules $\Ext^1_{\F[G]}(\gamma)^*$. Throughout the following discussion, we use the notation $\Ext^k_{\F[G]}(\rho_j, \rho_i)$ to factor into summands the cohomology elsewhere denoted as $H^k = H^k(G, \End_\F(\rho)) = \Ext^k_{\F[G]}(\rho,\rho)$. 

\begin{defn}
Let 
\[
I_\mathrm{cyc} \subset T_{\F^r} \Sigma \Ext^1_{\F[G]}(\rho,\rho)^* \cong T_{\F^r}\bigoplus_{1 \leq i,j \leq r} \Sigma\Ext^1_{\F[G]}(\rho_j, \rho_i)^*
\]
(note that this free $\F^r$-algebra is not completed) denote the ideal generated by the submodule of \emph{cyclic tensors} 
\[
\bigoplus_{\gamma \in SCP(\bold r)} \Ext^1_{\F[G]}(\gamma)^*
\]
(where because we are using non-symmetric tensors, we sum over all of the simple closed paths $SCP(\bold r)$ that constitute the simple cycles $SC(\bold r)$). Note that $SCP(\bold r)$ and $SC(\bold r)$ include the loop $i \ra i$ for each $i \in \bold r$. 

Let $\hat T_\mathrm{cyc} \Sigma \Ext^1(\rho,\rho)^*$ denote the \textit{cyclic completion} of $T_{\F^r} \Sigma \Ext^1_{\F[G]}(\rho,\rho)^*$, that is, its completion by $I_\mathrm{cyc}$, which admits an inclusion into $\hat T_{\F^r} \Sigma \Ext^1_{\F[G]}(\rho,\rho)$. Let $\hat S_\mathrm{cyc} \Sigma \Ext^1(\rho,\rho)^*$ denote the abelianization of $\hat T_\mathrm{cyc} \Sigma \Ext^1(\rho,\rho)^*$. 
\end{defn}

Notice that the submodule
\[
m^*\Ext^2_{\F[G]}(\rho,\rho) \subset \hat T_{\F^r} \Sigma \Ext^1_{\F[G]}(\rho,\rho)
\]
lies within $\hat T_\mathrm{cyc} \Sigma \Ext^1(\rho,\rho)^*$. This inclusion follows from the following fact, which we will use frequently. 
\begin{fact}
\label{fact: cyc bound}
Any non-zero simple tensor of degree $s$ on $\Ext^1_{\F[G]}(\rho,\rho)^*$ includes as a tensor-factor at least $\lfloor s/r\rfloor$ simple cyclic tensors, and is expressible as the product of a cyclic tensor with a tensor of degree bounded by $r$. 
\end{fact}

Consequently, we may sensibly define
\[
S^\mathrm{GMA}_D := \frac{\hat S_\mathrm{cyc} \Sigma \Ext^1_{\F[G]}(\rho,\rho)^*}
{m^* \Sigma \Ext^2_{\F[G]}(\rho,\rho)^*}.
\]
Here ``GMA'' refers to a moduli problem represented: $S^\mathrm{GMA}_D$ is a commutative $\F^r$-algebra and, in particular, is not a GMA. 

\begin{thm}
	\label{thm: adapted pres}
	Assume that $G$ satisfies the $\Phi_p$ finiteness condition. Assume also that the residual semi-simplification $\rho$ is multiplicity-free. Fix a presentation datum (Definition \ref{defn: pres datum}). These choices induce
	\begin{enumerate}[leftmargin=2em]
		\item an $\F^r$-algebra structure on the universal Cayley--Hamilton algebra $E_D$ over $D$, 
		\item a GMA structure on $E_D$ such that $D_\mathrm{GMA} = D^u$, whose idempotents are compatible with the $\F^r$-structure of part (1), and 
		\item a presentation $\Rep^\mathrm{GMA}_{E_D} \cong \Spf S^\mathrm{GMA}_D$ (where $S^\mathrm{GMA}_D$ is forgotten down from a $\F^r$-algebra to an object of $\Aff_\F$). 
\end{enumerate}
\end{thm}

\begin{proof}
The choice of presentation datum furnishes the choices to which Theorem \ref{thm: profinite r-pointed dual} applies. Thus we have the $\F^r$-algebra structure on $E^\wedge_\rho$ given by applying the composite of 
\begin{itemize}
\item the isomorphism $\rho^u : E^\wedge_\rho \risom M_d(\F) \uotimes R^\mathrm{nc}_\rho$ furnished by Theorem \ref{thm: profinite r-pointed dual} 
\item  the block diagonal subalgebra map $\bigoplus_{i=1}^r M_{d_i}(\F) \rinj M_d(\F)$ 
\end{itemize}
to  the identity elements of the matrix algebras $M_{d_i}(\F)$. 

The natural map
\[
E^\wedge_\rho \lra E_D
\]
then gives us result (1). To prove (2), observe that the set of $r$ ordered idempotents of $E^\wedge_\rho$ induced by the $\F^r$-structure induces a GMA structure on $E_D$. Indeed, just like the idempotents of $E_D$ supplied by \cite[Thm.\ 2.22(ii)]{chen2014} in Theorem \ref{thm: algfam main}(1), they lift the standard idempotents of $\F[G]/\ker\rho \cong E_D/\ker\rho \cong \bigoplus_{i=1}^r M_{d_i}(\F)$. Therefore, by \cite[Thm.\ 2.9.18(iii), pg.\ 242]{rowen1988} (just as in the argument for \cite[Lem.\ 5.6.8]{WWE1}), there is some conjugation in $E_D$ sending one ordered set of idempotents to the other. We also know that if one set of ordered idempotents supplies a GMA structure, so does its conjugate. 
 
 We have noted in Theorem \ref{thm: algfam main}(1) that the native pseudorepresentation $D^u : E_D \ra R_D$ is equal to the pseudorepresentation $D_\mathrm{GMA} : E_D \ra R_D$ induced by this GMA structure. This completes part (2). 
	
	We define an auxiliary moduli functor $\Rep^{\F^r}_{\F[G]^\wedge_{\ker\rho}}$ by sending $A \in \Aff_\F$ to the set of $\F^r$-algebra homomorphisms 
	\[
	\F[G]^\wedge_{\rho} \lra M_d(A)
	\]
	that are compatible with the maps from $\bigoplus_{i=1}^r M_{d_i}(\F)$ into each of them. (The map to $E^\wedge_\rho$ comes from the inverse of $\rho^u$.) We claim that the map $\F[G]^\wedge_\rho \ra E_D$ induces an isomorphism of functors on $\Aff_\F$
	\[
	\Rep^{\F^r}_{\F[G]^\wedge_{\rho}} \cong \Rep^\mathrm{GMA}_{E_D}.
	\]
	Indeed, because the residual semi-simplification of all representations parameterized by $\Rep^{\F^r}_{\F[G]^\wedge_{\ker\rho}}$ is $\rho$, Theorem \ref{thm: algfam main} provides that these representations factor through $E_D$. Because of the compatibility of idempotents arranged above, these factorizations $E_D \ra M_d(A)$ preserve the GMA structure. Therefore, each $A$-point of $\Rep^{\F^r}_{\F[G]^\wedge_{\rho}}$ induces an $A$-point of $\Rep^\mathrm{GMA}_{E_D}$ via this factorization. There is a left inverse map $\Rep^\mathrm{GMA}_{E_D} \rinj \Rep^{\F^r}_{\F[G]^\wedge_{\rho}}$ given by $\F[G]^\wedge_\rho \ra E_D$, which is also a right inverse because $\Rep^\mathrm{GMA}_{E_D}$ and $\Rep^{\F^r}_{\F[G]^\wedge_{\rho}}$ admit compatible monomorphisms into $\Rep^\square_\rho$. 
		
	Now we apply the Morita equivalence of $\F[G]^\wedge_\rho$ and $e\F[G]^\wedge_\rho e$ to draw an isomorphism of functors on $\Aff_\F$ 
	\[
	\Rep^{\F^r}_{e\F[G]^\wedge_{\rho}e} \cong \Rep^{\F^r}_{\F[G]^\wedge_{\rho}},
	\]
	where $\Rep^{\F^r}_{e\F[G]^\wedge_{\rho}e}$ sends $A \in \Aff_\F$ to $\F^r$-algebra homomorphisms
	\[
	e \F[G]^\wedge_\rho e \lra M_r(A),
	\]
	where the $\F^r$-structure on $M_r(A)$ comes from the standard diagonal idempotents with exactly one non-zero entry. 	(The same procedure is applied to GMAs in \cite[\S1.3.2]{BC2009}.) 
	
	The augmented $\F^r$-algebra isomorphism of Theorem \ref{thm: profinite r-pointed dual} 
	\[
	(e\rho^u e)^{-1} : R^\mathrm{nc}_\rho = \frac{\hat T_{\F^r} \Sigma \Ext^1_{\F[G]}(\rho,\rho)^*}{m^* (\Sigma \Ext^2_{\F[G]}(\rho,\rho)^*)} \lrisom e\F[G]^\wedge_{\rho}e
	\]
	allows us to calculate $\Rep^{\F^r}_{e\F[G]^\wedge_{\rho}e}$. Firstly we work in the case $\Ext^2_{\F[G]}(\rho,\rho) = 0$ and consider $\F^r$-homomorphisms 
	\begin{equation}
	\label{eq: Fr Homs from R}
	R^\mathrm{nc}_\rho \cong \hat T_{\F^r} \Sigma \Ext^1_{\F[G]}(\rho,\rho)^* \lra M_r(A). 
	\end{equation}
Write $I_A \subset A$ for an ideal that is maximal among ideals of definition for $A$. By considering the standard coordinate-wise formulas for matrix multiplication, we observe that homomorphisms \eqref{eq: Fr Homs from R} biject with the $A$-submodule of 
	\begin{equation}
	\label{eq: Ext1 coordinate decomp}
	\Sigma \Ext^1_{\F[G]}(\rho,\rho) \otimes A \cong \bigoplus_{1 \leq i, j \leq r} \Sigma \Ext^1_{\F[G]}(\rho_j, \rho_i)) \otimes A
	\end{equation}
	that is the intersection of the kernels of the maps
	\[
	f(\gamma, x) : \Sigma \Ext^1_{\F[G]}(\rho,\rho) \otimes A \ra A/I_A
	\]
	parameterized by $\gamma \in SCP(\bold r)$ and $x \in \Ext^1_{\F[G]}(\gamma)^*$. This submodule corresponds precisely to maps $T_{\F^r} \Sigma \Ext^1_{\F[G]}(\rho,\rho)^* \ra M_r(A)$ that factor through its completion at $I_\mathrm{cyc}$, that is, $\hat T_\mathrm{cyc}\Sigma \Ext^1_{\F[G]}(\rho,\rho)^*$. 
	
	\begin{rem}
	For the sake of clarity, we write out the definition of $f(\gamma,x)$. Given an element $y = (y_{ij})$ in $\Sigma \Ext^1_{\F[G]}(\rho,\rho) \otimes A$, where the coordinates $y_{ij}$ arise from the decomposition \eqref{eq: Ext1 coordinate decomp}, we obtain the element
	\[
	\bigotimes_{i=0}^{l_\gamma-1} x_{\gamma(i)\gamma(i+1)}(y_{\gamma(i)\gamma(i+1)}) \in A^{\otimes l_\gamma}. 
	\]
	Multiplication in $A$ yields the desired element of $A$. 
	\end{rem}
	
	Next, admit the case that $\Ext^2_{\F[G]}(\rho,\rho) \neq 0$. From the properties of free algebras, we deduce that $\Rep^{\F^r}_{e\F[G]^\wedge_\rho e}(A)$ is naturally isomorphic to 
	\[
	\Hom_{\F^r}(\frac{\hat T_\mathrm{cyc}\Sigma \Ext^1_{\F[G]}(\rho,\rho)^*}{m^* \Sigma \Ext^1_{\F[G]}(\rho,\rho)^*}, M_r(A)).
	\]
	This in turn is naturally isomorphic to 
	\[
	\Hom_{\F}(\frac{\hat T_\mathrm{cyc}\Sigma \Ext^1_{\F[G]}(\rho,\rho)^*}{m^* \Sigma \Ext^1_{\F[G]}(\rho,\rho)^*}, A),
	\]
	where we simply forget the $\F^r$-algebra structure on the domain. Finally, because $A$ is assumed to be commutative, we deduce the desired result (3). 
\end{proof}

\begin{rem}
One can derive from Theorem \ref{thm: adapted pres} the relationship established in \cite[\S1.5.3-1.5.4]{BC2009} between the structure of the GMA $E_D$ and various $\Ext$-groups, mainly $\Ext^1$. See especially \cite[Rem.\ 1.5.7]{BC2009}, which discusses the relationship with $\Ext^2$: the ambiguities there are controlled by the $m^*$ map on $\Ext^2_{\F[G]}(\rho,\rho)^*$. 
\end{rem}

\subsection{Presentation of the pseudodeformation ring}
\label{subsec: PsDef proofs}

Next, we apply Theorem \ref{thm: algfam main}(2) in order to present the pseudodeformation ring $R_D$. 

We recall from Definition \ref{defn: R1D} the complete Noetherian local ring $R^1_D$, which will be shown to present $R_D$ when $\Ext^2_{\F[G]}(\rho,\rho) = 0$. 

\begin{thm}
\label{thm: R_D pres body}
Assume that $G$ satisfies the $\Phi_p$ finiteness condition and assume that the residual semi-simplification $\rho$ is multiplicity-free. Fix a presentation datum (Definition \ref{defn: pres datum}). 

These choices produce a presentation of $R_D$ as a complete Noetherian local $\F$-algebra with residue field $\F$,
\begin{equation}
\label{eq: R_D pres body}
\frac{R^1_D}{
\Bigg(\displaystyle\bigoplus_{i,j \in \bold r} m^* \Sigma \Ext^2_G(\rho_j, \rho_i)^* \otimes 
\Big(\displaystyle\bigoplus_{\gamma \in SCC(i,j)} \Sigma \Ext^1_G(\gamma)^*\Big)\Bigg)
}
\lrisom 
R_D.
\end{equation}
\end{thm}

\begin{proof}
The presentation for $R_D$ follows from combining
\begin{itemize}[leftmargin=2em]
\item the presentation for $S^\mathrm{GMA}_{E_D}$ of Theorem \ref{thm: adapted pres} with
\item the result of \cite[\S2-3]{WE2018} stated in Theorem \ref{thm: algfam main}(2), that $R_D$ is the invariant subring of $S^\mathrm{GMA}_{E_D}$ under the adjoint action of the torus $Z(\rho)$. 
\end{itemize}
Because $Z = Z(\rho)$ is linearly reductive over $\F$ (even in positive characteristic, as it is a torus), its invariant functor is exact. In particular, for any affine $\F$-scheme $\Spec S$ that is a closed subscheme $\Spec S \cong \Spec S'/I \subset \Spec S'$ admitting a $Z$-action on $S'$ that preserves $S$, one has (see e.g.\ \cite[Rem.\ 4.11]{alper2013}) 
\[
S^Z = (S'/I)^Z \cong S'^Z/I^Z. 
\]
We apply this to the presentation of $S^\mathrm{GMA}_{E_D} = S'/I$ as a quotient, where 
\[
S' := \hat S_\mathrm{cyc} \Sigma \Ext^1_{\F[G]}(\rho,\rho)^*, \qquad I := (m^* \Sigma \Ext^2_{\F[G]}(\rho,\rho)^*);
\]
indeed, $I$ is $Z$-stable because it has generators that are isotypic for certain characters of $Z$ (these are in bijection with weights of the adjoint action of $Z$ as the torus in $\PGL_r$). 

We claim that $S'^Z \cong R^1_D$. Indeed, we see that a simple tensor in $S'$ is fixed by the adjoint action if and only if it is a cyclic tensor, and cyclic tensors generate precisely the image of \eqref{eq: R1D} within the codomain $\hat S \Sigma \Ext^1_{\F[G]}(\rho,\rho)^*$. Clearly the image is contained in the subring $\hat S_\mathrm{cyc} \Sigma \Ext^1_{\F[G]}(\rho,\rho)^*$. 

Similarly, we see that a generating set for $I^Z \subset I$ is formed as follows. Choose $\F$-basis $\{b_{i,j,k}\}$ for the generating vector space $m^*\Sigma \Ext^2_{\F[G]}(\rho,\rho)^*$ of $I$, such that its subset with fixed $(i,j)$ is a basis for $m^*\Sigma \Ext^2_{\F[G]}(\rho_j,\rho_i)^*$. Thus each $b_{i,j,k}$ is isotypic for the $Z$-action with the action depending only on $(i,j)$; call this character $\chi_{i,j}$. Then we observe that $I_Z$ is generated by $b_{i,j,k} \otimes c_{i,j,l}$, where for fixed $(i,j)$ the $c_{i,j,l}$ are a basis for a generating vector space of the $\chi_{i,j}^{-1} = \chi_{j,i}$-isotypic part of $\hat S_\mathrm{cyc} \Sigma \Ext^1_{\F[G]}(\rho,\rho)$ as an $R^1_D$-module. The minimal such vector space is
\[
\bigoplus_{\gamma \in SCC(i,j)} \Sigma \Ext^1_G(\gamma)^*. 
\]
Then observe that this generating set $\{b_{i,j,k} \otimes c_{i,j,l}\}$ is a basis for the vector space in the denominator of \eqref{eq: R_D pres body}. 
\end{proof}

Having presented $R_D$, we can prove the rest of the results of \S\ref{subsec: results2}, which are corollaries of the presentation. From now on, we refer to \S\ref{subsec: results2} for the statement of these corollaries. 

\begin{proof}[{Proof of Corollary \ref{cor: tangent pseudo} (Presentation of the tangent space $\frt_D$ of $R_D$)}]
Firstly we observe directly from the definition of $R^1_D$ (Definition \ref{defn: R1D}) that when $\Ext^2_{\F[G]}(\rho,\rho) = 0$, then the presentation datum induces a presentation of its tangent space $\frt^1_D$ as 
\[
\frt^1_D \cong \bigoplus_{\gamma \in SC(\bold r)} \Sigma \Ext^1_{\F[G]}(\gamma). 
\]
Using the presentation of $R_D$ in Theorem \ref{thm: R_D pres body} and unraveling the definition of $m^*$ in terms of the $A_\infty$-operations $m_n$ from the bar equivalence (\S\ref{subsec: bar}) for $n \leq r$, we produce the maps out of $\frt^1_D$ that appear in the statement of the corollary. Their common kernel is $\frt_D \subset \frt^1_D$. 
\end{proof}

\subsection{Canonical structure of the tangent space}
\label{subsec: tangent}

As emphasized in Warning \ref{warn: NC pseudo}, the direct sum expression of $\frt_D$ in Corollary \ref{cor: tangent pseudo} is highly non-canonical, being dependent on the presentation datum. In contrast, the \emph{complexity filtration} of \cite[\S3]{bellaiche2012} is a canonical filtration on $\frt_D$ whose graded pieces are summands appearing in Corollary \ref{cor: tangent pseudo}. 

\begin{defn}[Bella\"iche]
\label{defn: complexity}
The \emph{complexity} of a pseudodeformation $D_A$ of $D$ is the minimal integer $c(D_A)$ such that for all $\gamma \in SC(\bold r)$ with length strictly greater than $c(D_A)$, the image of $\Sigma\Ext^1_{\F[G]}(\gamma)^*$ in $R_D$ under the presentation \eqref{eq: R_D pres body} are sent to zero under the induced map
\[
R_D \ra A. 
\]
Equivalently, the map of tangent spaces $(\m_A/\m_A^2)^* \ra \frt_D$ has image contained in 
\[
\Fil_{c(D_A)} \frt_D := \bigoplus_{\substack{\gamma \in SC(\bold r) \\ l(\gamma) \leq c(D_A)}} \Sigma\Ext^1_{\F[G]}(\gamma). 
\]

The \emph{complexity filtration} on $\frt_D$ is the increasing exhaustive filtration consisting of the sums above. We denote it by $\Fil_k \frt_D \subset \frt_D$; observe that $\Fil_0 \frt_D = 0$ and $\Fil_r \frt_D = \frt_D$. 
\end{defn}

\begin{lem}
\label{lem: CF is can}
The complexity filtration of $\frt_D$ is canonical. 
\end{lem}
\begin{proof}
To see that the complexity filtration is canonical, first observe that the ideal of cycles is a canonical ideal $I_\mathrm{cyc} \subset S^\mathrm{GMA}_{E_D}$. Indeed, it is $I_\mathrm{cyc} = \m_D \cdot S^\mathrm{GMA}_{E_D}$, where $\m_D \subset R_D \rinj S^\mathrm{GMA}_{E_D}$ as the invariant subring of the $Z$-action (see the proof of Theorem \ref{thm: R_D pres body}). Then, we have a decreasing filtration of $S^\mathrm{GMA}_{E_D}$ given by $I_\mathrm{cyc}^n$, $n \geq 0$. The decreasing filtration on the cotangent space $\m_D/\m_D^2$ of $R_D$ given by the intersection of $\m_D$ with $I_\mathrm{cyc}^n$ is perfectly dual to the complexity filtration of $\frt_D$. 
\end{proof}

As we pointed out in Remark \ref{rem: refine Bel12}, the use of $A_\infty$-products or Massey higher products refines the result \cite[Thm.\ 1]{bellaiche2012}, which only used cup products. It results in a canonical determination of $\gr_k \frt_D := \Fil_k \frt_D / \Fil_{k-1} \frt_D$. 

\begin{proof}[{Proof of Corollaries \ref{cor: CFP} (determination of $\gr_k \frt_D$) and \ref{cor: pseudo tangent dim} (tangent dimension)}]
Notice that the sought-after presentation of $\gr_k \frt_D$ in the statement of Corollary \ref{cor: CFP} appears as the sub-summand of the direct sum expression for $\frt_D$ that appears in Corollary \ref{cor: tangent pseudo}. This sub-summand is cut out by restricting the indices $\frt_D = \bigoplus_{\gamma \in SC(\bold r)} (-)$ to those $\gamma$ with length $k$. Therefore, after applying Corollary \ref{cor: tangent pseudo} and Lemma \ref{lem: CF is can}, we get a map 
\begin{equation}
\label{eq: gr_k embedding}
\gr_k \frt_D \rinj \bigoplus_{\substack{\gamma \in SC(\bold r) \\ l_\gamma = k}} 
\Sigma\Ext^1_G(\gamma).
\end{equation}
It remains to show that the image, as claimed in Corollary \ref{cor: CFP}, is canonical (i.e.\ not dependent upon the choice of presentation datum). 

The content of \cite[\S3.3]{bellaiche2012} is the proof that there is a canonical injection exactly as in \eqref{eq: gr_k embedding}. It remains to check that this injection is compatible with \eqref{eq: gr_k embedding} (which arises from Corollary \ref{cor: tangent pseudo}). Both of these injections ultimately arise from the canonical $Z$-equivariant isomorphism $\ker(\rho)/\ker(\rho)^2 \cong \Sigma\Ext^1_{\F[G]}(\rho,\rho)^*$ (where $\ker\rho \subset \F[G]$) by taking symmetric tensor powers and $Z$-invariants, we have the desired compatibility. 

Counting the dimensions of the vector spaces appearing in the expression for $\gr_k \frt_D$, we deduce Corollary \ref{cor: pseudo tangent dim}. 
\end{proof}

\subsection{Input from invariant theory and quiver representation theory} 
\label{subsec: invariants quivers}

Like the tangent dimension, the bounds on the Krull dimension of $R_D$ claimed in Corollary \ref{cor: pseudo krull bounds} follow mostly from counting dimensions in the presentation of $R_D$ of Theorem \ref{thm: R_D pres body}. The additional ingredient is our extra knowledge about the ring $R^1_D$ from invariant theory. These have been stated in Fact \ref{fact: R1D}, which is a brief and partial summary of extensive literature about invariant subrings of regular rings under (linearly) reductive group actions. The point is that there are finite combinatorial objects controlling $R^1_D$. In particular, the relations cutting out these rings are polynomial, in contrast to the power series arising from $A_\infty$-products. 

\begin{proof}[{References for Fact \ref{fact: R1D}}]
It is clear that $R^1_D$ is reduced, as it is a subring of a domain. The fact that a subring of invariants of a regular commutative algebra under the action of a linearly reductive algebraic group is normal and Cohen-Macaulay is due to Hochster \cite{hochster1972}. 

The decomposition into tensor factors, each arising from a strongly connected component, follows from the generation of $R^1_D$ by cyclic tensors. 

The claim about the Krull dimension is \cite[Thm.\ 6]{LBP1990}. There, the authors use quivers and their representations; their paper begins with an introduction to these notions. The representation-unobstructed setting we are in can be translated to theirs by observing that the representation theory of $\hat T_{\F^r} \Sigma\Ext^1_{\F[G]}(\rho,\rho)^*$ is the same as the representation theory of the quiver with $r$ vertices labeled by $\{1, \dotsc, r\}$, and with $h^1_{i,j}$ directed arrows from $i$ to $j$. The dimension vector $\alpha$ (in the notation of \emph{loc.\ cit.}) is $\alpha = (1, \dotsc, 1) \in \bN^{\oplus r}$, because each $\rho_i$ appears with multiplicity 1 in $\rho$. The ``Ringel bilinear form'' $R$ on $\Z^{\oplus r} \times \Z^{\oplus r}$ is represented with the matrix
\[
R = (R_{i,j}) = \dim_\F \Hom_{\F[G]}(\rho_i, \rho_j) - \dim_\F \Ext^1_{\F[G]}(\rho_i, \rho_j) = \delta_{i,j} - h^1_{i,j}. 
\]
The we see in \emph{loc.\ cit.}\ that the Krull dimension of $R^1_D$ is $R(\alpha, \alpha)$, which is equal to $1-r + \sum_{1 \leq i,j \leq r} h^1_{i,j}$, as claimed. The final claim follows from Krull dimension being additive under tensor products of commutative $\F$-algebras. 
\end{proof}

\begin{rem}
It has been determined when $R_D$ is regular \cite[\S1.1, Thm.~4]{KKMS1973}, complete intersection \cite{nakajima1986} 
and Gorenstein \cite{stanley1978}. 
\end{rem}

\begin{rem}
For a broader perspective on quiver representations in relation to this discussion, see \cite[\S\S5.7-5.8]{lebruyn2008}, which relies on work of Bocklandt. The statement on Krull dimension is reproduced in [\textit{loc.\ cit.}, Lem.\ 5.13]. 
\end{rem} 

\begin{eg}
\label{eg: R1D LCI}
There is a single simple cycle $\gamma = (12)$, and the only relation that must be imposed is the extra commutativity relation 
\begin{equation}
\label{eq:r=2_obs}
(x_{12} \otimes x_{21})\cdot (y_{12} \otimes y_{21}) = (x_{12} \otimes y_{21}) \cdot (y_{12} \otimes x_{21}),
\end{equation}
which results in an quadratic obstruction of dimension $\binom{d_{12}}{2} \binom{d_{21}}{2}$. In particular, $R_D$ is regular when either of $d_{12}$ or $d_{21}$ is $\leq 1$. 

Consider the case $d_{12} = d_{21} = 2$. We can then take $\{x_{ij}, y_{ij}\} \subset \Ext^1_G(\rho_j, \rho_i)^*$ to be a $k$-basis, so that the space of relations is $1$-dimensional, generated by \eqref{eq:r=2_obs} as written. If we write 
\[
W = x_{12} \otimes x_{21}, \quad X = x_{12} \otimes y_{21}, \quad Y = y_{12} \otimes x_{21}, \quad Z = y_{12} \otimes y_{21},
\]
then we readily see that
\[
R_D \cong \frac{k\lb W,X,Y,Z\rb}{(WZ-XY)}.
\]
\end{eg}

\begin{eg}
The singularity of Example \ref{eg: R1D LCI} is the only possible singularity of $R^1_D$ when its Krull dimension is 3. For a classification of singularities in $R^1_D$ when its Krull dimension is $\leq 6$, see \cite{BLV2003}. 
\end{eg}

\subsection{Obstruction theory}
\label{subsec: obs theory}

Here we prove the remaining statements of \S\ref{subsec: results2}. All that is left to add is the following basic formulation of obstruction theory. While the obstructions $\beta$ are representation-theoretic, the obstructions $\alpha$ are combinatorial and can be calculated using the content of \S\ref{subsec: invariants quivers}. 

\begin{proof}[{Proof of Fact \ref{fact: basic obstruction}}]
It is a standard fact that whenever $(S, \m_S)$ is a regular local $\F$-algebra surjecting onto $S'$ with kernel $I$, then the obstruction to lifting a homomorphism $\ell_n : S' \ra \F[\varep]/(\varep)^n$ to a homomorphism $S' \ra \F[\varep]/(\varep)^{n+1}$ is an element of $I/\m I$ that can be produced from $\ell_n$. See, for example, \cite[\S1.5, Prop.\ 2, pg.\ 399]{mazur1989}. 

Then Fact \ref{fact: basic obstruction} follows from observing that $h^2(\cC(D))$, as defined in Notation \ref{note: Ext}, is a basis for $I/\m I$ in this case. Remark \ref{rem: bounded deg obs} follows from observing that $J$ of Notation \ref{note: Ext} is finitely generated as a monoid. 
\end{proof}

\begin{proof}[{Proof of Corollary \ref{cor: obstruction}}]
In light of the presentation for $R_D$ of Theorem \ref{thm: R_D pres body}, part (1) follows from Fact \ref{fact: basic obstruction}. Given that the obstruction $\alpha(D_n)$ of part (1) vanishes, Part (2) follows from the principle of Fact \ref{fact: basic obstruction}; the following computation realizes this principle. 

Choose $i,j \in \bold r$ and  element $\omega \in \Sigma\Ext^2_{\F[G]}(\rho_j, \rho_i)^*$. We have $m^*(\omega)$ in the $(i,j)$-part of $\hat T_{\F^r} \Ext^1_{\F[G]}(\rho,\rho)^*$. For each $\gamma \in SCC(i,j)$ and $\kappa \in \Sigma\Ext^1_{\F[G]}(\gamma)^*$, $m^*(\omega) \otimes \kappa$ is an element of the $(k,k)$-part of $\hat T_{\F^r} \Sigma\Ext^1_{\F[G]}(\rho,\rho)^*$. Therefore, taken as an element of $\hat S_\mathrm{cyc} \Sigma\Ext^1_{\F[G]}(\rho,\rho)^*$, $m^*(\omega) \otimes \kappa$ is in the subring $R^1_D$. It was killed by $\psi_n^1$ (because $\psi_n$ exists), therefore $\psi_{n+1}^1(m^*(\omega) \otimes \kappa) \in \epsilon^{n+1}\F \cong \F$. By duality, we see that $\psi_{n+1}^1$ gives rise to an element of the $\F$-linear dual of the denominator of the presentation  \eqref{eq: R_D pres body}, as desired. 
\end{proof}

\section{Galois representations satisfying arithmetic conditions}
\label{sec: conds}

The goal of this section is to prove Theorem \ref{thm: conds}. This theorem claims that there exists a dg-algebra to which one can apply the theory of \S\ref{sec: profinite presentations} to compute these deformation spaces with an extra condition $\cC$. Indeed, so far we have considered only the ``unrestricted case,'' presenting deformation spaces for representations of $G$.  

To do this, we produce an algebra quotient $E^\cC$ of $\F\lb G\rb$ that factors exactly those actions with condition $\cC$. Then we use the Hochschild cochain complex of the endomorphism module of a representation of this algebra to compute the deformations, applying the theory of \S\ref{sec: profinite presentations}.

\subsection{Stable conditions and Cayley--Hamilton conditions}
\label{subsec: stable or CH}

We begin by recalling previous results that construct the moduli spaces of representations of $G$ with condition $\cC$. We begin with a description of the conditions $\cC$ that we will consider. Instead of working with particular conditions, we set up two different sorts of conditions to which we can apply our theorems. 

In this particular section, we work with mixed characteristic coefficients, using $\Z_p$ for simplicity. 

\begin{defn}[Stable condition]
\label{defn: stable}
A \emph{stable condition} $\cC$ is a subcategory of finite length $\Z_p[G]$-modules that is closed under subquotients and finite direct sums. 
\end{defn}
Stable conditions first appeared in the study of Galois representations by Ramakrishna \cite{ramakrishna1993}, in the context of deformations rings of representations. When $\rho$ has scalar endomorphisms, he produced a quotient $R_\rho \rsurj R^\cC_\rho$ parameterizing exactly those deformations satisfying condition $\cC$. 

In the author's joint work with Preston Wake \cite{WWE4}, stable conditions were shown to naturally extend to Cayley--Hamilton representations, thereby allowing for 
\begin{itemize}[leftmargin=2em]
\item a sensible definition of pseudorepresentations of a profinite group $G$ with condition $\cC$,
\item a universal pseudodeformation ring $R^\cC_D$ parameterizing deformations of $D$ satisfying $\cC$, which cuts out a closed condition $R_D \rsurj R^\cC_D$ in the whole deformation space, 
\item the construction of a Zariski-closed subspaces $\Rep^{\square, \cC}_\rho \subset \Rep^\square_\rho$, $\Rep^\cC_\rho \subset \Rep_\rho$ of representations with residual semi-simplification $\rho$ and condition $\cC$. 
\end{itemize}

Here is the second kind of condition that we will study. In order to formulate this notion, we use the category of Cayley--Hamilton representations over a residual semi-simplification $\rho$; see \cite[Defn.\ 2.1.5]{WWE4}. 

\begin{defn}
\label{defn: CH cond}
Let $\rho$ be a residual semi-simplification of $G$. We say that $\cC$ is a \emph{Cayley--Hamilton condition (over $\rho$)} when it applies to any Cayley--Hamilton representation over $\rho$ and is ``representable,'' in the sense that there exists a universal object in the category of Cayley--Hamilton representations over $\rho$. We will denote such a universal object by $E^\cC_D$. 
\end{defn}

In particular, a Cayley--Hamilton condition $\cC$ over $\rho$ is not assumed to apply to arbitrary finite length $\F[G]$-modules, nor even to those with composition factors among the factors of $\rho$. A structure of a Cayley--Hamilton representation is required.

\begin{eg}[Ordinary representations]
\label{eg: ordinary}
A 2-dimensional representation of $G_\Q$ is called \emph{ordinary} when its restriction to a decomposition group at $p$ admits a 1-dimensional unramified quotient. In \cite[\S5]{WWE1}, as well as \cite[\S7]{WE2018}, this condition was extended to a Cayley--Hamilton condition. In each of these works, a ``residually $p$-distinguished'' condition was required in order to use the structure of a GMA on all Cayley--Hamilton representations (in contrast, see \cite{CS2019}). The condition was then that the GMA representation must be upper-triangular after restriction to a decomposition group at $p$, with the quotient character being unramified. Following \cite{CS2019}, a GMA-structure-free notion of the ordinary Cayley--Hamilton condition was produced in \cite[\S3.7]{WWE5}. 
\end{eg}

\subsection{Construction of a universal associative algebra with condition $\cC$}
\label{subsec: stable constr}

In this section, provided that $\cC$ is a stable condition, we generalize the construction of \cite[\S2.4]{WWE4} in order to produce an associative algebra $E^\cC$. This algebra $E^\cC$ play the role of $\F\lb G\rb$ after imposing condition $\cC$. Once we have $E^\cC$, then for any representation $\rho$ of $G$ satisfying $\cC$, the completion $E^\cC_\rho := (E^\cC)^\wedge_{\ker \rho}$ is a desirable analogue of $\F\lb G\rb_{\ker \rho}$. 

\begin{rem}
In the case that $\cC$ is a Cayley--Hamilton condition over $\rho$, we only have the analogue $E^\cC_D$ of $\F[G]^\wedge_{\ker\rho}$. This is a limitation of the Cayley--Hamilton case relative to the stable case. 
\end{rem}

Let $\cC$ be a stable condition. Notice that $\Z_p\lb G\rb$ is a profinite algebra, which is a topological limit of its finite quotient algebras
\[
E(a,b) := \Z_p/p^a\Z_p[G_b],
\]
where $G \cong \varprojlim_b G_b$ is a profinite presentation for $G$. Therefore, condition $\cC$ may be sensibly applied to $E(a,b)$, taking it as a left $\Z_p\lb G\rb$-module. By \cite[Lem.\ 2.3.5]{WWE4}, there is a maximal quotient module $E(a,b) \rsurj E(a,b)^\cC$ satisfying $\cC$. Therefore, for any $a' \geq a, b' \geq b$, the $\Z_p\lb G\rb$-module quotient $E(a', b') \rsurj E(a,b)^\cC$ factors through $E(a',b')^\cC$. Therefore we can produce a new limit 
\[
E^\cC := \varprojlim_{a,b} E(a,b)^\cC.
\]
Now, as in \cite[Lem.\ 2.4.3(2)]{WWE4}, considering the the right action of $\Z_p\lb G\rb$ on $E(a,b) \rsurj E(a,b)^\cC$ allows one to find that $E(a,b)^\cC$ has the structure of an algebra quotient of $E(a,b)$. Naturally, in the limit this makes $E^\cC$ an algebra quotient of $\Z_p\lb G\rb$. 

When $\rho$ has property $\cC$, then $\rho : \F\lb G\rb \ra \End_\F(V)$ factors through $E^\cC$, and we let 
\[
E^\cC_\rho := \varprojlim_i E^\cC/\ker(\rho)^i. 
\]

\begin{rem}
This generalizes \cite[\S2.4]{WWE4} only in that the construction discussed there is applied to a general profinite algebra instead of only Cayley--Hamilton algebras. 
\end{rem}

\subsection{Proof of Theorem \ref{thm: conds}}

Let $\cC$ be a stable condition and let $E^\cC$ be as constructed above. We choose a representation $\rho$ of $G$ with condition $\cC$. Notice that we now have a natural inclusion of Hochschild cochain complexes
\[
C^\bullet(E^\cC, \End_\F(\rho)) \subset C^\bullet(\F[G], \End_\F(\rho))
\]
that is an inclusion of dg-$\F$-algebras. Choose compatible presentation data (as in Definition \ref{defn: pres datum}) for $E$ and $\F[G]$, meaning that the homotopy retract data $(i,p,h)$ on $C^\bullet(\F[G], \End_\F(\rho))$ restricts to that on $C^\bullet(E^\cC, \End_\F(\rho))$; and that the choices of idempotents of $E^\wedge_\rho$ and $\F[G]^\wedge_\rho$ are compatible under $\F[G] \ra E^\cC$. 

Write $C^\bullet_\cC$ for $C^\bullet(E^\cC, \End_\F(\rho))$, and write $H^\bullet_\cC$ for its cohomology. We now have a more specific statement of Theorem \ref{thm: conds}. 
\begin{thm}
\label{thm: conds body}
Let $G$ be a profinite group satisfying finiteness condition $\Phi_p$, let $\rho$ be a multiplicity-free residual semi-simplification, and let $\cC$ be a stable condition. The compatible presentation data above induce 
\begin{itemize}[leftmargin=2em]
\item a presentation for $E^\cC_\rho$ as a $\F^r$-algebra, exactly as in Theorem \ref{thm: profinite r-pointed dual},
\item presentations for $\Rep^{\cC,\mathrm{GMA}}_\rho$ and $\Rep^\cC_\rho$ as formal moduli spaces over $\Spf \F$, exactly as in Theorem \ref{thm: adapted pres}, and 
\item a presentation for $R^\cC_D$ as an object of $\hat\cA_\F$, exactly as in Theorem \ref{thm: R_D pres body}
\end{itemize}
equipped with morphisms to (resp.\ from) the analogous unrestricted objects $\F[G]^\wedge_{\ker\rho}$, $\Rep^{\square}_\rho$, $\Rep_\rho$, and $R_D$. All are closed immersions (resp.\ quotients of rings). 
\end{thm}

\begin{proof}
Using the argument for \cite[Lem.\ 2.4.3(3)]{WWE4}, we know that a $\F\lb G\rb$-module has property $\cC$ if and only if the action factors through $E^\cC$. With this understood, we may repeat the proof of Theorem \ref{thm: adapted pres}. 
\end{proof}

Naturally, the formulas for the tangent space, representation-unobstructed case, obstruction theory, and Krull dimension also follow from the presentations of Theorem \ref{thm: conds body}, just as in \S\ref{subsec: tangent} and \S\ref{subsec: obs theory}. 

\begin{rem}
\label{rem: Ext2 discrepancy}
We discuss the contrast between the two cases 
\begin{itemize}
\item $\cC$ is a stable condition 
\item $\cC$ is a Cayley--Hamilton condition over $\rho$ and $D = \psi(\rho)$. 
\end{itemize}
In order to compare them, it is instructive to consider the Cayley--Hamilton condition $\cC^\mathrm{CH}$ over $\rho$ arising from a stable condition, and the difference in the Hochschild complex for $E^\cC$ vs.\ $E^{\cC^\mathrm{CH}}_D$. (Because the maximal Cayley--Hamilton quotient of $E^\cC$ over $\rho$ is $E^\cC_D$, we write $E^\cC_D$ for $E^{\cC^\mathrm{CH}}_D$.) It follows from the theorem that they will each compute exactly the same deformation space. It is also the case that $\Ext^k_{\cC}(\rho_j, \rho_i) \cong \Ext^k_{E^\cC_D}(\rho_j, \rho_i)$ for $k = 0,1$. But there is a difference when $k=2$. The classes of $\Ext^2_{E^\cC_D}(\rho_j, \rho_i)$ are only those which appear as Massey products of representations with filtrations with graded pieces among the $(\rho_i)_{i=1}^r$. But $E^\cC$, having been constructed without any completion at $\ker\rho$, has the full $\Ext^2$-groups of the subcategory $\cC$ of $\F[G]$-modules. 

We regard the $\Ext$-groups arising from $E^\cC$ as being meaningful even in a derived sense, while those that arise from $E^{\cC^\mathrm{CH}}$ as being useful only for computations of classical deformations. 
\end{rem}

\section{Ranks of $p$-adic Hecke algebras}
\label{sec: ranks}

In this section, we give examples of $p$-adic completions (or interpolations) $\bT$ of classical Hecke algebras acting on modular forms. These are known to be free of finite rank over a regular local complete Noetherian $\Z_p$-algebra. Here, we give an expression for this rank in terms of $A_\infty$-products. This is a measure of the size of the module of congruent modular forms. It would be interesting to relate this quantity to some sort of analytic invariant, i.e.\ an appropriate modulo $p$ version of an adjoint $L$-function. 

The approach is to use a known isomorphism $R \risom \bT$ where $R$ is some deformation ring of Galois representations, and then use $A_\infty$-products to calculate the rank of $R$. 

We discuss four cases.
\begin{itemize}[leftmargin=2em]
\item The finite-flat residually non-Eisenstein setting of Wiles \cite{wiles1995},
\item the ordinary residually non-Eisenstein setting of Wiles \textit{op.\ cit}., 
\item the ordinary Eisenstein setting of Ribet's converse to Herbrand's theorem \cite{ribet1976}, and  
\item the finite-flat residually Eisenstein setting of Mazur's Eisenstein ideal \cite{mazur1978}, following \cite{WWE3}. 
\end{itemize}
``Finite-flat'' and ``ordinary'' are conditions on Galois representations. The residually Eisenstein/non-Eisenstein distinction is more serious, because on the Galois side this is the residually irreducible/reducible distinction. 

Let $G = G_{\Q,S}$ be the Galois group of $\Q$ with ramification only at a finite set of places $S$, where $S$ is the support of $Np\infty$ and where $N \in \Z_{\geq 1}$ satisfies $p \nmid N$. When $\ell$ is a prime number, let $G_\ell = \Gal(\oQ_\ell/\Q_\ell) \to G$ be the homomorphism arising from a choice of decomposition subgroup for $\ell$ in $G_{\Q,S}$. 

\subsection{Weight 2 non-Eisenstein non-ordinary Hecke algebras}
\label{subsec: flat Wiles}

In this section and the next, we discuss the Hecke algebras proved to be isomorphic to a Galois deformation ring by Wiles \cite{wiles1995}. Here, we focus on the finite-flat case, which is case (ii) on [pg.\ 456, \textit{loc.\ cit.}]. 

Write $\rho$ here for the residual representation
\[
\rho : G_{\Q,S} \lra \GL_2(\F).
\]
written $\rho_0$ in \textit{loc.\ cit.}; in particular, it is odd and absolutely irreducible, and $\rho\vert_{G_p}$ is finite-flat. We assume that $\rho$ is ramified at all primes $\ell \mid N$, and that $\rho\vert_{I_\ell}$ satisfies cases (B) or (C) of [pg.\ 458, \textit{loc.\ cit.}]. (We do not permit case (A).) This makes for a deformation condition denoted $\cD$ there, which includes the finite-flat condition on restriction to $G_p$. We let $G$ be the maximal quotient of $G_{\Q,S}$ in which $\ker(\rho\vert_{I_\ell}) \subset I_\ell$ vanishes for all primes $\ell \mid N$. Then, let $E_\cD$ be the maximal quotient of $\F\lb G\rb$ that is finite-flat upon restriction to $G_p$, in the sense of \S\ref{subsec: stable constr}. 

\begin{prop}
The Hochschild cochain complex $C_\cD := \cC^\bullet(E_\cD, \End_\F(\rho))$ calculates the deformation ring $R_\cD/pR_\cD$ of \cite[pg.\ 458]{wiles1995} via Theorems \ref{thm: main A-inf} and \ref{thm: conds body}. In particular, the $A_\infty$-products in $H^\bullet(C_\cD)$ determine the rank of $R_\cD \cong \bT$. 
\end{prop}

\begin{proof}
First we establish that a deformation $\tilde \rho$ of $\rho$ as a $G_{\Q,S}$-representation satisfies Wiles's deformation problem $\cD$ if and only if $\tilde \rho$ factors through $\F\lb G_{\Q,S}\rb \rsurj \F\lb G\rb$. Firstly, we note that the deformation condition at $\ell \mid N$ is exactly that $\tilde \rho(I_\ell) \cong \rho(I_\ell)$ in cases $B$ and $C$. Thus these deformation conditions amount to factoring through $G_{\Q,S} \rsurj G$. The remaining condition is the finite-flat condition on $\tilde \rho\vert_{G_p}$, which is satisfied exactly when $\tilde \rho$ factors through $\F\lb G\rb \rsurj E_\cD$. 

Now we may apply Theorems \ref{thm: main A-inf} and \ref{thm: conds body}. Because $\rho$ is irreducible, one has $r=1$ and $S^{\cD,\mathrm{GMA}}_\rho \cong R_\cD$ represents $\Rep_\rho$. Because $R_\cD$ is free of finite rank over $W(\F)$ by the $R_\cD \risom \bT$ theorem of \cite[Thm.\ 3.3]{wiles1995}, the $\F$-dimension of $R_\cD/p R_\cD$ is equal to this rank. This $\F$-dimension can be read off from the presentation for $R_\cD/pR_\cD$ furnished by Theorems \ref{thm: main A-inf} and \ref{thm: conds body}. 
\end{proof}

Using \cite[(1.5), pg.\ 460]{wiles1995}, we find that $H^1(C_\cD)$ is canonically isomorphic to the $p$-torsion of the ``adjoint Selmer group'' $H^1_\cD(\Q_S, \End_\F(\rho))$ given there. The $A_\infty$-products or Massey products $m_n : H^1(C_\cD)^{\otimes n} \ra H^2(C_\cD)$ are obstruction classes in the category of finite-flat left $\F[G]$-modules. 

\begin{eg}
\label{eg: 1-dim H1}
For instance, if $\dim_\F H^1(C_\cD) = 1$, we know that the rank is at least 2, i.e.\ there is some non-trivial congruence between modular forms. Then $R_\cD/pR_\cD \simeq \F[\epsilon_n]$ for some $n \geq 1$. The rank of $\bT$ is then $n+1$. Let $a \in H^1(C_\cD)$ be a basis. Then $n$ is the greatest $i \geq 1$ such that $m_i(b^{\otimes i}) = 0$ in $H^2(C_\cD)$. This cohomology class also can be calculated as a Massey power (Definition \ref{defn: Massey powers}). 
\end{eg}

\subsection{Ordinary non-Eisenstein Hecke algebras} 
\label{subsec: ord Wiles}

We now assume that $\rho$ is as in \S\ref{subsec: flat Wiles}, with the exception that we now move to case (i) on [pg.\ 456, \textit{loc.\ cit.}], the ordinary case. We form the maximal quotient $G'$ of $G_{\Q,S}$ such that $\ker(\rho\vert_{I_\ell})$ for all primes $\ell \mid N$, just as in \S\ref{subsec: flat Wiles}. Next, we take the universal Cayley--Hamilton quotient $E$ of $\F[G']$ over $\rho$, and then take its ordinary Cayley--Hamilton quotient $E_{\cD}$ of \cite[\S5]{WWE1} (see Example \ref{eg: ordinary}). This matches the ordinary condition (i-b) on [pg.\ 457, \textit{loc.\ cit.}]. 

Let $E_{\cD'}$ be the further Cayley--Hamilton quotient with some fixed determinant valued in $W(\F)$; denote the corresponding deformation ring by $R_{\cD'}$. This is essentially the ``Selmer'' deformation condition of (i-a) on [pg.\ 456, \textit{loc.\ cit.}]. Note also that our use of $\cD$ and $\cD'$ matches that of Wiles [bottom of pg.\ 458, \textit{loc.\ cit.}]. 

Under a mild condition, we claim that we can express condition $\cD'$ (modulo $p$) as a stable condition in a somewhat limited sense (which we discuss at the end of this section). We thank Shaunak Deo for conversations giving rise to this idea. 

The mild condition is the requirement that $\chi_1 \chi_2^{-1}$ has order at least 3, where  $\chi_1, \chi_2$ are the Jordan-H\"older factors of $\rho\vert_{G_p}$. The stable condition is given as follows. Let $\chi_{i,A} : G_p \ra \F\lb t_i\rb$ be any unramified deformations of $\chi_i$ to $A \in \cC_\F$, and let $e_A$ be any extension of $\chi_{2,A}$ by $\chi_{1,A}$. Let $H \subset G_p$ be the intersection of the kernels of the $G_p$-action on all such extensions. Then we observe that a deformation of $\rho\vert_{G_p}$ is ordinary if and only if $H$ is in its kernel. Thus we may express the ``ordinary with fixed determinant'' condition by considering deformations of $\rho$ that kill $H$: let $G$ be the maximal quotient of $G'$ in which $H$ vanishes. Let $E_{\cD'} := \Z_p[G]^\wedge_\rho$. 

In this setting, $R_{\cD} \cong \bT$ is a ``big'' ordinary Hecke algebra, originally constructed by Hida. This is a finite rank free $\Lambda$-module, where $\Lambda \simeq W(\F)\lb t\rb$ and $\Lambda \ra \bT$ is the weight map induced by the inclusion of the Hecke algebra of diamond operators. Equivalently, $\Lambda \ra \bT$ parameterizes the determinant of the Galois representation valued in $\bT$, so $R_{\cD'} \cong R_{\cD} \otimes_\Lambda W(\F)$ for a map $\Lambda \ra W(\F)$ explained in [pg.\ 459, \textit{loc.\ cit.}]. Thus the $\Lambda$-rank of $R_{\cD} \cong \bT$ is equal to the $\F$-dimension of $R_{\cD'}/pR_{\cD'}$. 

\begin{prop}
The Hochschild cochain complex $C_{\cD} := \cC^\bullet(E_{\cD}, \End_\F(\rho))$ calculates the deformation ring $R_{\cD}/pR_\cD$ via Theorem \ref{thm: conds body}. Similarly, the Hochschild cochain complex $C_{\cD'} := \cC^\bullet(E_{\cD'}, \End_\F(\rho))$ calculates the deformation ring $R_{\cD'}/pR_{\cD'}$. In particular, the $A_\infty$-products in $H^\bullet(C_{\cD'})$ determine the $\Lambda$-rank of $R_\cD \cong \bT$. 
\end{prop}

As in \S\ref{subsec: flat Wiles}, one can check that $H^1(C_\cD)$ is canonically isomorphic to the reduction modulo $p$ of the ordinary $H^1_\cD(\Q_S, \End_\F(\rho))$ of \cite[(1.5), pg.\ 460]{wiles1995}. While $H^2(C_\cD)$ functions correctly in deformation-theoretic computations, the category $\cE$ for which it computes an $\Ext^2_\cE(\rho,\rho)$ is limited, in the sense of Remark \ref{rem: Ext2 discrepancy}. Indeed, $\cE$ is the category of $E_\cD$-actions on finite dimensional $\F$-vector spaces.

\subsection{Ordinary residually Eisenstein Hecke algebras}
\label{subsec: ord Eis 1}

We follow \cite[\S2]{WWE1} (and the references therein) for the setting of ordinary modular forms we work with, and the assumptions of the main theorems. We work in the case $S = \{p,\infty\}$, i.e.\ we work with ordinary modular forms of level 1, for simplicity. It is possible to work with more general level, but we leave this out so that the assumptions of the main theorems reduce to the Kummer--Vandiver conjecture. 

Let $\omega$ denote the modulo $p$ cyclotomic character of $G_{\Q,S}$, and let $\kappa$ denote its $p$-adic cyclotomic character. The relevant residual semi-simplification is $\rho \cong \omega^{k-1} \oplus 1$ over $\F_p$, where $2 \leq a \leq p-3$ is even. Ribet \cite{ribet1976} proved that there is a non-trivial $\omega^{1-k}$-part $A[\omega^{1-k}]$ of the $p$-cotorsion $A$ of the class group of $\Q(\zeta_p)$ if and only if $p$ divides the numerator of the Bernoulli number $B_k$. Ribet's idea was to use a modular eigenform produce a $\F_p$-linear Galois representation of the form
\[
\ttmat{\omega^{k-1}}{0}{*}{1}: G_{\Q,S} \ra \GL_2(\F_p),
\]
which is not semi-simple, but is semi-simple after restriction to $G_p$. Let $\bT$ be the Hida Hecke algebra with residual eigensystem corresponding to $\rho$. It is a finite free $\Lambda$-algebra (in the same fashion as in \S\ref{subsec: ord Wiles}) admitting a homomorphism $\bT \ra \Lambda$ corresponding to the interpolation of ordinary $p$-stabilizations of Eisenstein series of level 1 and weight congruent to $k$ modulo $p-1$. Assume also that $p \mid B_k$, so that $\bT$ has rank at least 2, reflecting that it parameterizes some cusp forms with a congruence with these Eisenstein series. 

Upon the assumption that $A[\omega^{-k}]$ is zero --- which follows from Vandiver's conjecture, as $-k$ is even --- there is proved in \cite[Thm.\ 4.2.8]{WWE2} an isomorphism 
\[
R_D^\ord \risom \bT,
\]
where $R_D^\ord$ is an ordinary pseudodeformation ring for $D := \psi(\rho)$, as in \S\ref{subsec: stable or CH}. Because $\omega^{k-1}$ is not of order 2, the ``ordinary with fixed determinant $\kappa^{k-1}$'' condition can be expressed as a stable condition, just as in \S\ref{subsec: ord Wiles}. We let $\cD$ be this deformation condition, in equal characteristic $p$. Under this assumption, the $\Ext^1$-group for the Cayley--Hamilton condition $\cD$ (writing $\chi = \omega^{k-1}$ for convenience) is 
\[
\Ext^1_{E_\cD}(\rho,\rho) \cong 
\ttmat{\Ext^1_{E_\cD}(\chi,\chi)}
{\Ext^1_{E_\cD}(1,\chi)}
{\Ext^1_{E_\cD}(\chi,1)}
{\Ext^1_{E_\cD}(1,1)},
\]
and one can compute that 
\begin{align*}
\Ext^1_{E_\cD}(\chi,\chi) \cong \Ext^1_{\F[G_{\Q,S}]}(\chi,\chi)^{I_p\text{-triv}} \cong H^1(G_{\Q,S}, \F_p)^{I_p\text{-triv}} \cong 0,&\\
\Ext^1_{E_\cD}(1,\chi) \cong \Ext^1_{\F[G_{\Q,S}]}(1,\chi) \cong H^1(G_{\Q,S}, \chi),& \quad \dim_{\F_p} = t\\
\Ext^1_{E_\cD}(\chi,1) \cong \Ext^1_{\F[G_{\Q,S}]}(\chi,1)^{G_p\text{-triv}} \cong H^1_{(p)}(G_{\Q,S}, \chi^{-1}) \cong A[\chi^{-1}]^*,& \quad \dim_{\F_p} = s\\
\Ext^1_{E_\cD}(1,1) \cong \Ext^1_{\F[G_{\Q,S}]}(1,1)^{I_p\text{-triv}} \cong 0.&
\end{align*}
Here $H^i_{(p)}$ denotes cohomology supported at $p$, which is the same as cohomology with compact support since $S = \{p,\infty\}$; see e.g.\ \cite[\S6.2]{WWE1} and the references there. We have assumed that $p \mid B_k$, so $A[\chi^{-1}] \neq 0$, i.e.\ $s \geq 1$. The assumption $A[\omega^{-k}] = 0$ implies that $t=1$. 

Writing $\{b\}$, $\{c_1, \dotsc, c_s\}$ for a $\F_p$-basis for the $\F_p$-duals of the suspensions of the two non-zero $\F_p$-vector spaces, we find that the auxiliary ring $R^{1,\cD}_D$ has the presentation
\[
R^{1,\cD}_D \cong \F_p\lb bc_1, \dotsc, bc_s\rb.
\]
In particular, this ring is formally smooth and all obstructions to the pseudodeformations parameterized by $R_D^\cD$ are representation-theoretic; i.e.\ $\alpha = 0$ in Corollary \ref{cor: obstruction}. Using an injection from $\Ext^2_{E_\cD}(\rho,\rho)$ into appropriate cohomology groups over $\Q$, we also find that $\Ext^2_{E_\cD}(\rho,\rho)$ is non-trivial only in its $\Ext^2_{E_\cD}(\chi,1)$-term, namely,
\begin{equation}
\label{eg: ord Ext2}
\Ext^2_{E_\cD}(\chi,1) \rinj H^2_{(p)}(G_{\Q,S}, \chi^{-1}) \cong H^1_{(p)}(G_{\Q,S}, \chi^{-1}) \cong A[\chi^{-1}]^*,
\end{equation}
where the injection may not (a priori) be an isomorphism for the reasons given in Remark \ref{rem: Ext2 discrepancy}. Letting $\gamma_1, \dotsc, \gamma_s$ be a basis for $\Sigma H^2_{(p)}(G_{\Q,S}, \chi^{-1})^*$, Theorem \ref{thm: main pseudo} yields a presentation
\[
R_D^\cD \cong \frac{\F_p\lb a, bc_1, \dotsc, bc_s\rb}
{(m^*(\gamma_1)b, \dotsc, m^*(\gamma_s)b)}.
\]
Due to the isomorphism $R_D^\ord \risom \bT$ and the fact that $R_D^\cD \cong R_D^\ord/\m_\Lambda R_D^\ord$, we know that $R_D^\cD$ has Krull dimension 0, and that its $\F_p$-dimension is the $\Lambda$-rank of $\bT$. In particular, we know that the injection of \eqref{eg: ord Ext2} is an isomorphism. 

It is expected that $s=1$; this is a consequence of the assumption $A[\omega^k] = 0$, which follows from the Kummer--Vandiver conjecture but is different than our running assumption that $A[\omega^{-k}] = 0$. In this case, the presentation for $R_D^\cD$ is of the form $\F_p[\epsilon_n]$. Here $n$ may be computed as follows. Let $c = c_1$. 
\begin{prop}
The Hochschild cochain complex $C_\cD := \cC^\bullet(E_\cD, \End_\F(\rho))$ calculates the deformation ring $R_\cD$ via Theorem \ref{thm: conds body}. Assume 
\begin{itemize}[leftmargin=2em]
\item  Vandiver's conjecture, and
\item assume that $p \mid B_k$, so that the $\Lambda$-rank of $\bT$ is at least 2 and $s \geq 1$.
\end{itemize}
Then we may read off the presentation for $R_\cD$ that the $\Lambda$-rank of $\bT$ is equal to $n+1$ and $\bT/\m_\Lambda \bT \cong \F[\epsilon_n]$, where $n \in \bZ_{\geq 2}$ is the greatest such that 
\[
m_{2n-3}(c \otimes b \otimes c \otimes \dotsm \otimes b \otimes c) = 0. 
\]
\end{prop}

\begin{rem}
As a complement to the discussion above, we see that a failure of the Kummer--Vandiver conjecture in the form of the existence of $p$ and $k$ such that 
\begin{itemize}[leftmargin=2em]
\item $t > 1$, when $s \geq 1$ as well; or
\item $s > 1$ 
\end{itemize}
is detectable by the tangent dimension of $R_D^\cD$ being greater than 1. See \cite[\S1.2]{wake2} for a discussion of known relationships between the Kummer--Vandiver conjecture (and its consequences) and the Gorenstein property of Hecke algebras. It is interesting to compare these, because we might guess that $R^\ord_D \risom \bT$ even when the hypothesis $A[\omega^{-k}] = 0$ fails. 
\end{rem}

\subsection{The finite-flat residually Eisenstein Hecke algebra of Mazur}
\label{subsec: Mazur Eis}

In this example of a Hecke algebra, we give an example of a Cayley--Hamilton condition that cannot be expressed as a stable condition, but nonetheless has deformation theory controllable by Massey products with defining systems chosen to match the Cayley--Hamilton condition, according to \cite{WWE3}. 

Let $\bT^0$ be the Hecke algebra arising from the Hecke action on weight 2 level $\Gamma_0(N)$ cusp forms with a congruence with the Eisenstein series modulo $p$, where $N$ is a prime. Mazur proved that $\bT^0 \neq 0$ if and only if $p$ divides the numerator of $(N-1)/12$, and asked for an expression for the rank of $\bT^0$ \cite[\S19]{mazur1978}. 

In \cite[Thm.\ 1.3.1]{WWE3}, an expression is given in terms of Massey products in Galois cohomology for the $\Z_p$-rank of $\bT^0$. This theorem relies on an isomorphism $R^\cC_D \risom \bT$, where $\bT \rsurj \bT^0$ is the cuspidal quotient of the full Eisenstein-congruent Hecke algebra $\bT$, and $\mathrm{rank}_{\Z_p} \bT - \mathrm{rank}_{\Z_p} \bT^0 = 1$. We will now explain the Cayley--Hamilton condition condition $\cC$ that determines $R^\cC_D$. 

In this setting, the residual semi-simplification is the representation $\rho = \omega \oplus 1 : G_{\Q,S} \ra \GL_2(\F_p)$, where $S = \{p,N,\infty\}$. Let $D = \psi(\rho)$ be the induced pseudorepresentation. We want to study the deformations of $\rho$ and $D$ satisfying $\cC$, where $\cC$ is the combination of the following two conditions: 
\begin{itemize}[leftmargin=2em]
\item the finite-flat condition upon restriction to $G_p$, which is a stable condition (in the sense of Definition \ref{defn: stable}), and 
\item the Cayley--Hamilton condition that the restriction to $G_N$ induces a trivial pseudodeformation on $I_N$; see \cite[Defn.\ 10.1.2]{WWE3}. We will call this condition ``pseudo-unramified at $N$.'' 
\end{itemize}

The Massey products of \cite[Thm.\ 1.3.1]{WWE3} are valued in $H^2(G_{\Q,S}, \End_{\F_p}(\rho))$ and are given in terms of defining systems of lifts of $\rho$ (to coefficients in $\F_p[\epsilon_n]$) that are chosen inductively: see \cite[\S10]{WWE3}. This is basically the same process as the inductively constructed lifts and Massey powers in \S\ref{subsec: massey express}. These defining systems are designed so that the lifts satisfy $\cC$. 

In contrast with the previous three examples, it does not seem that the pseudo-unramified at $N$ condition can be imposed as a stable condition on deformations of $\rho$. 

However, we can apply our theory to the part of condition $\cC$ that is stable, that is, the finite-flat part. This gives a formulation of the finite-flat Galois cohomology referred to in \cite[Rem.\ 10.6.3]{WWE3}. In particular, it values the Massey products in a cohomology group $H^2_\mathrm{flat}(G_{\Q,S}, \End_{\F_p}(\rho))$ realized as the Hochschild cohomology we now define. Namely, just as in \S\ref{subsec: flat Wiles}, we let $E_\mathrm{flat}$ be the maximal quotient of $\F_p\lb G_{\Q,S}\rb$ that is finite-flat upon restriction to $G_p$. 
\begin{prop}
\label{prop: finite-flat cohom}
The cohomology groups
\[
H^i_\mathrm{flat}(G_{\Q,S}, \End_{\F_p}(\rho)) := H^i(E_\mathrm{flat}, \End_{\F_p}(\rho)). 
\] 
satisfy the desiderata of finite-flat cohomology of \cite[Rem.\ 10.6.3]{WWE3} and support an $A_\infty$-algebra structure quasi-isomorphic to the dg-algebra $C^\bullet(E_\mathrm{flat}, \End_{\F_p}(\rho))$. 
\end{prop}

Then the Massey product computations of \cite[\S10]{WWE3} can be repeated in the cochain complex $C^\bullet(E_\mathrm{flat}, \End_{\F_p}(\rho))$. When this is done, the finite-flat condition on a deformation will be automatic, when the Massey product vanishes. This contrasts with how the arguments of \emph{loc.\ cit.}\ must arrange for an unrestricted deformation to be ``adjusted'' so that it becomes finite-flat. It remains that the pseudo-unramified-at-$N$ condition must be arranged for by hand. 

\begin{rem}
It was possible, for the purposes of \cite{WWE3}, to make the aforementioned adjustments and forgo constructing this finite-flat cohomology because it was understood that once it was produced, there would be an injection 
\[
H^2_\mathrm{flat}(G_{\Q,S}, \End_{\F_p}(\rho)) \rinj H^2(G_{\Q,S}, \End_{\F_p}(\rho)),
\]
so that the vanishing of Massey products could be calculated in unrestricted Galois cohomology without additional complications. However, this does not always hold for similar deformation problems of interest, so statements like Proposition \ref{prop: finite-flat cohom} are useful for formulation computations in such settings.
\end{rem}

\bibliographystyle{alpha}
\bibliography{CWEbib-2019}

\end{document}